\documentclass[11pt, a4paper]{amsart}
\usepackage{amssymb,latexsym,amsmath, graphicx} 
\usepackage{enumerate,enumitem,color,hyperref} 
\usepackage{mathrsfs} 
\usepackage{calligra}
\usepackage[T1]{fontenc}

\input{xypic}
\xyoption{all}

\title[\bf Eisenstein cohomology and $L$-values]{Eisenstein Cohomology for ${\rm GL}_N$ and the \\
special values of Rankin--Selberg $L$-functions \\
over a totally imaginary number field}

\author{ \bf A. Raghuram} 

\date{\today}      

\address{Dept.\,of Mathematics, Fordham University at Lincoln Center, New York, NY 10023, USA.} 

\email{araghuram@fordham.edu}

\numberwithin{equation}{section}   
\newtheorem{lemma}[equation]{Lemma}%[section]
\newtheorem{sublemma}[equation]{Sublemma}
\newtheorem{thm}[equation]{Theorem}
\newtheorem{prop}[equation]{Proposition}%[section]
\newtheorem{cor}[equation]{Corollary}
\newtheorem{con}[equation]{Conjecture}

\newtheorem{defn}[equation]{Definition}
\newtheorem{rem}[equation]{Remark}

\newtheorem{exam}[equation]{Example}

\usepackage{bm}
\usepackage{upgreek}
\makeatletter
\newcommand{\bfgreek}[1]{\bm{\@nameuse{up#1}}}
\def\bpi{\bfgreek{pi}}

\let\oldtocsection=\tocsection
\let\oldtocsubsection=\tocsubsection
\let\oldtocsubsubsection=\tocsubsubsection
\renewcommand{\tocsection}[2]{\hspace{0em}\oldtocsection{#1}{#2}}
\renewcommand{\tocsubsection}[2]{\hspace{1em}\oldtocsubsection{#1}{#2}}
\renewcommand{\tocsubsubsection}[2]{\hspace{2em}\oldtocsubsubsection{#1}{#2}}

\begin{document}

% \tableofcontents
 
 \begin{abstract}{\tiny
Rationality results are proved for the ratios of 
critical values of Rankin--Selberg $L$-functions of 
${\rm GL}(n) \times {\rm GL}(n')$ over a totally imaginary field $F$, by studying rank-one 
Eisenstein cohomology for the group ${\rm GL}(N)/F,$ where $N = n+n',$ generalizing the methods and results of previous work with 
G\"unter Harder \cite{harder-raghuram-book} where the base field was totally real. In contrast to the totally real situation, the internal structure
of the totally imaginary base field has a delicate effect on the rationality results.}
\end{abstract} 

\maketitle

 \def\R{\mathbb{R}}
\def\C{\mathbb{C}}
\def\Z{\mathbb{Z}}
\def\Q{\mathbb{Q}}
\def\A{\mathbb{A}}
\def\F{\mathbb{F}}
 \def\J{\mathbb{J}}
\newcommand\D{{\mathbb{ D}}}
 \def\L{\mathbb{L}}
\def\bG{\mathbb{G}}
\def\N{\mathbb{N}}
\def\BH{\mathbb{H}}
\newcommand\Qi{\mathbb{Q}(\i)}

\newcommand{\vless}{\rotatebox[origin=c]{-90}{$<$}}
\newcommand{\vgreat}{\rotatebox[origin=c]{90}{$<$}}
\newcommand{\vgreater}{\rotatebox[origin=c]{90}{$\leq$}}

%Some Harish-Chandra modules over \Z makros
\newcommand\Dm{\D_\lambda}
\newcommand\Dmp{\D_{\lambda^\prime}}
 \newcommand\Dim{\D_{\io\lambda}}
 \newcommand\Dimp{\D_{\io\lambda^\prime}}
\newcommand\Dum{\D_{\ul{\lambda}}}
\newcommand\Dium{\D_{\io\ul{\lambda}}} 
\newcommand\DumN{\D_{\ul{\lambda}-N\gamma_P}} 
\newcommand\DiumN{\D_{\io\ul{\lambda}-N\gamma_P}} 
\newcommand\cal{\mathcal}
\newcommand\SMK{{\cal S}^M_{K^M_f}}
\newcommand\tMZl{\tM_{\lambda,\Z}} 
\newcommand\Gm{{\mathbb G}_m}
\newcommand\cA{\cal A}
\newcommand\cC{\cal C}
\newcommand\calL{\cal L}
\newcommand\cO{\cal O}
\newcommand\cU{\cal U}
\newcommand\cK{\cal K}   
\newcommand\cW{\cal W}     
\newcommand\HH{{\cal H}}
\newcommand\cF{\mathcal{F}} 
\newcommand\G{\mathcal{G}}
\newcommand\cB{\mathcal{B}}
\newcommand\cT{\mathcal{T}}
\newcommand\cS{\mathcal{S}}
\newcommand\cP{\mathcal{P}}

\newcommand\GL{{ \rm  GL}}
\newcommand\Gl{{ \rm  GL}}
\newcommand\U{{ \rm  U}}
\def\SU{{\rm SU}}
\def\S{{\bf S}}
\newcommand\Gsp{{\rm Gsp}}
\newcommand\Lie {{ \rm Lie}} 
\newcommand\Sl{{ \rm  SL}}
\newcommand\SL{{ \rm  SL}}
\newcommand\SO {{ \rm  SO}}
\newcommand{\Sp}{\text{Sp}}
\newcommand\Ad{\text{Ad}}
\newcommand\AI{\text{AI}}
\newcommand\Sym{\text{Sym}}

\def\ringO{\mathcal{O}}
\def\idealP{\mathfrak{P}} 
 \def\g{\mathfrak{g}}
\def\k{\mathfrak{k}}
\def\z{\mathfrak{z}}
\def\s{\mathfrak{s}}
\def\c{\mathfrak{c}}
\def\b{\mathfrak{b}}
\def\t{\mathfrak{t}}
\def\q{\mathfrak{q}}
\def\l{\mathfrak{l}}
\def\gl{\mathfrak{gl}}
\def\sl{\mathfrak{sl}}
\def\u{\mathfrak{u}} 
\def\su{\mathfrak{su}}
\def\fp{\mathfrak{p}} 
\def\p{\mathfrak{p}}   
\def\r{\mathfrak{r}}
\def\fd{\mathfrak{d}}
\def\fR{\mathfrak{R}}
\def\fI{\mathfrak{I}}
\def\fJ{\mathfrak{J}}
\def\i{\mathfrak{i}}
\def\perm{\mathfrak{S}}
\newcommand\fg{\mathfrak g}
\newcommand\fk{\mathfrak k}
\newcommand\fgK{(\mathfrak{g},K_\infty^0)}
\newcommand\gK{ \mathfrak{g},K_\infty^0 }

\newcommand\ul{\underline} 
\newcommand\tp{{  {\pi}_f}}
\newcommand\tv{ {\pi}_v}
\newcommand\ts{{ {\sigma}_f}}
\newcommand\pts{ {\sigma}^\prime_f}
\newcommand\usf{\ul{\sigma}_f}
\newcommand\pusf{\ul{\sigma}^\prime_f}
\newcommand\usvp{\ul{\sigma}^\prime_v}
\newcommand\usv{\ul{\sigma}_v}
\newcommand\io{{}^\iota}
\newcommand\uls{{\underline{\sigma}}}

\newcommand\Spec{\hbox{\rm Spec}} 
\newcommand\SGK{\mathcal{S}^G_{K_f}}
\newcommand\SMP{\mathcal{S}^{M_P}}
 \newcommand\SGn{\mathcal{S}^{G_n}}
 \newcommand\SGp{\mathcal{S}^{G_{n^\prime}}}
\newcommand\SMPK{\mathcal{S}^{M_P}_{K_f^{M_P}}}
\newcommand\SMQ{\mathcal{S}^{M_Q}}
\newcommand\SMp{\mathcal{S}^{M }_{K_f^M}}
\newcommand\SMq{\mathcal{S}^{M^\prime}_{K_f^{M^\prime}}}
\newcommand\uSMP{\ul{\mathcal{S}}^{M_P}}
\newcommand\SGnK{\mathcal{S}^{G_n}_{K_f}}
\newcommand\SG{\mathcal{S}^G}
\newcommand\SGKp{\mathcal{S}^G_{K^\prime_f}}
\newcommand\piKK{{ \pi_{K_f^\prime,K_f}}}
\newcommand\piKKpkt{\pi^{\pkt}_{K_f^\prime,K_f}}
\newcommand\BSC{ \bar{\mathcal{S}}^G_{K_f}}
\newcommand\PBSC{\partial\SGK}
\newcommand\pBSC{\partial\SG}
\newcommand\PPBSC{\partial_P\SGK}
\newcommand\PQBSC{\partial_Q\SGK}
\newcommand\ppBSC{\partial_P\mathcal{S}^G}
\newcommand\pqBSC{\partial_Q\mathcal{S}^G}
\newcommand\prBSC{\partial_R\mathcal{S}^G}
\newcommand \bs{\backslash} 
 \newcommand \tr{\hbox{\rm tr}}
 \newcommand\ord{\text{ord}}
\newcommand \Tr{\hbox{\rm Tr}}
\newcommand\HK{\mathcal{H}^G_{K_f}}
\newcommand\HKS{\mathcal{H}^G_{K_f,\place}}
\newcommand\HKv{\mathcal{H}^G_{K_v}}
\newcommand\HGS{\mathcal{H}^{G,\place}}
\newcommand\HKp{\mathcal{H}^G_{K_p}}
\newcommand\HKpo{\mathcal{H}^G_{K_p^0}}
\newcommand\ch{{\bf ch}}

\newcommand\M{\mathcal{M}}
\newcommand\Ml{\M_\lambda}
\newcommand\tMl{\tilde{\Ml}}
\newcommand\tM{\widetilde{\mathcal{M}}}
\newcommand\tMZ{\tM_\Z}
\newcommand\tsigma{\ul{\sigma}}
\newcommand \pkt{\bullet}
\newcommand\tH{\widetilde{\mathcal{H}}}
\newcommand\Mot{{\bf M}} 
\newcommand\eff{{\rm eff}}
\newcommand\Aql{A_{\q}(\lambda)}
\newcommand\wl{w\cdot\lambda}
\newcommand\wlp{w^\prime\cdot\lambda} 

\def\w{{\bf w}} 
\def\d{{\sf d}}
\def\e{{\bf e}} 
\def\x{{\tt x}}
\def\y{{\tt y}}
\def\v{{\sf v}}
\def\q{{\sf q}} 
\def\ff{{\bf f}}
\def\bk{{\bf k}}
 
\def\Ext{{\rm Ext}}
\def\Aut{{\rm Aut}}
\def\Hom{{\rm Hom}}
\def\Ind{{\rm Ind}}
\def\aInd{{}^{\rm a}{\rm Ind}}
\def\aIndPG{\aInd_{\pi_0(P(\R)) \times P(\A_f)}^{\pi_0(G(\R)) \times G(\A_f)}}
\def\aIndQG{\aInd_{\pi_0(Q(\R)) \times Q(\A_f)}^{\pi_0(G(\R)) \times G(\A_f)}}
\def\Gal{{\rm Gal}}
\def\End{{\rm End}} 
\def\cm{{\rm cm}} 
\newcommand\Coh{{\rm Coh}}  
\newcommand\Eis{{\rm Eis}}
\newcommand\Res{\mathrm{Res}}
\newcommand\place{\mathsf{S}}
\newcommand\emb{\mathcal{I}} 
\newcommand\LB{\mathcal{L}}  
\def\Hod{{\mathcal{H}od}}
\def\Crit{{\rm Crit}}
  
 \newcommand\ip{\pi_f\circ \iota} 
\newcommand\Wp{W_{\pi_\infty\times \ip}}
\newcommand\Wpc{W^{\text{cusp}}_{\pi_\infty\times \pi_f\circ\iota}}
\newcommand\Lcusp{  L^2_{\text{cusp}}(G(\Q)\bs G(\A_f)/K_f)}
\newcommand\MiC{\tM_{\iota\circ\lambda,\C} }
\newcommand\miC{\M_{\iota\circ\lambda,\C} }
\newcommand\tpl{{}^\iota}
\newcommand\Id{\rm Id}
\newcommand\Lr{L^{\text{rat}}}
\newcommand \iso{ \buildrel \sim \over\longrightarrow} 
\newcommand\us{\ul\sigma}
\newcommand\qvs{q_v^{-z}}
\newcommand \into{\hookrightarrow}
\newcommand\ppfeil[1]{\buildrel #1\over \longrightarrow}
\newcommand\eb{{}^\iota}
    
\def\bfpi{\mathbf{\Pi}}
\def\bfdelta{\mathbf{\Delta}}

\def\sI{\mathscr{I}}
\def\sU{\mathscr{U}}
\def\sJ{\mathscr{J}}

\bigskip

\begin{center}
{\bf Introduction}
\end{center}

\bigskip

The principal aim of this article is to prove a rationality result for the ratios of successive critical values of Rankin--Selberg $L$-functions of $\GL(n) \times \GL(n')$ over a totally imaginary number field $F$ via a study of rank-one Eisenstein cohomology for the group $\GL(N)/F$ where $N = n+n'.$ 
This article is a generalization of the methods and results of a previous work with 
G\"unter Harder \cite{harder-raghuram-book} that studied such a situation for a totally real base field. 
A fundamental tool is the cohomology of local systems on the Borel--Serre compactification of a locally symmetric space for $\GL(N)/F$. The technical heart of the article pertains to analyzing the cohomology of the Borel--Serre boundary, especially for the contribution coming from maximal parabolic subgroups, that leads to an interpretation of the celebrated theorem of Langlands on the constant term of an Eisenstein series in terms of maps in cohomology. \medskip

Let $F$ be a totally imaginary number field and $F_0$ its maximal totally real subfield. There is at most one totally imaginary 
quadratic extension $F_1$ of $F_0$ contained in $F$, giving us two distinct cases that has a bearing on much that is to follow: 
\begin{enumerate} 
\item {\bf CM}: when there is indeed such an $F_1$; then $F_1$ is the maximal CM subfield of $F$; 
\item {\bf TR}: if not then put $F_1 = F_0$; here $F_1$ is the maximal totally real subfield of $F$. 
\end{enumerate}
The {\bf TR}-case imposes the restriction that existence of a critical point for Rankin--Selberg $L$-functions implies $nn'$ is even. The {\bf CM}-case, arguably the more interesting of the two,  
will impose no such restrictions; furthermore, whether $F$ itself is CM ($F = F_1$) or not ($[F:F_1] \geq 2$) has a delicate effect on Galois equivariance properties of the rationality results.

\medskip

Put $G = G_N = \Res_{F/\Q}(\GL(N)/F),$ and $T = T_N$ the restriction of scalars of the diagonal torus in $\GL(N).$ 
Let $E$ stand for a large enough finite Galois extension of $\Q$ in which $F$ can be embedded. The meaning of large enough will be clear from context. Take a dominant integral weight $\lambda \in X^*(T \times E),$ and let $\M_{\lambda, E}$ be the algebraic finite-dimensional absolutely-irreducible representation of $G \times E$ with highest weight $\lambda.$ For a level structure $K_f \subset G(\A_f)$, where $\A_f$ is the ring of finite adeles of $\Q,$ let $\tM_{\lambda, E}$ denote the sheaf of $E$-vector spaces on the locally symmetric space $\SGK$ of $G$ with level $K_f$ 
(see Sect.\,\ref{sec:sheaves-loc-sym-sp}).
A fundamental object of interest is the cohomology group 
$H^\bullet(\SGK, \tM_{\lambda, E})$. The Borel--Serre compactification $\BSC = \SGK \cup \partial\SGK$ gives the long exact sequence 
$$ \ \cdots   H^i_c(\SGK, \tM_{\lambda, E}) 
\stackrel{\mathfrak{i}^\bullet}{\longrightarrow}   H^i(\BSC, \tM_{\lambda,E}) 
\stackrel{\mathfrak{r}^\bullet}{\longrightarrow } H^i(\partial \SGK, \tM_{\lambda,E})  
\stackrel{\fd^\bullet}{\longrightarrow} H^{i+1}_c( \SGK, \tM_{\lambda, E} )  \cdots \ 
$$
of modules for the action of a Hecke algebra $\HK.$ 
Inner cohomology is defined as $H^\bullet_! = {\rm Image}(H^\bullet_c \to H^\bullet),$ within which is 
a subspace 
$H^\bullet_{!!} \subset H^\bullet_!$ called strongly-inner cohomology which has the property of capturing cuspidal cohomology at an arithmetic level, i.e., for any embedding of fields 
$\iota : E \to \C$, one has $H^\bullet_{!!}(\SGK, \tM_{\lambda,E}) \otimes_{E,\iota} \C = 
H^\bullet_{\rm cusp}(\SGK, \tM_{{}^\iota\lambda, \C}).$ If $\pi_f$ is a simple Hecke module appearing in $H^\bullet_{!!}(\SGK, \tM_{\lambda,E})$, 
then ${}^\iota\pi_f$ is the $K_f$-invariants of the finite part of a cuspidal automorphic representation ${}^\iota\pi$ of $G(\A) = \GL_N(\A_F),$ 
whose archimedean component ${}^\iota\pi_\infty$ has nonzero relative Lie algebra cohomology 
with respect to $\M_{{}^\iota\lambda, \C}$; 
denote this as $\pi_f \in \Coh_{!!}(G,\lambda)$. Only strongly-pure dominant integral weights will support cuspidal cohomology; the structure of the set $X^+_{00}(T \times E)$ 
of all such strongly-pure weights has an important bearing on the entire article; see Sect.\,\ref{sec:pure}. 
The cohomology of the Borel--Serre boundary $H^\bullet(\partial \SGK, \tM_{\lambda,E}),$ as a Hecke-module, 
is built via a spectral sequence from modules that are parabolically induced from the cohomology of Levi subgroups; see Sect.\,\ref{sec:coh-of-bdry}. 
For $N = n+n',$ with positive integers $n$ and $n'$, similar notations will be adopted for 
$G_n = \Res_{F/\Q}(\GL(n)/F),$ $T_n$, $G_{n'},$ $T_{n'},$ etc.  
Let $\mu \in X^+_{00}(T_n \times E)$ and $\mu' \in X^+_{00}(T_{n'} \times E),$ and consider 
$\sigma_f \in \Coh_{!!}(G_n,\mu)$ and $\sigma'_f \in \Coh_{!!}(G_{n'},\mu').$ 
  The contragredient  of ${}^\iota\sigma'$  is  denoted ${}^\iota\sigma'^\v.$
For $\iota : E \to \C,$ a point $m \in \tfrac{N}{2} + \Z$ is said to be critical for the completed
Rankin--Selberg $L$-function $L(s, {}^\iota\sigma \times {}^\iota\sigma'^\v),$ if the archimedean 
$\Gamma$-factors on either side of the functional equation are finite at $s = m.$ 
The critical set for $L(s, {}^\iota\sigma \times {}^\iota\sigma'^\v)$ is described in Prop.\,\ref{prop-crit-mu-mu'}. The main result (Thm.\,\ref{thm:main}) of this article is the following

\medskip

\noindent {\bf Theorem.}{\it  \ \  
Assume that $m$ and $m+1$ are critical for $L(s, {}^\iota\sigma \times {}^\iota\sigma'^\v).$  
\begin{enumerate}
\medskip
\item[(i)] If $L(m+1, {}^\iota\sigma \times {}^\iota\sigma'^\v) = 0$ for some $\iota$, then 
$L(m+1, {}^\iota\sigma \times {}^\iota\sigma'^\v) = 0$ for every $\iota.$ 
\medskip
\item[(ii)] Assume $F$ is in the {\bf CM}-case. 
Suppose $L(m+1, {}^\iota\sigma \times {}^\iota\sigma'^\v) \neq 0$, then  
$$
|\delta_{F/\Q}|^{- \tfrac{n n'}{2}} \cdot \frac{L(m, {}^\iota\sigma \times {}^\iota\sigma'^\v)}{L(m+1, {}^\iota\sigma \times {}^\iota\sigma'^\v)} 
\ \in \ \iota(E), 
$$ 
where, $\delta_{F/\Q}$ is the discriminant of $F/\Q$. 
For any $\gamma \in \Gal(\bar\Q/\Q),$ we have:
\begin{multline*}
   \gamma
   \left(|\delta_{F/\Q}|^{- \tfrac{n n'}{2}} \cdot \frac{L(m, {}^\iota\sigma \times {}^\iota\sigma'^\v)}{L(m+1, {}^\iota\sigma \times {}^\iota\sigma'^\v)}\right)
    \\ 
    = \ 
   \varepsilon_{\iota, w}(\gamma) \cdot \varepsilon_{\iota, w'}(\gamma) \cdot 
   |\delta_{F/\Q}|^{- \tfrac{n n'}{2}} \cdot 
   \frac{L(m, {}^{\gamma\circ\iota}\sigma \times {}^{\gamma \circ\iota} \sigma'^\v)}{L(m+1, {}^{\gamma \circ \iota}\sigma \times {}^{\gamma \circ\iota}\sigma'^\v)} \, ,
\end{multline*}
where $ \varepsilon_{\iota, w}(\gamma), \, \varepsilon_{\iota, w'}(\gamma) \in \{\pm1\}$ are certain signatures (see Def.\,\ref{def:signature}) 
whose product 
is trivial if $F$ is a CM field but can be nontrivial in general. 
\medskip
\item[(iii)] Assume $F$ is in the {\bf TR}-case. Then $nn'$ is even. 
Suppose $L(m+1, {}^\iota\sigma \times {}^\iota\sigma'^\v) \neq 0$, then  
$$
\frac{L(m, {}^\iota\sigma \times {}^\iota\sigma'^\v)}{L(m+1, {}^\iota\sigma \times {}^\iota\sigma'^\v)} 
\ \in \ \iota(E), 
$$ 
and for any $\gamma \in \Gal(\bar\Q/\Q),$ we have:
$$   
\gamma \left(
 \frac{L(m, {}^\iota\sigma \times {}^\iota\sigma'^\v)}{L(m+1, {}^\iota\sigma \times {}^\iota\sigma'^\v)}\right)
    \ = \ 
\frac{L(m, {}^{\gamma\circ\iota}\sigma \times {}^{\gamma \circ\iota} \sigma'^\v)}{L(m+1, {}^{\gamma \circ \iota}\sigma \times {}^{\gamma \circ\iota}\sigma'^\v)}.
$$
\end{enumerate}
}

\bigskip

For the proof, consider Eisenstein cohomology of $G$ which,  
by definition, is the image of $H^\bullet(\BSC, \tM_{\lambda,E}) \stackrel{\mathfrak{r}^\bullet}{\longrightarrow } H^\bullet(\partial \SGK, \tM_{\lambda,E})$. Let $P = \Res_{F/\Q}(P_{(n,n')})$, where $P_{(n,n')}$ is the standard maximal parabolic subgroup of $\GL_N$ of type 
$(n,n'),$ and let $U_P$ be the unipotent radical of $P.$
The first technical theorem (Thm.\,\ref{thm:manin-drinfeld}) stated as the `Manin--Drinfeld principle',  
says that the algebraically and parabolically induced representation $\aInd_{P(\A_f)}^{G(\A_f)}(\sigma_f \times \sigma'_f)$ 
together with its partner across a standard intertwining operator splits off as an isotypic component from the cohomology of the boundary as a Hecke module. 
The next technical result (Thm.\,\ref{thm:rank-one-eis}) 
is to prove that the image of Eisenstein cohomology in this isotypic component is analogous to a line in a two-dimensional plane. If one passes to a transcendental situation using an embedding $\iota : E \to \C$, then via Langlands's constant term theorem, 
the slope of this line is the ratio of $L$-values $L(m, {}^\iota\sigma \times {}^\iota\sigma'^\v)/L(m+1, {}^\iota\sigma \times {}^\iota\sigma'^\v),$ 
times the factor $|\delta_{F/\Q}|^{-n n'/2}.$ This latter factor involving the discriminant of the base field 
arises as the volume of $U_P(\Q)\backslash U_P(\A)$ needed to normalise  
the measure so that the constant term map, in cohomology, is the restriction map to the boundary stratum corresponding to $P$.

\bigskip

There are two subproblems to solve along the way whose proofs are totally different from those of the corresponding statements in \cite{harder-raghuram-book}. The first is  
a {\it combinatorial lemma} (see Sect.\,\ref{sec:comb-lemma}) involving the weights $\mu$ and $\mu'$ which says that the points $-\tfrac{N}{2}$ and $1-\tfrac{N}{2}$ are critical if and only if the 
induced module considered above contributes to cohomology of the boundary in an optimal degree; this involves combinatorial subtleties on Kostant representatives in Weyl groups. 
The ingredient $w$ in the signature $\varepsilon_{\iota, w}(\gamma)$ is a Kostant representative determined by $\mu$ and $\mu'$ via this combinatorial lemma, and 
$w'$ in $\varepsilon_{\iota, w'}(\gamma)$ is a Kostant representative determined by $w$ via Lem.\,\ref{lem:kostant-P-Q}. 
The combinatorial lemma also says that we only need to prove a rationality result for the particular ratio 
$L(-N/2, {}^\iota\sigma \times {}^\iota\sigma'^\v)/L(1-N/2, {}^\iota\sigma \times {}^\iota\sigma'^\v),$ 
for a sufficiently general class of weights $\mu$ and $\mu'$; see \ref{para:comb-lemma-reduction}. 
The second is to show that the map induced in relative Lie algebra cohomology by the archimedean standard intertwining operator contributes the ratio of local archimedean $L$-values; 
see Sect.\,\ref{sec:T-st-infinity}, and especially Prop.\,\ref{prop:basic-Tst-GLN-in-cohomology}; this uses the factorization of the standard intertwining operator into rank-one operators.

\bigskip

Previous work on the arithmetic of $L$-functions over a totally imaginary field especially worth mentioning in the context of this article are as follows. 
For 
$n = n' =1$, the rationality result in $(ii)$ is due to Harder \cite[Cor.\,4.2.2]{harder-inventiones}. In general, see Blasius \cite{blasius-annals} and Harder \cite{harder-inventiones} for $\GL_1$, see also Harder--Schappacher \cite{harder-schappacher}; 
Hida \cite{hida-duke} for $\GL_2 \times \GL_1$ and $\GL_2 \times \GL_2$; Grenie \cite{grenie} for $\GL_n \times \GL_n$; Harris \cite{harris-crelle} for standard $L$-functions for unitary groups which may be 
construed as a subclass of $L$-functions for $\GL_n \times \GL_1$; Harder \cite{harder-gln} and M\oe glin \cite{moeglin} 
for some general aspects of $\GL_n$--the result contained in $(i)$ %for the standard $L$-function for $\GL(n)$ of a cuspidal cohomological representation 
is due to M\oe glin \cite[Sect.\,5]{moeglin}, although our proof is different from \cite{moeglin}. 
Furthermore, see the author's paper \cite{raghuram-forum}, Grobner--Harris 
\cite{grobner-harris}, and Januszewski \cite{januszewski} 
for $\GL_n \times \GL_{n-1}$; Sachdeva \cite{sachdeva} for $\GL_3 \times \GL_1$; and Lin \cite{lin}, Grobner--Harris--Lin \cite{grobner-harris-lin}, Grobner--Lin \cite{grobner-lin}, and Grobner--Sachdeva \cite{grobner-sachdeva} for different aspects for $\GL_n \times \GL_{n'}$.  
Amongst these, the results of \cite{grobner-harris-lin}, \cite{grobner-lin}, \cite{grobner-sachdeva}, and \cite{lin} come close in scope to the results of this paper, however, their methods are different and work over a base field that is assumed to be CM, while often needing a polarization assumption on their representations to descend to a unitary group, and in some situations being conditional on expected but unproven hypotheses. In contrast, the method pursued here, which is a generalization of Harder \cite{harder-inventiones} and 
my work with Harder \cite{harder-raghuram-CR}, \cite{harder-raghuram-book}, is totally different from all the other papers mentioned above, and furthermore, our results are self-contained, unconditional, and complete.

\medskip

There is a celebrated conjecture of Deligne \cite[Conj.\,2.7]{deligne} on the critical values of motivic $L$-functions. A fundamental aspect of the Langlands program is a 
conjectural dictionary between strongly-inner Hecke modules $\sigma_f$  and pure regular rank $n$ motives $M(\sigma_f)$ over $F$ with coefficients in $E$ (see, for example, 
\cite[Chap.\,7]{harder-raghuram-book}). 
Granting this dictionary, 
Deligne's conjecture applied to $M := \Res_{F/\Q}(M(\sigma_f) \otimes M(\sigma'^\v_f))$ conjecturally describes a rationality result for the array 
$\{L(m, {}^\iota\sigma \times {}^\iota\sigma'^\v)\}_{\iota : E \to \C}$ of critical values in terms of certain periods $c^\pm(M)$ of $M$. To see the main theorem of this article from the perspective of motivic $L$-functions necessitates a relation between $c^+(M)$ and $c^-(M)$, for which we refer the reader to the forthcoming article with Deligne \cite{deligne-raghuram}. The appearance of the 
signatures $\varepsilon_{\iota, w}(\gamma)$ and $\varepsilon_{\iota, w'}(\gamma)$ was in fact suggested by certain calculations in \cite{deligne-raghuram} that also allows us to 
recast Thm.\,\ref{thm:main} 
more succinctly as follows. Suppose $F$ is in the {\bf CM}-case, and suppose 
$F_1 = F_0(\sqrt{D})$ for a totally negative $D \in F_0$, then define 
$\Delta_{F} = N_{F_0/\Q}(D)^{[F:F_1]/2}.$ Suppose $F$ is in the {\bf TR}-case then define $\Delta_{F} =1.$ Fix $\i = \sqrt{-1}.$
The rationality result can be restated as
$$
(\i^{d_F/2} \Delta_F)^{nn'} \, \frac{L(m, {}^\iota\sigma \times {}^\iota\sigma'^\v)}{L(m+1, {}^\iota\sigma \times {}^\iota\sigma'^\v)} \ \in \iota(E), 
$$
(see \ref{sec:thm-implies-conj}) and the reciprocity law takes the shape that for every $\gamma \in \Gal(\bar\Q/\Q)$ one has: 
$$
\gamma \left(
(\i^{d_F/2} \Delta_F)^{nn'} \, \frac{L(m, {}^\iota\sigma \times {}^\iota\sigma'^\v)}{L(m+1, {}^\iota\sigma \times {}^\iota\sigma'^\v)} \right) \ = \ 
(\i^{d_F/2} \Delta_F)^{nn'} \, 
\frac{L(m, {}^{\gamma\circ\iota}\sigma \times {}^{\gamma\circ\iota}\sigma'^\v)}{L(m+1, {}^{\gamma\circ\iota}\sigma \times {}^{\gamma\circ\iota}\sigma'^\v)}.
$$
In the {\bf TR}-case, existence of a critical point will necessitate $nn'$ to be even, and so we may ignore the term $(\i^{d_F/2} \Delta_F)^{nn'} = \pm 1$ from 
the rationality result and the reciprocity law.

\medskip

To conclude the introduction, it is worth amplifying the dictum that 
whereas the analytic theory of $L$-functions is not sensitive to the arithmetic nature of the ground field $F$, 
but the arithmetic of special values of $L$-functions is definitively sensitive to the inner structure of $F$. For example, if $F$ is totally real, 
then the integral that Rankin and Selberg studied giving $L$-functions for $\GL(2) \times \GL(2)$ does not seem to admit a cohomological interpretation as an integral coming from, say, Poincar\'e or Serre duality. However, if $F$ is totally imaginary, then the Rankin--Selberg intergal does indeed admit an interpretation in terms of Poincar\'e duality; 
see Hida~\cite{hida-duke}. In a different direction, the period integrals of cusp forms on $\GL(2n)$ integrated over $\GL(n) \times \GL(n)$ that 
Friedberg--Jacquet \cite{friedberg-jacquet} studied to get the standard $L$-function of $\GL(2n)$ 
can be interpreted in cohomology over a totally real field (see my papers with Grobner \cite{grobner-raghuram}, and with Dimitrov and Januszewski \cite{dimitrov-januszewski-r}), however, if the base field is imaginary then there seems to be no such cohomological interpretation. 
This dependence on the arithmetic of the base field stems not only from the cohomological vagaries of the representations of 
$\GL_m(\R)$ vis-\`a-vis those of $\GL_m(\C)$, but also because the inner structure of the base field informs some of the constructions with algebraic groups over such base fields--this is why one sees the signatures $\varepsilon_{\iota, w}(\gamma)$ and 
$\varepsilon_{\iota, w'}(\gamma)$ when $F$ is in the {\bf CM}-case but not when $F$ is in the {\bf TR}-case; such terms did not appear 
when the base field is totally real \cite{harder-raghuram-book} or a CM field \cite{raghuram-imrn}.

\medskip

{\SMALL
{\it Suggestions to the reader:} Any one wishing to read this paper seriously, will need my monograph with Harder \cite{harder-raghuram-book} 
by his/her side. I have tried to make this manuscript reasonably self-contained, but any time I felt there was nothing to be gained by repetition, I have referenced 
\cite{harder-raghuram-book}. For a finer appreciation, the reader should compare the formal similarities of the results of this manuscript and the results of 
\cite{harder-raghuram-book}, while noting the very different proofs--especially with the proofs of the combinatorial lemma in 
Sect.\,\ref{sec:comb-lemma}, and the calculations involving the archimedean intertwining operator in Sect.\,\ref{sec:T-st-infinity}. 
For a first reading I recommend that the reader skim through Sect.\,\ref{sec:prelims} to get familiar with 
the notations, and assume the statements of Prop.\,\ref{prop-crit-mu-mu'}, Lem.\,\ref{lem:comb-lemma}, Prop.\,\ref{prop:irreducible-isomorphism-GLN} and 
Prop.\,\ref{prop:basic-Tst-GLN-in-cohomology} without worrying too much about their technical proof.}

\medskip

{\SMALL
{\it Acknowledgements:} 
I am intellectually indebted to G\"unter~Harder. The main result in this paper is a direct offshoot of my decade-long collaboration with Harder when we worked on Eisenstein cohomology for $\GL_N$ over a totally real field. This manuscript could well have been a joint article with Harder, except that he generously let me work on $\GL_N$ over a totally imaginary field by myself. I thank Don~Blasius, Michael~Harris, and Freydoon~Shahidi for their interest and constant encouragement to all my endeavours with the special values of various automorphic $L$-functions. I thank Haruzo~Hida for helpful correspondence on totally imaginary fields. I am grateful to the Charles Simonyi Endowment that funded my membership at the Institute for Advanced Study, Princeton, during the Spring and Summer terms of 2018, when I mostly worked on fixing a proof of the combinatorial lemma and had a first draft of this manuscript. Finally, I thank Pierre~Deligne for 
several discussions and especially 
his motivic explanations on the appearance of certain signatures which I had overlooked in an earlier version of this manuscript. I acknowledge support from a MATRICS grant MTR/2018/000918 of the Science and Engineering Research Board, Department of Science and Technology, Government of India.}

{\SMALL \tableofcontents}

\medskip
\section{\bf Preliminaries}
\label{sec:prelims}

\medskip
\subsection{Some basic notation}

\subsubsection{\bf The base field}
Let $F$ stand for a totally imaginary finite extension of $\Q$ of degree $d_F =  [F:\Q].$ 
Let $\Sigma_{F}  = \Hom(F,\C)$ be the set of all complex embeddings, and $\place_\infty$ denote the set of archimedean places of $F$; denote the cardinality of 
$\place_\infty$ by ${\sf r}$, hence 
$d_F = 2{\sf r}.$  There is a canonical surjection 
$\Sigma_{F} \to \place_\infty;$ the fibre over $v \in \place_\infty$ is a pair $\{\eta_v, \bar{\eta}_v\}$ 
of conjugate embeddings; via such a non-canonical choice of $\eta_v$ fix the identification $F_v \simeq \C.$ 
Let $\A = \A_\Q$ be the ad\`ele ring of $\Q$, and 
$\A_f = \A^\infty$ the ring of finite ad\`eles. Then $\A_F = \A \otimes_\Q F$, and $\A_{F,f} = \A_f \otimes_\Q F$. 
When $F$ is a CM field, i.e., a totally imaginary quadratic extension of a totally real extension $F^+$ (say) of $\Q;$ then $\Sigma_{F^+}  = \Hom(F^+,\C) = \Hom(F^+,\R),$ and 
the restriction from $F$ to $F^+$ gives a canonical surjection $\Sigma_F \to \Sigma_{F^+}$; the fiber over $\eta \in \Sigma_{F^+}$ is a pair of conjugate embeddings that will be denoted  as $\{\eta, \bar\eta\}$, with the understanding that the choice of $\eta$ in $\{\eta, \bar\eta\}$ though not canonical is nevertheless fixed once and for all. 
If $\Sigma_F = \{\nu_1, \dots, \nu_{d_F}\}$, $\{\omega_1,\dots,\omega_{d_F}\}$ is a $\Q$-basis of $F$, and $\theta_F = \det[\sigma_i(\omega_j)]$, then $\theta_F^2$ is the absolute discriminant 
$\delta_{F/\Q}$ of $F$. The square root of the absolute value of the discriminant, $|\delta_{F/\Q}|^{1/2},$ as an element of $\R^\times/\Q^\times,$ is independent of the enumeration and the choice of basis. Let $\i$ denote a fixed choice of $\sqrt{-1}.$ Since $F$ is totally imaginary, $\i^{d_F/2} \cdot \theta_F$ is a real number whose absolute value is $|\delta_{F/\Q}|^{1/2}.$

\medskip
\subsubsection{\bf The groups}
For an integer $N \geq 2$, let $G_{0} = \GL_N/F$, and put $G = \Res_{F/\Q}(G_{0})$ as the $\Q$-group obtained by the Weil restriction of scalars. To emphasize the dependence on $N$, 
$G_0$ will also be denoted $G_{N,0}$ and similar notation will be adopted for other groups to follow.  
Let $B_0$ be the subgroup of $G_0$ of upper-triangular matrices,  
$T_0$ the diagonal torus in $B_0$, and $Z_0$ the center of $G_0$; the corresponding $\Q$-groups via $\Res_{F/\Q}$ 
will be denoted $B, T,$ and $Z$, respectively. Let $S$ stand for the maximal $\Q$-split torus of $Z$; note that $S \simeq \mathbb{G}_m.$
Let $n$ and $n'$ be positive integers such that $n + n' = N,$ and let $P_0$ be the maximal parabolic subgroup of $G_0$ containing $B_0$ of type 
$(n,n').$ The unipotent radical of $P_0$ is denoted $U_{P_0}$ and  
Levi quotient of $P_0$ is $M_{P_0} = \GL_n \times \GL_{n'}.$  Put $P = \Res_{F/\Q}(P_0)$, and similarly $U_P$ and $M_P.$ The dimension of $U_P$ is $nn'd_F = 2nn'{\sf r}.$

\medskip
\subsection{Sheaves on locally symmetric spaces}
\label{sec:sheaves-loc-sym-sp}
This brief section is very similar to the situation over a totally real base field \cite{harder-raghuram-book}. Most of the concepts in this section apply, possibly with minor modifications, 
to related groups like $\GL_n, \ \GL_{n'}, \ M_{P_0},$ etc.

\medskip
\subsubsection{\bf Locally symmetric spaces}
\label{sec:loc-sym-spaces}
Note that $
G(\R) \ = \ G_0(F \otimes_\Q \R) \ = \ \prod_{v \in \place_\infty} \GL_N(F_v)  \ \simeq \ \prod_{v \in \place_\infty} \GL_N(\C).$
Similarly, $Z(\R) = Z_0(F \otimes_\Q \R) \simeq  \prod_{v \in \place_\infty} \C^\times 1_N,$
where $1_N$ is the identity $N \times N$-matrix; $S(\R) = \R^\times$ sits diagonally in $Z(\R)$. 
The maximal compact subgroup of $G(\R)$ will be denoted $C_\infty$; we have
$C_\infty \ = \ \prod_{v \in \place_\infty} \U(N),$
where $\U(N),$ the usual compact unitary group in $N$-variables, is a maximal compact group of $\GL_N(\C)$. Put $K_\infty = C_\infty S(\R)$ and note that $K_\infty = C_\infty S(\R)^\circ$ is a connected group, since 
$-1 \in S(\R)$ gets absorbed into $C_\infty.$ Define the symmetric space of $G$ as
$\SG \ := \ G(\R)/K_\infty .$
For any open compact subgroup $K_f \subset G(\A_f)$, define the ad\`elic symmetric space: 
$G(\A)/K_\infty  K_f \ = \ 
\SG \times (G(\A_f)/K_f).$ 
On this space $G(\Q)$ acts  properly discontinuously  and we get a quotient
 \begin{equation}
 \label{eqn:map-pi}
 \pi \, : \, G(\R)/K_\infty \times G(\A_f)/K_f  \ \longrightarrow \    
 G(\Q) \backslash \left( G(\R)/K_\infty \times G(\A_f)/K_f \right). 
\end{equation}
The target space, called the {\it ad\`elic locally symmetric space of $G$ with level structure $K_f$}, is denoted: 
$\SGK = G(\Q) \backslash G(\A) / K_\infty K_f.$
A typical element in the ad\`elic group $G(\A) = G(\R) \times G(\A_f)$ will be denoted $\underline g = g_\infty \times{ \ul  g}_f.$ 
As in \cite[Sect.\,2.1.4]{harder-raghuram-book} one has $\SGK \  \cong \ \coprod_{i=1}^m  \ \Gamma_i \backslash G(\R) / K_\infty$;  
if necessary, replacing $K_f$ by a subgroup of finite-index, assume that each $\Gamma_i$ is torsion-free. 
It is easy to see that
${\rm dim}(\SGK) \ = \ {\rm dim}(G(\R)/K_\infty) \ = \ {\rm dim}(G(\R)/C_\infty) - 1 \ = \ {\sf r}N^2 - 1.$

\medskip
\subsubsection{\bf The field of coefficients $E$} 
Throughout this paper, let $E/\Q$ be a `large enough' finite Galois extension that takes a copy of $F.$
(The meaning of $E$ being large enough will depend on the context: for example, large enough so that some Hecke summand in inner-cohomology would split over $E$. To relate cohomology groups with automorphic forms, one could drop finiteness and take $E = \C$; or anticipating $p$-adic interpolation of the $L$-values considered here, $E$ could be a large enough $p$-adic field.) An embedding $\iota : E \to \C$ gives a bijection 
$\iota_* : \Hom(F,E) \to \Hom(F,\C)$ given by composition: $\iota_*\tau = \iota \circ \tau.$
If $E = \C$, then there is a natural notion of complex-conjugation on $\Hom(F,\C)$: defined by 
$\bar{\eta}(x) = \overline{\eta(x)}.$ But, on $\Hom(F,E)$ there is no natural notion of complex-conjugation; however, using 
$\iota : E \to \C$ we can consider the conjugate $\overline{\tau}^\iota$ of $\tau$ defined 
as: $\iota_*(\overline{\tau}^\iota) = \overline{\iota_*\tau}.$ 
If $F$ is a CM field, then let $\{1, c\}$ denote the Galois group of $F/F^+;$ restriction $\tau \mapsto \tau|_{F^+}$ gives a surjective map
$\Hom(F, E) \twoheadrightarrow \Hom(F^+,E);$ for $\tau \in \Hom(F, E)$ define $\tau^c$ by $\tau^c(x) = \tau(c(x))$ for all $x \in F$, then 
$\{\tau, \tau^c\}$ is the fiber above $\tau|_{F^+}.$ If $E = \C$, then $\tau^c = \bar{\tau}.$

\medskip
\subsubsection{\bf Characters of the torus $T$} For $E$ as above, let 
$X^*(T \times E) := \Hom_{E-{\rm alg}}(T \times E, \mathbb{G}_m),$
where $ \Hom_{E-{\rm alg}}$ is to mean homomorphisms of $E$-algebraic groups. There is a natural action of ${\rm Gal}(E/\Q)$ on 
$X^*(T \times E)$. Since $T = \Res_{F/\Q}(T_0)$, one has  
$$
X^*(T \times E) \ = \ \bigoplus_{\tau : F \to E} X^*(T_0 \times_{F,\tau} E)  \ = \ \bigoplus_{\tau : F \to E} X^*(T_0), 
$$
where the last equality is because $T_0$ is split over $F$. Let $X^*_\Q(T \times E) = X^*(T \times E) \otimes \Q.$ 
The weights are parametrized as in \cite{harder-raghuram-book}: $\lambda \in X^*_\Q(T \times E)$ will be written as 
$\lambda = (\lambda^\tau)_{\tau : F \to E}$ with 
$$
\lambda^\tau \ = \ \sum_{i=1}^{N-1} (a^\tau_i-1)  \bfgreek{gamma}_i \ + \ d^\tau \cdot \bfgreek{delta}_N 
 \ = \ (b^\tau_1, b^\tau_2, \dots, b^\tau_N), 
$$
where, $\bfgreek{gamma}_i$ is the $i$-th fundamental weight for $\SL_N$ extended to $\GL_N$ by making it trivial on the center, and $\bfgreek{delta}_N$ is the determinant character of $\GL_N.$ If $r_{\lambda} := (Nd - \sum_{i=1}^{N-1} i (a_i-1))/N,$ 
then $b_1 =  a_1 + a_2 + \dots + a_{N-1} - (N-1) + r_{\lambda}, \ 
b_2  =  a_2 + \dots + a_{N-1} - (N-2) + r_{\lambda}, \dots, b_{N-1} =  a_{N-1} - 1 + r_{\lambda}, \ 
b_N  =  r_{\lambda},$ and conversely, 
$a_i  - 1 =  b_i - b_{i+1}, \ d  =  (b_1+\dots+b_N)/N.$
A weight $\lambda = \sum_{i=1}^{N-1} (a_i-1) \bfgreek{gamma}_i+ d \cdot \bfgreek{delta}_N  = (b_1,\dots,b_N)  \in X^*_\Q(T_0)$ is an integral weight if and only if 
$$
\lambda \in X^*(T_0) \ \Longleftrightarrow \ 
b_i \in \Z, \ \forall i 
 \ \Longleftrightarrow \ 
\left\{\begin{array}{l} 
a_i \in \Z, \quad 1 \leq i \leq N-1, \\
Nd \in \Z, \\
Nd \equiv \sum_{i=1}^{N-1} i (a_i-1) \pmod{N}.
\end{array}\right.
$$
A weight $\lambda = (\lambda^\tau)_{\tau : F \to E} \in X^*_\Q(T \times E)$ is integral if and only if each $\lambda^\tau$ is integral.
Next, an integral weight $\lambda \in X^*(T_0)$ is dominant, for the choice of the Borel subgroup being $B_0$,  if and only if 
$$
b_1 \geq b_2 \geq \dots \geq b_N
 \ \Longleftrightarrow \ a_i \geq 1 \  
 \mbox{for $1 \leq i \leq N-1.$ \ (There is no condition on $d$.)}
$$
A weight $\lambda = (\lambda^\tau)_{\tau : F \to E} \in X^*_\Q(T \times E)$ is dominant-integral if and only if each $\lambda^\tau$ is dominant-integral. Let 
$X^+(T \times E)$ stand for the set of all dominant-integral weights.

\medskip
\subsubsection{\bf The sheaf $\tM_{\lambda, E}$}
For $\lambda \in X^+(T \times E)$, put 
 $\M_{\lambda, E} \ = \ \bigotimes_{\tau : F \to E} \M_{\lambda^\tau},$
 where $\M_{\lambda^\tau}/E$ is the absolutely-irreducible finite-dimensional representation of 
 $G_0 \times_\tau E = \GL_n/F \times_\tau E$ with highest weight  $\lambda^\tau.$ Denote this representation as 
 $(\rho_{\lambda^\tau}, \M_{\lambda^\tau})$. The group $G(\Q) = \GL_n(F)$ acts on $\M_{\lambda, E}$ diagonally, i.e., 
$a \in G(\Q)$ acts on a pure tensor
$\otimes_\tau m_\tau$ via: 
$a \cdot (\otimes_\tau m_\tau) \ = \ \otimes_\tau \rho_{\lambda^\tau}(\tau(a))(m_\tau).$
This representation gives a sheaf $\tM_{\lambda, E}$ of $E$-vector spaces on $\SGK$: the sections 
over an open subset $V\subset \SGK$ are the locally constant functions 
$s: \pi^{-1}(V) \to \M_{\lambda, E}$ such that $s(a v)  = \rho(a) s(v)$ for all $a \in G(\Q),$  
where $\pi$ is as in (\ref{eqn:map-pi}). 

\medskip
Let us digress for a moment to clarify a certain point that seemingly causes some confusion. 
In the definition of $\SGK,$ one could have divided by $Z(\R) C(\R)$ instead of $K_\infty$, i.e., one can 
consider $G(\Q) \backslash G(\A) / Z(\R) C(\R) K_f;$ over this space the  
same construction of the sheaf $\tM_{\lambda, E}$ carries through, however, for it to be nonzero 
the central character of $\rho_\lambda$ has to have the type of an algebraic Hecke character of $F$; 
(see \cite[1.1.3]{harder-inventiones}). Let $\lambda = (\lambda^\tau)_{\tau : F \to E}\in X^+(T \times E)$, and suppose 
$\lambda^\tau = \sum_{i=1}^{N-1} (a^\tau_i-1) \bfgreek{gamma}_i + d^\tau \cdot  \bfgreek{delta},$ the condition on the 
central character means $d^{\iota \circ \tau} + d^{\overline{\iota \circ \tau}}$ 
is a constant independent every embedding $\iota : E \to \C$, and every $\tau \in \Hom(F,E).$ 
Define $X^+_{\rm alg}(T \times E)$ to be the subset of $X^+(T \times E)$ consisting of all dominant-integral weights which 
satisfy the {\it algebraicity} condition that 
``$d^{\iota \circ \tau} + d^{\overline{\iota \circ \tau}} = {\rm constant}$'' for all $\tau \in \Hom(F,E)$ and for all $\iota : E \hookrightarrow \C$. 
To end the digression, for 
the sheaf $\tM_{\lambda,E}$ on $\SGK$, at this moment we do not need to impose this algebraicity condition, however, later on for the sheaf to support interesting 
cohomology, such as cuspidal cohomology, we will be needing the condition of strong-purity that will imply algebraicity.  

\medskip

If $\lambda \in X^+_{\rm alg}(T \times E)$ and $K_f$ small enough as in \ref{sec:loc-sym-spaces} 
then every stalk of $\tM_{\lambda, E}$ 
is isomorphic to the $E$-vector space $\M_{\lambda,E},$ in which case the sheaf $\tM_{\lambda,E}$ is a local system.

\medskip
\section{\bf The cohomology of $\GL_N$ over a totally imaginary number field}
\label{sec:cohomology_gln}

For $\lambda \in X^+_{\rm alg}(T \times E)$, a basic object of study is the sheaf-cohomology group $H^\bullet(\SGK, \tM_{\lambda,E})$. One of the main tools is a long exact sequence coming from the Borel--Serre compactification. Another tool is the relation of these cohomology groups, by passing to a transcendental situation using an embedding $E \hookrightarrow \C$, to the theory of automorphic forms on $G$. The reader should appreciate that Sect.\,\ref{sec:pure} on strongly-pure weights has some novel features that do not show up over a totally real base field or over a CM field.

\medskip
\subsection{Inner cohomology}
\label{sec:long-e-seq}

Let $\BSC$ be the Borel--Serre compactification of $\SGK$, i.e., 
$\BSC = \SGK \cup \partial\SGK$, where the boundary is stratified as 
$\partial \SGK  = \cup_P \partial_P\SGK$ with $P$ running through the $G(\Q)$-conjugacy classes of proper parabolic subgroups defined over $\Q$. (See Borel--Serre \cite{borel-serre}.) 
The sheaf $\tM_{\lambda, E}$ on $\SGK$ naturally extends to a sheaf on 
$\BSC$ which we also denote by $\tM_{\lambda, E}$. Restriction from $\BSC$ to $\SGK,$  induces an isomorphism in cohomology: 
$
H^\bullet(\BSC, \tM_{\lambda, E}) \iso  H^\bullet(\SGK, \tM_{\lambda, E}).$
Consider the Hecke algebra $\HK = C^\infty_c(G(\A_f)/ \! \!/K_f)$ of all locally constant and compactly supported bi-$K_f$-invariant $\Q$-valued functions on $G(\A_f);$  
take the Haar measure on $G(\A_f)$ to be the product of local Haar measures, and for every prime $p$, the local measure is normalized so that 
$\mathrm{vol}(G(\Z_p)) = 1;$ then $\HK$ is a $\Q$-algebra under convolution of functions. 
The cohomology of the boundary $H^\bullet(\partial \SGK, \tM_{\lambda, E})$ and the cohomology 
with compact supports $H^\bullet_c(\SGK, \tM_{\lambda, E})$ are modules for $\HK$. 
There is a long exact sequence of $\HK$-modules: 
\begin{multline*}
\cdots  \longrightarrow H^i_c(\SGK, \tM_{\lambda, E}) 
\stackrel{\mathfrak{i}^\bullet}{\longrightarrow}   H^i(\BSC, \tM_{\lambda,E}) 
\stackrel{\mathfrak{r}^\bullet}{\longrightarrow } H^i(\partial \SGK, \tM_{\lambda,E}) \stackrel{\fd^\bullet}{\longrightarrow} \\
\stackrel{\fd^\bullet}{\longrightarrow} H^{i+1}_c(\SGK, \tM_{\lambda,E}) \longrightarrow \cdots
\end{multline*}
The image of cohomology with compact supports inside the full cohomology is called {\it inner} or {\it interior} cohomology and is denoted 
$H^{\bullet}_{\, !} := {\rm Image}(\mathfrak{i}^\bullet) = {\rm Im}(H^{\bullet}_c \to H^{\bullet}).$ 
The theory of Eisenstein cohomology is designed to describe the image of 
the restriction map $\mathfrak{r}^\bullet$. 
Inner cohomology  is a semi-simple module for the Hecke-algebra. %(which is seen by going to a transcendental level and seeing that it is a sub-module of square-integrable cohomology). 
If $E/\Q$ is sufficiently large, then there is an isotypical decomposition:  
\begin{align}\label{deco}  
H^\bullet_{\,!}(\SGK, \M_{\lambda,E}) \ =\ 
\bigoplus_{\pi_f \in {\rm Coh}_!(G,K_f,\lambda)} H^\bullet_{\,!}(\SGK, \M_{\lambda,E})(\pi_f), 
\end{align}
where $\pi_f$ is an isomorphism type of an absolutely  irreducible $\HK$-module, i.e., there is an $E$-vector space $V_{\pi_f}$ with an absolutely irreducible action $\pi_f$ of $\HK$.  
Let $\HKp = C^\infty_c(G(\Q_p)/ \! \!/K_p)$ 
be the local Hecke-algebra. The local factors  $\HKp$ are commutative outside a finite set $\place = \place_{K_f}$ of primes and the factors for two different primes commute with each other.  For $p \not\in \place$ the commutative algebra 
$\HKp$ acts
on $V_{\pi_f} $  by a homomorphism $\pi_p : \HKp \to E.$ Let $V_{\pi_p}$ be the one-dimensional $E$-vector space $E$ with the distinguished  basis element   $1\in E$  and with the action $\pi_p$ on it. Then  
$V_{\pi_f}  \ = \ V_{\pi_f,\place} \otimes \otimes^\prime_{p\not\in \place} V_{\pi_p} \ = \ \otimes_{p\in \place} V_{\pi_p}\otimes E,$
where the  absolutely-irreducible $\HKS$-module $ V_{\pi_f,\place}$ module
is decomposed as a tensor product $V_{\pi_f,\place}=  \otimes_{p\in \place}V_{\pi_p}$
of absolutely irreducible $\HKp$-modules.  The Hecke algebra decomposes as 
$\HK\ = \ \HKS\times \otimes_{p\not\in \place}\HKp \ = \ \HKS \times \HGS,$
where the first factor acts on the first factor  $V_{\pi_f, \place}$ and the second factor acts
via the homomorphism $\pi_f^\place :  \HGS \to E.$
The set ${\mathrm{Coh}}_!(G, K_f, \lambda)$ of isomorphism classes which occur with strictly positive multiplicity in \eqref{deco} is called the inner spectrum of $G$ with $\lambda$-coefficients and level structure $K_f.$ Taking the union over all $K_f$, the inner spectrum of $G$ with $\lambda$-coefficients is defined to be: 
${\mathrm{Coh}}_!(G, \lambda) \ = \ \bigcup_{K_f} {\mathrm{Coh}}_!(G, K_f, \lambda). $
Since the inner spectrum is captured, at a transcendental level, by the cohomology of the discrete spectrum, it follows from the strong multiplicity one theorem for the 
discrete spectrum for  $\GL_n$ (see Jacquet \cite{jacquet-residual} and  M\oe glin--Waldspurger \cite{moeglin-waldspurger}) that 
$\pi_f$ is determined by its restriction  $\pi_f^\place$
to the central subalgebra $\HGS$ of $\HK.$

\medskip
\subsection{Cuspidal cohomology} 
\label{sec:cuspidal-coh}

Take $E = \C$ and consider $\lambda \in X^+_{\rm alg}(T \times \C).$ 
Denote $\fg_\infty$ (resp., $\fk_\infty$) the Lie algebra of $G(\R)$ (resp., of $K_\infty = C_\infty S(\R).$)
The cohomology $H^\bullet(\SGK,\tM_{\lambda, \C})$ is the cohomology of the de~Rham complex denoted $\Omega^\bullet(\SGK, \tM_{\lambda, \C}).$
The de~Rham complex is isomorphic to the relative Lie algebra complex: 
$$  
\Omega^\bullet(\SGK, \tM_{\lambda, \C} ) \ = \  \Hom_{K_\infty} (\Lambda^\bullet(\fg_\infty/\fk_\infty), \, 
\cC^\infty(G(\Q)\backslash G(\A)/K_f, \omega_{\lambda}^{-1}|_{S(\R)^0}) \otimes  \M_{\lambda, \C}), 
 $$
 where $\cC^\infty(G(\Q)\backslash G(\A)/K_f, \omega_{\lambda}^{-1}|_{S(\R)^0})$ consists of all smooth functions 
 $\phi : G(\A) \to \C$ such that $\phi(a \, \ul g \, \ul k_f \, s_\infty) = \omega_{\lambda}^{-1}(s_\infty) \phi(\ul g),$ 
 for all $a \in G(\Q)$, $\ul g \in G(\A)$, $\ul k_f \in K_f$ and $s_\infty \in S(\R)^0.$ 
 Abbreviating $\omega_{\lambda}^{-1}|_{S(\R)^0}$ as $\omega_\infty^{-1},$ if $t \in \R_{>0} \cong S(\R)^0$ then 
 $ \omega_{\lambda}(t) \ = \ t^{N \sum_{\tau : F \to \C} d^\tau} \ = \ t^{\sum_\tau \sum_i b_i^\tau}.$
 The identification of the complexes gives an identification between our basic object of interest over $\C$ with the relative Lie algebra cohomology of the space of smooth automorphic forms twisted by the coefficient system: 
 $$
H^\bullet(\SGK, \tM_{\lambda,\C}) \ = \ 
H^\bullet(\fg_\infty, \fk_\infty; \cC^\infty(G(\Q)\backslash G(\A)/K_f, \omega_{\lambda}^{-1}|_{S(\R)^0}) \otimes  \M_{\lambda, \C}).
$$
The inclusion $\cC_{\rm cusp}^\infty(G(\Q)\backslash G(\A)/K_f, \omega_\infty^{-1}) \subset   
  \cC^\infty(G(\Q)\backslash G(\A)/K_f, \omega_\infty^{-1}),$ 
of the space of smooth cusp forms,  
induces an inclusion in relative Lie algebra cohomology (due to Borel \cite{borel-duke}), and cuspidal cohomology is defined as: 
$$
 H_{\rm cusp}^\bullet(\SGK, \tM_{\lambda, \C}) \ := \ 
H^\bullet \left(\fg_\infty, \fk_\infty;  \, 
\cC^\infty_{\rm cusp}(G(\Q)\backslash G(\A)/K_f, \omega_\infty^{-1})  \otimes \M_\lambda \right).
$$
Furthermore, 
$H_{\rm cusp}^\bullet(\SGK, \M_{\lambda, \C}) \subset H_{!}^\bullet(\SGK, \M_{\lambda, \C}).$
Define $\Coh_{\rm cusp}(G,\lambda,K_f)$ as the set of all $\pi_f \in \Coh_!(G, \lambda,K_f)$ 
which contribute to cuspidal cohomology.  The decomposition of cuspforms into cuspidal automorphic representations, 
gives the following fundamental decomposition for cuspidal cohomology: 
\begin{equation}
\label{eqn:cuspidal-coh-spectrum}
 H_{\rm cusp}^\bullet(\SGK, \tM_{\lambda, \C}) \ := \ 
 \bigoplus_{\pi \in \Coh_{\rm cusp}(G,\lambda,K_f)}  
 H^\bullet(\fg_\infty, \fk_\infty; \pi_\infty \otimes \M_{\lambda, \C}) \otimes \pi_f.
\end{equation}
To clarify a slight abuse of notation: if a cuspidal automorphic representation $\pi$ contributes to the above decomposition, then its representation at infinity is $\pi_\infty$ 
(which admits an explicit description that will be crucial for all the archimedean calculations), and $\pi_f$ denotes the $K_f$-invariants of its finite part. The level structure 
$K_f$ will be clear from context, hence whether $\pi_f$ denotes the finite-part or its $K_f$-invariants will be clear from context.  
Define ${\mathrm{Coh}}_{\rm cusp}(G, \lambda) \ = \ \bigcup_{K_f} {\mathrm{Coh}}_{\rm cusp}(G, K_f, \lambda).$

\medskip
\subsection{Pure weights and strongly-pure weights} 
\label{sec:pure} 
  
\medskip
\subsubsection{\bf Strongly-pure weights over $\C$}
If a weight $\lambda = (\lambda^\eta)_{\eta:F \to \C} \in X^+_{\rm alg}(T \times \C)$ supports cuspidal cohomology, i.e., if $H_{\rm cusp}^\bullet(\SGK, \tM_{\lambda, \C}) \neq 0$, then 
$\lambda$ satisfies the purity condition: 
\begin{equation}
\label{eqn:purity-def}
a_i^\eta = a_{N-i}^{\bar\eta} \ \mbox{for all $\eta : F \to \C$}\ \iff \ \mbox{$\exists \,\w$ such that $b^\eta_i + b^{\bar\eta}_{N-i+1} = \w$ for all $\eta$ and $i$,}
\end{equation}
which follows from the purity lemma \cite[Lem.\,4.9]{clozel}. The integer $\w$ is called the {\it purity weight} of $\lambda.$ The weight $\lambda$ is said to be {\it pure} if it 
satisfies \eqref{eqn:purity-def}, and denote by $X^+_{0}(T \times \C)$ the set of all such pure weights. 
Next, recall a theorem of Clozel that says that cuspidal cohomology for $\GL_N/F$ admits a rational structure 
\cite[Thm.\,3.19]{clozel}, from which it follows that any $\varsigma \in \Aut(\C)$ stabilizes cuspidal cohomology, i.e., ${}^\varsigma\lambda$ also 
satisfies the above purity condition, where if $\lambda = (\lambda^\eta)_{\eta: F \to \C}$ and $\varsigma \in \Aut(\C)$, then 
${}^\varsigma\lambda$ is the weight $({}^\sigma\lambda^\eta)_{\eta: F \to \C}$ where ${}^\varsigma\lambda^\eta = \lambda^{\varsigma^{-1}\circ \eta}$. 
A pure weight $\lambda$ will be called {\it strongly-pure}, if ${}^\varsigma\lambda$ is pure with purity-weight $\w$ for every $\varsigma \in \Aut(\C);$ denote by 
$X^+_{00}(T \times \C)$ the set of all such strongly-pure weights. 
For $\lambda \in X^+_{00}(T \times \C),$ note that 
$$
b^{\varsigma^{-1} \circ \eta}_j + b^{\varsigma^{-1} \circ \overline{\eta}}_{N-j+1} = \w, \ \mbox{for all \ $1 \leq j \leq N, \ \eta : F \to \C, \ \varsigma \in \Aut(\C)$}. 
$$ 
We have the following inclusions  
inside the character group of $T \times \C,$ which are all, in general, strict inclusions: 
 $$
 X^+_{00}(T \times \C) \ \subset \  X^+_{0}(T \times \C) \ \subset \ X^+_{\rm alg}(T \times \C) \ \subset \ X^+(T \times \C) \ \subset \ X^*(T \times \C).
 $$

\medskip
\subsubsection{\bf Strongly-pure weights over $E$}
\label{sec:strongly-pure-E} 
 
The set of strongly-pure weights may be defined at an arithmetic level. Recall the standing assumption on $E$ that is a finite Galois extension of $\Q$ that 
takes a copy of $F$; in particular, any embedding $\iota: E \to \C$ factors as $\iota: E \to \bar\Q \subset \C.$ Furthermore, $\iota : E \to \C$ gives 
a bijection $\iota_* : \Hom(F,E) \to \Hom(F,\C)$ as $\iota_*(\tau) = \iota \circ \tau$, which in turn gives a 
bijection $X^*(T \times E) \to X^*(T \times \C)$ that maps $\lambda = (\lambda^\tau)_{\tau : F \to E}$ to 
${}^\iota\lambda = ({}^\iota \lambda^\eta)_{\eta : F \to \C} = (\lambda^{\iota^{-1} \circ \eta})_{\eta : F \to \C}.$

\begin{prop}\label{prop:strong-pure-E}
Let $\lambda \in X^+_{\rm alg}(T \times E)$ be an algebraic dominant integral weight. Suppose 
$\lambda = (\lambda^\tau)_{\tau : F \to E}$ with $\lambda^\tau = (b^\tau_1 \geq \cdots \geq b^\tau_N)$. Then, the following are equivalent: 

\smallskip
\begin{enumerate}
\item[(i)] There exists $\iota : E \to \C$ such that ${}^\iota\lambda \in X^+_{00}(T \times \C)$, i.e., for every $\gamma \in \Gal(\bar\Q/\Q)$ we have 
${}^{\gamma \circ \iota}\lambda \in X^+_{0}(T \times \C)$ with the same purity weight: 
\begin{multline*}
``\exists \, \iota : E \to \C, \ \exists \, \w \in \Z \ \ \mbox{such that} \ \ 
b_j^{\iota^{-1}\circ \gamma^{-1} \circ \eta} + b_{N-j+1}^{\iota^{-1}\circ \gamma^{-1} \circ \bar\eta} \ = \ \w, \\ 
\forall \gamma \in \Gal(\bar\Q/\Q), \ \forall \eta : F \to \C, \ 1 \leq  j \leq N."
\end{multline*}

\smallskip
\item[(ii)] For every $\iota : E \to \C$, ${}^\iota\lambda \in X^+_{00}(T \times \C)$, i.e., for every $\gamma \in \Gal(\bar\Q/\Q)$ we have 
${}^{\gamma \circ \iota}\lambda \in X^+_{00}(T \times \C)$ with the same purity weight: 
\begin{multline*}
``\exists\, \w \in \Z \ \ \mbox{such that} \ \ 
b_j^{\iota^{-1}\circ \gamma^{-1} \circ \eta} + b_{N-j+1}^{\iota^{-1}\circ \gamma^{-1} \circ \bar\eta} \ = \ \w, \\  
\forall \, \iota : E \to \C, \ \forall\, \gamma \in \Gal(\bar\Q/\Q), \ \forall\, \eta : F \to \C, \ 1 \leq  j \leq N."
\end{multline*}

\smallskip
\item[(iii)] For every $\iota : E \to \C$, ${}^\iota\lambda \in X^+_{0}(T \times \C)$ with the same purity weight: 
$$
``\exists\, \w \in \Z \ \ \mbox{such that} \ \ 
b_j^{\iota^{-1}\circ \eta} + b_{N-j+1}^{\iota^{-1}\circ \bar\eta} \ = \ \w, \ 
\forall \, \iota : E \to \C,  \ \forall\, \eta : F \to \C, \ 1 \leq  j \leq N."
$$
\end{enumerate}
\end{prop}

\begin{proof}
Fix $\iota_0 : E \to \C$. Since $E/\Q$ is a finite Galois extension, the inclusions
$$
\{\gamma \circ \iota_0 \ | \ \gamma \in \Gal(\bar\Q/\Q)\} \ \subset \ 
\{\gamma \circ \iota \ | \ \gamma \in \Gal(\bar\Q/\Q), \iota : E \to \C\} \ \subset \ 
\Hom(E,\C)
$$
are all equalities. 
\end{proof}

\medskip 
The set of strongly-pure weights over $E$, denoted $X^+_{00}(T \times E)$ consists of the algebraic dominant integral weights 
$\lambda \in X^*(T \times E)$ that satisfy any one, and hence all, of the
conditions in the above proposition. It is most convenient to work with the characterization in $(iii)$.
There are the following inclusions within the character group of $T \times E$, which are all, in general, strict inclusions: 
 $$
 X^+_{00}(T \times E) \ \subset \ X^+_{\rm alg}(T \times E) \ \subset \ X^+(T \times E) \ \subset \ X^*(T \times E).
 $$

\medskip

The existence of a strongly-pure weight over a totally imaginary base field $F$ depends on the internal structure of $F$; this is explained over the course of the next four paragraphs.

\medskip
\subsubsection{\bf Interlude on (strongly-)pure weights for a CM field}
\label{sec:interlude-pure-1}
When the base field $F$ is a CM field, then a pure weight is also strongly-pure. 
Given any $\varsigma \in \Aut(\C)$, one can check that 
$\varsigma_*(X^+_{0}(T \times \C)) \ = \ X^+_{0}(T \times \C).$

\begin{lemma}
\label{sec:CM-field-embeddings}
Let $\eta : F \to \C$ and $\varsigma : \C \to \C$ be field homomorphisms, and let $\c : \C \to \C$ stand for complex conjugation. Then 
$$
\varsigma \circ \c \circ \eta \ = \  \c \circ \varsigma \circ \eta,
$$
i.e., complex conjugation and any automorphism of $\C$ commute on the image of a CM field. 
\end{lemma}  
 
 \begin{proof}
 Let $\eta_1 = \varsigma \circ \c \circ \eta$ and $\eta_2 = \c \circ \varsigma \circ \eta.$ Then $\eta_1|_{F^+} = \eta_2|_{F^+}$ (recall that $F^+$ is the maximal totally real subfield of $F$). 
 This means that $\eta_1 = \eta_2$ or $\eta_1 = \c \circ \eta_2$; if the latter, then $\varsigma \circ \c \circ \eta =  \varsigma \circ \eta$; evaluate both sides on $x \in F - F^+$ on which $\c(\eta(x)) = -\eta(x)$ to get a contradiction. 
 \end{proof}
 
 Let $\lambda = (\lambda^\eta)_{\eta : F \to \C} \in X^+_{0}(T \times \C)$, hence $d^\eta + d^{\bar\eta} = \w$ for all $\eta : F \to \C.$ Take any $\varsigma \in \Aut(\C)$ and consider 
 ${}^\varsigma\lambda$; to see its purity note that
 $$
 b_{j}^{\varsigma^{-1} \circ \eta} + b_{N-j+1}^{\varsigma^{-1} \circ \bar\eta} \ = \ 
 b_{j}^{\varsigma^{-1} \circ \eta} + b_{N-j+1}^{\varsigma^{-1} \circ \c \circ \eta} \ = \
 b_{j}^{\varsigma^{-1} \circ \eta} + b_{N-j+1}^{\c \circ \varsigma^{-1} \circ \eta} \ = \ 
 b_{j}^{\varsigma^{-1} \circ \eta} + b_{N-j+1}^{\overline{\varsigma^{-1} \circ \eta}} = \w, 
 $$
 where the second equality is from Lem.\,\ref{sec:CM-field-embeddings} above. Hence, $\lambda$ is strongly-pure.

\medskip
\subsubsection{\bf Interlude on strongly-pure weights for a general totally imaginary field}
\label{sec:interlude-pure-2}
 When $F$ is totally imaginary but not CM, there may exist weights that are pure but not strongly-pure. The following example is instructive: 
 take $F = \Q(2^{1/3},\omega),$ where $2^{1/3}$ is the real cube root of $2$ and $\omega = e^{2\pi i/3}.$ Then $\Sigma_F = \Gal(F/\Q) \simeq S_3$ the permutation group in 3 letters taken to be $\{2^{1/3}, \, 2^{1/3}\omega, \, 2^{1/3}\omega^2\}.$ Let $s \in S_3$ correspond to $\eta_s : F \to \C.$ Consider the weights
$\lambda = (\lambda^{\eta_s})_{s \in S_3}$ and $\mu = (\mu^{\eta_s})_{s \in S_3}$ for $\Res_{F/\Q}(\GL(1)/F)$ described in the table: 
$$
\begin{array}{|c||c|c|c|c|c|c||} 
\hline
s & \ e \ & (12) & (23) & (13) & (123) & (132) \\ \hline \hline
\lambda^{\eta_s} & a & b & \w-a & c & \w-c & \w-b \\ \hline 
\mu^{\eta_s} & a & \w-a & \w-a & \w-a & a & a \\ \hline
\end{array}
$$
where $a,b,c,\w \in \Z$. 
For the tautological embedding $F \subset \C,$ the set $\Sigma_F$ is paired into complex conjugates as 
$\{(\eta_{e}, \eta_{(23)}), (\eta_{(12)}, \eta_{(132)}), (\eta_{(13)}, \eta_{(123)})\},$
from which it follows that $\lambda$ is a pure weight. All other possible pairings of $\Sigma_F$ into conjugates via automorphisms of $\C$ 
are given by: 
$\{(\eta_{e}, \eta_{(12)}), (\eta_{(23)}, \eta_{(123)}), (\eta_{(13)}, \eta_{(132)})\},$ and  
$\{(\eta_{e}, \eta_{(13)}), (\eta_{(23)}, \eta_{(132)}), (\eta_{(12)}, \eta_{(123)})\};$ 
($F$ being Galois this simply boils down to composing these embeddings $\eta_s$ by a fixed one $\eta_{s_0},$ and using 
$\eta_{s_0} \circ \eta_{s} = \eta_{s_0s}$).  It follows that $\lambda$ is not strongly-pure if $\w-a, b$ and $c$ are not all equal, 
but $\mu$ is strongly-pure and has purity weight $\w.$

\medskip
\subsubsection{\bf On the internal structure of a general totally imaginary field}
\label{sec:internal-structure-II}

Let $F$ be a totally imaginary field as before. Let $F_0$ be the largest totally real subfield in $F$. Then there is at most one totally imaginary quadratic extension $F_1$ of $F_0$ inside $F$. 
(See, for example, Weil \cite{weil}.) If $\alpha$ and $\beta$ are two totally negative elements of $F_0$ giving two possible such extensions $F_0(\sqrt{\alpha})$ and $F_0(\sqrt{\beta})$, then by 
maximality of $F_0$ one has $\sqrt{\alpha \beta} \in F_0$, i.e., $\alpha = t^2 \beta$ for $t \in F_0$, whence $F_0(\sqrt{\alpha}) = F_0(\sqrt{\beta}).$ There are two distinct cases to consider: 
\begin{enumerate}
\smallskip
\item[(i)] {\bf CM}: when there is indeed such an imaginary quadratic extension $F_1$ of $F_0$; then $F_1$ is the maximal CM subfield of $F$; of course $[F_1:F_0] = 2$.   
For example, if $F = \Q(2^{1/3},\omega)$ as in \ref{sec:interlude-pure-2} then $F_0 = \Q$ and $F_1 = \Q(\omega)$. 
\smallskip
\item[(ii)] {\bf TR}: when there is no imaginary quadratic extension of $F_0$ inside $F$, then put $F_1 = F_0$ for the maximal totally real subfield of $F$. 
For example, take $F_0$ to be a cubic totally real field 
(e.g., $F_0 = \Q(\zeta_7+\zeta_7^{-1})$, $\zeta_7 = e^{2\pi i/7}$), and take non-square elements $a,b \in F_0$ whose conjugates $a,a',a''$ and $b,b',b''$ are such that 
$a>0, \ a'<0, a''<0$ and $b<0, \ b'<0,\ b'' >0$; such $a$ and $b$ exist by weak-approximation; take $F = F_0(\sqrt{a}, \sqrt{b})$. Then there is no intermediate CM-subfield 
between $F_0$ and $F$; hence $F_1 = F_0.$ 
\end{enumerate}

\medskip 
As will be explained later on, that in case {\bf TR}, asking for a critical point for a Rankin--Selberg $L$-function for $\GL(n) \times \GL(n')/F$, will impose the restriction  $nn'$ is even. 
This should not be surprising, because, as is well-known, for an algebraic Hecke character $\chi$ over $F$, if the $L$-function $L(s,\chi)$ has critical points then that forces us to be 
in case {\bf CM} (see \cite{raghuram-hecke}). %And if $F$ is in the {\bf CM}-case then there are richer details involving arithmetic of the special values of $L$-functions. 

\medskip
\paragraph{\bf Notation in the {\bf CM}-case.} 
\label{sec:notation-CM}
Suppose $\place_\infty(F)$ (resp., $\place_\infty(F_1)$) is the set of archimedean places of $F$ (resp., $F_1$).
Enumerate $\place_\infty(F_1)$ as $\{w_1,\dots, w_{\sf r_1}\}$, where ${\sf r}_1 = d_{F_1}/2 = [F_1:\Q]/2.$  
For $ 1 \leq j \leq {\sf r}_1$, let $\{\nu_j, \bar\nu_j\} \subset \Sigma_{F_1}$ be the pair of conjugate 
embeddings corresponding to $w_j,$ the non-canonical choice of $\nu_j$ is fixed, and is distinguished in the sense that $\nu_j$ induces the isomorphism $F_{1,w_j} \simeq \C.$ 
Let $k = [F:F_1].$ Let $v_{j1},\dots,v_{jk}$ be the set of 
places in $\place_\infty(F)$ above $w_j.$ Let $\varrho : \Sigma_F \to \Sigma_{F_1}$ denote the restriction map $\varrho(\eta) = \eta|_{F_1}.$ Suppose 
$\varrho^{-1}(\nu_j) = \{\eta_{j1},\dots, \eta_{jk}\}$, then $\varrho^{-1}(\bar\nu_j) = \{\bar\eta_{j1},\dots, \bar\eta_{jk}\},$ with the indexing being such that the pair of conjugate 
embeddings $\{\eta_{jl}, \bar\eta_{jl}\}$ corresponds to $v_{jl} \in \place_\infty(F)$ for all $1 \leq j \leq {\sf r}_1$ and $1 \leq l \leq k.$

\medskip
\paragraph{\bf Notation in the {\bf TR}-case.} 
\label{sec:notation-TR}
Let $\place_\infty(F_1) = \{w_1,\dots, w_{d_{F_1}}\}$ be an enumeration of the set of archimedean places of $F_1$, 
where $d_{F_1} = [F_1:\Q];$ since $F_1$ is the maximal totally real subfield of the totally imaginary $F$, the degree  
$d_{F_1}$ can be either even or odd, but the index $k = [F:F_1]$ is even; suppose $k = 2k_1.$ 
For $ 1 \leq j \leq d_{F_1}$, let $\nu_j \in  \Sigma_{F_1}$ be 
complex embedding corresponding to $w_j.$ 
As before, $\varrho : \Sigma_F \to \Sigma_{F_1}$ denote the restriction map $\varrho(\eta) = \eta|_{F_1}.$ 
Let $v_{j1},\dots,v_{jk_1}$ be the set of 
places in $\place_\infty(F)$ above $w_j$, and suppose $v_{ji}$ corresponds to the pair of conjugate embeddings $\{\eta_{ji}, \bar\eta_{ji}\}$, then 
$\varrho^{-1}(\nu_j) = \{\eta_{j1}, \bar\eta_{j1}, \eta_{j2}, \bar\eta_{j2}, \dots, \eta_{jk_1}, \bar\eta_{jk_1}\}$. 
 
\medskip
\subsubsection{\bf Strongly-pure weights over $F$ are base-change from $F_1$}
\label{sec:base-change-of-weights}

\begin{prop}
\label{prop:strong-pure-weights-base-change}
Suppose $\lambda \in X^+_{00}(\Res_{F/\Q}(T_{N,0}) \times E)$ is a strongly-pure weight. Then there exists $\kappa \in X^+_{00}(\Res_{F_1/\Q}(T_{N,0}) \times E)$ such that 
$\lambda$ is the base-change of $\kappa$ from $F_1$ to $F$ in the sense that for any $\tau : F \to E$, $\lambda^\tau = \kappa^{\tau|_{F_1}}.$
\end{prop}

For brevity, the conclusion will denoted as $\lambda = {\rm BC}_{F/F_1}(\kappa).$

\begin{proof}
It suffices to prove the proposition over $\C$, i.e., if $'\lambda \in X^+_{00}(\Res_{F/\Q}(T_{N,0}) \times \C)$ then it suffices to show the existence 
$'\kappa \in X^+_{00}(\Res_{F_1/\Q}(T_{N,0}) \times \C)$ such that $'\lambda = {\rm BC}_{F/F_1}('\kappa);$ because, then given the $\lambda$ in the proposition,
take an embedding $\iota : E \to \C$, and let $'\lambda = {}^\iota\lambda$, to which using the statement over $\C$ one gets $'\kappa$, which defines a unique $\kappa$ 
via ${}^\iota\kappa =  {}'\kappa.$ It is clear that $\lambda = {\rm BC}_{F/F_1}(\kappa)$ because this is so after applying $\iota$. 

\smallskip
To prove the statement over $\C$, take $\lambda \in X^+_{00}(\Res_{F/\Q}(T_{N,0}) \times \C)$, and suppose $\lambda = (\lambda^\eta)_{\eta : F \to \C}$ with 
$\lambda^\eta = (b^\eta_1 \geq b^\eta_2 \geq \cdots \geq b^\eta_N).$ Strong-purity gives 
$$
b^{\gamma \circ \eta}_{N-j+1} + b^{\gamma \circ \bar{\eta}}_j = \w, \quad \forall \gamma \in \Gal(\bar\Q/\Q), \ \forall \eta \in \Sigma_F, \ 1 \leq j \leq N.
$$
Also, one has: 
$$
b^{\gamma \circ \eta}_{N-j+1} + b^{\overline{\gamma \circ \eta}}_j = \w, \quad \forall \gamma \in \Gal(\bar\Q/\Q), \ \forall \eta \in \Sigma_F, \ 1 \leq j \leq N.
$$
Hence, we get 
$b^{\gamma \circ \bar{\eta}}_j \ = \ b^{\overline{\gamma \circ \eta}}_j.$
Exactly as explicated in the proof of Prop.\,26 in \cite{raghuram-hecke}, one gets $b^{\gamma \circ \eta}_j = b^{\eta}_j$ for all 
$\gamma$ in the normal subgroup of $\Gal(\bar\Q/\Q)$ generated by the commutators $\{g \c g^{-1} \c : g \in \Gal(\bar\Q/\Q)\}$, and all $\eta : F \to \C.$ 
This means that $b^{\eta}_j$ depends 
only on $\eta|_{F_1}.$ 
\end{proof}

\medskip
\subsection{Strongly inner cohomology} 
\label{sec:strong-inner-coh}
The problem of giving an arithmetic characterization of cuspidal cohomology is addressed in \cite[Chap.\,5]{harder-raghuram-book} in great detail for 
$\GL_N$ over a totally real field. In this article, for $\GL_N$ over a totally imaginary $F$, we will only discuss it {\it en passant}, and contend ourselves in making the following  
\begin{defn}
\label{def:strongly-inner}
Take a field $E$ large enough (as before), and let $\lambda \in X^+_{00}(T \times E).$ The strongly inner spectrum of $\lambda$ for level structure $K_f$ is defined as:
\begin{multline*}
\Coh_{!!}(G, K_f, \lambda) \ = \\ 
\{\pi_f \in \Coh_!(G, K_f, \lambda) \ : \ \mbox{${}^\iota\pi_f \in \Coh_{\rm cusp}(G, K_f, {}^\iota\lambda)$ for some embedding $\iota : E \to \C$}\}.
\end{multline*}
\end{defn}
An irreducible Hecke-summand $\pi_f$ in inner cohomology, is strongly-inner if under some embedding $\iota$ rendering the context transcendental it contributes to cuspidal cohomology. The point of view in {\it loc.\,cit.} is that the definition 
is independent of $\iota$, and hence giving a rational origin (i.e., over $E$) to cuspidal summands giving another proof of a result of Clozel that cuspidal cohomology for $\GL_N$ admits a rational structure 
\cite[Thm.\,3.19]{clozel}. In this article one simply appeals to Clozel's theorem to observe that the definition of strongly inner spectrum is independent of the choice of embedding $\iota$, i.e., if $\iota, \iota' : E \to \C$ are two such embeddings then 
$$
{}^\iota\pi_f \in \Coh_{\rm cusp}(G, {}^\iota\lambda) \ \iff \ {}^{\iota'}\!\pi_f \in \Coh_{\rm cusp}(G, {}^{\iota'}\!\lambda). 
$$
Define strongly-inner cohomology as: 
$$
H_{!!}^\bullet(\SGK, \tM_{\lambda, E}) \ = \ 
\bigoplus_{\pi_f \in \Coh_{!!}(G,\lambda,K_f)}  H_{!}^\bullet(\SGK, \tM_{\lambda, E})(\pi_f). 
$$
Then, since cuspidal cohomology is contained in inner cohomology, it is clear that 
$$
H_{!!}^\bullet(\SGK, \tM_{\lambda, E}) \otimes_{E,\iota} \C \ \simeq \ 
H_{\rm cusp}^\bullet(\SGK, \tM_{{}^\iota\!\lambda, \C}). 
$$

For $\lambda \in X^+_{00}(T \times E),$ $\pi_f \in \Coh_{!!}(G, \lambda)$ (ignoring the level structure) and $\iota : E \to \C$, since ${}^\iota\pi_f \in \Coh_{\rm cusp}(G, {}^\iota\lambda)$, let ${}^\iota\pi$ stand for the unique global cuspidal automorphic representation of $G(\A_\Q) = \GL_N(\A_F)$ 
whose finite part is $({}^\iota\pi)_f = {}^\iota\pi_f.$ 
The representation at infinity 
$({}^\iota\pi)_\infty,$ to be denoted $\J_{{}^\iota\!\lambda }$ below, will be explicitly described in Sect.\,\ref{sec:arch-consideration}.

\medskip
\subsubsection{\bf Tate twists} 
\label{sec:tate-twists}

Let $m \in \Z.$ For $\lambda \in X^+(T_N \times E)$, define $\lambda + m \bfgreek{delta}_N$ by the rule that if 
$\lambda^\tau = (b_1^\tau,\dots, b_N^\tau)$, then $(\lambda + m \bfgreek{delta}_N)^\tau = (b_1^\tau +m,\dots, b_N^\tau+m)$ 
for $\tau : F \to E.$ It is  clear that  
$\lambda \mapsto \lambda + m \bfgreek{delta}_N$  preserves each of the properties: dominant, integral, algebraic and (strongly-)pure. 
As in \cite[Sect.\,5.2.4]{harder-raghuram-book}, cupping with 
the fundamental class $e_{\bfgreek{delta}_N} \in H^0(\SGK, \Q[\bfgreek{delta}_N])$ gives us an isomorphism 
$
T^\bullet_{\mathrm{Tate}}(m) : H^\bullet(\SGK, \tM_\lambda) \ \to \ H^\bullet(\SGK, \tM_{\lambda + m \bfgreek{delta}_N}), 
$
that maps (strongly-)inner cohomology to (strongly-)inner cohomology, and suppose $\pi_f \in \Coh_{!!}(G,\lambda)$, then 
$T^\bullet_{\mathrm{Tate}}(m)(\pi_f) \ = \ \pi_f(-m),$
where, $\pi_f(-m)$ is defined by $\pi_f(-m)(g_f) = \pi_f(g_f) \otimes |\!|g_f|\!|^{-m}.$

\medskip
\subsection{Archimedean considerations} 
\label{sec:arch-consideration}

\medskip
\subsubsection{\bf Cuspidal parameters and cohomological representations of $\GL_N(\C)$}
\label{sec:cuspidal-parameters}

Given a weight $\lambda = (\lambda^\eta)_{\eta :F \to \C} \in X^+_{00}(T \times \C),$ for each $v \in \place_\infty$ (recall that $v$ corresponds to a pair of complex embeddings $\{\eta_v, \bar\eta_v\}$ of $F$ into $\C$, with $\eta_v$ used to identify $F_v$ with $\C$), define the 
{\it cuspidal parameters} of $\lambda$ at $v$ by:
\begin{equation*}
\alpha^v \ := \ -w_0\lambda^{\eta_v} + \rho \ \ {\rm and} \ \ 
\beta^v \ := \ - \lambda^{\bar\eta_v} - \rho. 
\end{equation*} 
If $\lambda^\eta = (b^\eta_1,\dots,b^\eta_N)$ then 
\begin{equation}
\label{eqn:cuspidal-parameters-alpha}
\alpha^v \ = \ (\alpha^v_1, \dots, \alpha^v_N)
 \ = \ \left(-b^{\eta_v}_N+ \tfrac{(N-1)}{2}, \, -b^{\eta_v}_{N-1}+ \tfrac{(N-3)}{2}, \, \dots, \, -b^{\eta_v}_1 - \tfrac{(N-1)}{2}\right),  
\end{equation} 
and, similarly,  
\begin{equation}
\label{eqn:cuspidal-parameters-beta}
\beta^v \ := \ (\beta^v_1, \dots, \beta^v_N)
 \ = \ \left(-b^{\bar\eta_v}_1 - \tfrac{(N-1)}{2}, \, -b^{\bar\eta_v}_2 - \tfrac{(N-3)}{2}, \, \dots, \, -b^{\bar\eta_v}_N + \tfrac{(N-1)}{2}\right). 
\end{equation} 
Purity  implies that $\alpha^v_j + \beta^v_j = -\w.$ Define a representation of $\GL_N(F_v) \simeq \GL_N(\C)$ as: 
\begin{equation}
\label{eqn:j-lambda-v}
\J_{\lambda_v} \ := \ \J(\lambda^{\eta_v}, \lambda^{\bar\eta_v}) \ := \ 
\Ind_{B_N(\C)}^{\GL_N(\C)}
\left(z^{\alpha^v_1} \bar{z}^{\beta^v_1} \otimes \cdots \otimes z^{\alpha^v_N} \bar{z}^{\beta^v_N} \right),  
\end{equation}
where $B_N$ is the subgroup of all upper-triangular matrices in $\GL_N,$ and by $\Ind$ we mean normalized (i.e., unitary) parabolic induction.  Now define a representation of $G(\R) = \prod_v \GL_N(F_v)$: 
\begin{equation}
\label{eqn:j-lambda}
\J_{\lambda} \ := \ \bigotimes_{v \in \place_\infty} \J_{\lambda_v}. 
\end{equation}

\medskip

\begin{rem}\label{rem:interchange-alpha-beta}
{\rm 
Recall that the choice of the embedding $\eta_v$ in the pair $\{\eta_v, \bar\eta_v\}$ was fixed. If the roles of the $\eta_v$ and $\bar\eta_v$ are reversed, then it is easy to see 
that the pair $(\alpha^v, \beta^v)$ of cuspidal parameters would be replaced by $(w_0\beta^v, w_0\alpha^v)$, whence, the representation 
$\J_v$ is replaced by its conjugate $\bar\J_v$. See \ref{sec:arch-constituents} below. 
}\end{rem}

\medskip

Some basic properties of $\J_\lambda$ are described in the following two propositions. 

\begin{prop}
\label{prop:j-lambda}
Let $\lambda \in X^+_{00}(T \times \C)$  and $\J_\lambda$ as above.  Then:
\begin{enumerate}
\item $\J_\lambda$ is an irreducible essentially tempered representation admitting a Whittaker model. 
\smallskip
\item $H^\bullet(\g, K_\infty; \, \J_\lambda \otimes \M_{\lambda, \C}) \neq 0.$
\smallskip
\item Let $\J$ be an irreducible essentially tempered representation of $G(\R).$ \\ 
Suppose 
that $H^\bullet(\g, K_\infty; \J \otimes \M_{\lambda, \C}) \neq 0$ then $\J = \J_\lambda.$
\smallskip
\item If $\pi \in \Coh_{\rm cusp}(G,\lambda)$, i.e., $\pi$ is a global cuspidal automorphic representation that contributes to cuspidal cohomology 
with respect to a strongly-pure weight $\lambda$, then $\pi_\infty \cong \J_\lambda.$ 
\end{enumerate}
\end{prop}

These are well-known results for $\GL_N(\C)$ and are all easily seen from this elementary observation: for $z \in \C$, let $|z|_\C = z \bar{z}$; then the representation 
$$
\J_{\lambda_v} \otimes |\ |_\C^{\w/2} \ = \ 
\Ind_{B_N(\C)}^{\GL_N(\C)}
\left(z^{\alpha^v_1+ \w/2} \bar{z}^{\beta^v_1+\w/2} \, \otimes \, \cdots \, \otimes \, z^{\alpha^v_N+\w/2} \bar{z}^{\beta^v_N+\w/2} \right)
$$
 is unitarily induced from unitary characters (because of purity) and hence irreducible. 
A representation irreducibly induced from essentially discrete series representation is essentially tempered. Admitting a Whittaker model is a hereditary property. Nonvanishing of cohomology follows from Delorme's Lemma
(Borel--Wallach \cite[Thm.\,III.3.3]{borel-wallach}), and that relative Lie algebra cohomology satisfies a K\"unneth theorem. Finally, among all representations with given infinitesimal character there is at most one that is essentially tempered. 

\medskip
Define the following numbers:
\begin{equation}
\label{eqn:bn-tn}
\begin{split}
b_N^\C & \ = \ N(N-1)/2, \\
t_N^\C & \ = \ (N^2-1) - b_N^\C, \\ 
b_N^F & \ = \ {\sf r} \cdot b_N^\C, \\
t_N^F & \ = \ \dim(\SG) - b_N^F. 
\end{split}
\end{equation}
The relation between $b_N^F$ and $t_N^F$ is mitigated by an appropriate version of Poincar\'e duality, which is the reason why the `top-degree' is defined in terms of the `bottom-degree' and the dimension of the intervening symmetric space.  

\begin{prop}
\label{prop:j-lambda-range-degrees}
Let $\lambda \in X^+_{00}(T \times \C)$  and $\J_\lambda$ as above.  Then  
$$
H^q(\g, K_\infty; \, \J_{\lambda} \otimes \M_{\lambda, \C}) \neq 0 \ \ \iff \ \ 
b_N^F \ \leq \ q\ \leq \ t_N^F. 
$$
Furthermore, for extremal degrees $q \in \{b_N^F, \, t_N^F\},$ we have $\dim(H^q(\g, K_\infty; \, \J_{\lambda} \otimes \M_{\lambda, \C})) = 1.$ 
\end{prop}

\begin{proof}
For each $v \in \place_\infty$, we have 
$H^q(\gl_N(\C), \U(N)Z_{N,0}(\R)^0; \, \J_{\lambda_v} \otimes \M_{\lambda_v, \C}) \neq 0$ if and only if % \ \ \iff \ \ 
$b_N^\C \leq q \leq t_N^\C.$
This follows, after a minor modification, from Clozel~\cite[Lemme 3.14]{clozel}. The cohomology is in fact an exterior algebra (up to shifting in degree 
by $b_N^\C$), giving one-dimensionality in bottom and top degree. Then use the fact that relative Lie algebra cohomology satisfies a K\"unneth theorem. This gives $(\g, C_\infty Z(\R)^0)$-cohomology from which the reader may easily deduce the above details for $(\g, C_\infty S(\R)^0) = (\g, K_\infty)$-cohomology; it is helpful to note that: 
$t_N^F \ = \ {\sf r} t_N^\C + ({\sf r}-1) \ = \ {\sf r} t_N^\C + \dim(Z(\R)^0/S(\R)^0).$
\end{proof}

There is a piquant numerological relation between the bottom or top degee for the cuspidal cohomology of Levi subgroup 
$\GL_n \times \GL_{n'}$ of a maximal parabolic subgroup $P$ of an ambient $\GL_N$, the corresponding bottom or top degree for 
$\GL_N$, and the dimension of the unipotent radical of $P$ given in the following proposition that has a crucial bearing 
on certain degree-considerations for Eisenstein cohomology. 
For any positive integer $r$, define $b_r^\C, \, t_r^\C, \, b_r^F$ and $t_r^F$ as in \eqref{eqn:bn-tn} replacing $N$ by $r$. 

\begin{prop}
\label{prop-numerology}
Let $n$ and $n'$ be positive integers with $n+n' = N.$ Then: 
\begin{enumerate}
\item $b_n^F + b_{n'}^F + \frac12 \dim(U_P) \ = \ b_N^F .$
\medskip
\item $t_n^F + t_{n'}^F + \frac12 \dim(U_P) \ = \ t_N^F - 1.$
\end{enumerate}
\end{prop}

\begin{proof}
Keeping in mind that $N = n+ n'$, (1) follows from the identity: 
$$
{\sf r} \frac{n(n-1)}{2} + {\sf r} \frac{n'(n'-1)}{2} + {\sf r} nn' \ = \ {\sf r} \frac{(n+n')(n+n'-1)}{2}. 
$$
For (2), observe that 
$t_n^F \ = \ (n^2{\sf r} - 1) - {\sf r}n(n-1)/2 = {\sf r}n(n+1)/2 - 1.$
Now (2) follows from: 
$$
\left({\sf r} \frac{n(n+1)}{2} -1\right) + \left({\sf r} \frac{n'(n'+1)}{2} - 1\right)+ {\sf r} nn' \ = \  \left({\sf r} \frac{(n+n')(n+n'+1)}{2} - 1\right) - 1.
$$
\end{proof}

\medskip
\subsubsection{\bf Archimedean constituents: {\bf CM}-case}
\label{sec:arch-constituents}

If $\pi_\infty = \otimes_{v \in \place_\infty} \pi_v$ is an irreducible 
representation of $G(\R) = \prod_{v \in \place_\infty} \GL_N(\C)$ then the 
set $\{\pi_v : v \in \place_\infty\}$ of the irreducible factors, up to equivalence, 
will be called as the set of constituents of $\pi_\infty$.  
Let $\lambda \in X^+_{00}(\Res_{F/\Q}(T_{N,0}) \times E)$, $\pi_f \in \Coh_{!!}(G,\lambda),$ and $\iota : E \to \C$. 
The set of constituents of ${}^\iota\pi_\infty$ may be explicitly described. 

\medskip
\paragraph{{\bf CM}-case}

Recall from Prop.\,\ref{prop:strong-pure-weights-base-change} that 
$\lambda = {\rm BC}_{F/F_1}(\kappa),$ i.e., $\lambda^\tau = \kappa^{\tau|_{F_1}};$ after applying $\iota$, one has
${}^\iota\lambda^\eta = {}^\iota\kappa^{\eta|_{F_1}},$ which is the same as 
$\lambda^{\iota^{-1}\circ\eta} = \kappa^{\iota^{-1}\circ\eta|_{F_1}}.$ 
 Using the notations fixed in \ref{sec:notation-CM}, 
for any place $v_{jl} \in \place_\infty(F)$ above $w_j \in \place_\infty(F_1),$ the ordered pair $(\eta_{v_{jl}}, \bar\eta_{v_{jl}})$
of conjugate embeddings of $F$ restricts to the ordered pair $(\nu_{w_j}, \bar\nu_{w_j})$ of conjugate embeddings of $F_1$; 
hence the ordered pair of weights 
$({}^\iota\lambda^{\eta_{v_{jl}}}, {}^\iota\lambda^{\bar\eta_{v_{jl}}})$ is equal to the ordered pair $({}^\iota\kappa^{\nu_{w_j}}, {}^\iota\kappa^{\bar\nu_{w_j}});$ whence the archimedean component 
${}^\iota\pi_{v_{jl}}$ is equivalent to $\J({}^\iota\kappa^{\nu_{w_j}}, {}^\iota\kappa^{\bar\nu_{w_j}})$. Just for the moment, for brevity, denoting $\J({}^\iota\kappa^{\nu_{w_j}}, {}^\iota\kappa^{\bar\nu_{w_j}})$ by $\J_{w_j}$, one concludes that the constituents of ${}^\iota\pi_\infty$ is the multi-set 
$\{\J_{w_1},\dots, \J_{w_1}, \  \dots \ ,\J_{w_{{\sf r}_1}}, \dots, \J_{w_{{\sf r}_1}} \},$
with each $\J_{w_j}$ appearing $k = [F:F_1]$-many times; this multi-set may also be variously written as
$\{[F:F_1] \cdot \J_{w} \ | \ w \in \place_\infty(F_1)  \} \ = \ \{[F:F_1] \cdot \J({}^\iota\kappa^{\nu_{w}}, {}^\iota\kappa^{\bar\nu_{w}}) \ | \ w \in \place_\infty(F_1)  \}.$ Putting these together one has
\begin{multline}
\label{eqn:arch-components-CM}
{}^\iota\pi_\infty \ = \ \bigotimes_{v \in \place_\infty(F)} {}^\iota\pi_v \ = \
\bigotimes_{v \in \place_\infty(F)} \J({}^\iota\lambda^{\eta_{v}}, {}^\iota\lambda^{\bar\eta_{v}})  \\ 
 = \ \bigotimes_{w \in \place_\infty(F_1)} \bigotimes_{v|w} 
 \J({}^\iota\kappa^{\nu_{w}}, {}^\iota\kappa^{\bar\nu_{w}})   \ = \ 
\bigotimes_{w \in \place_\infty(F_1)} \bigotimes_{v|w} 
 \J(\kappa^{\iota^{-1} \circ \nu_{w}}, \kappa^{\iota^{-1} \circ \bar\nu_{w}}). 
\end{multline}

\medskip
\paragraph{{\bf TR}-case}

We still have from Prop.\,\ref{prop:strong-pure-weights-base-change} that 
$\lambda = {\rm BC}_{F/F_1}(\kappa),$ i.e., $\lambda^\tau = \kappa^{\tau|_{F_1}};$ after applying $\iota$, one has
${}^\iota\lambda^\eta = {}^\iota\kappa^{\eta|_{F_1}},$ which is the same as 
$\lambda^{\iota^{-1}\circ\eta} = \kappa^{\iota^{-1}\circ\eta|_{F_1}}.$ 
 Using the notations fixed in \ref{sec:notation-TR}, 
for any place $v_{jl} \in \place_\infty(F)$ above $w_j \in \place_\infty(F_1),$ both the embeddings in the ordered pair $(\eta_{v_{jl}}, \bar\eta_{v_{jl}})$
restrict to $\nu_{w_j}$. Hence the ordered pair of weights 
$({}^\iota\lambda^{\eta_{v_{jl}}}, {}^\iota\lambda^{\bar\eta_{v_{jl}}})$ is equal to the ordered pair $({}^\iota\kappa^{\nu_{w_j}}, {}^\iota\kappa^{\nu_{w_j}})$--note that both weights in 
the ordered pair are the same;  whence the archimedean component 
${}^\iota\pi_{v_{jl}}$ is equivalent to $\J({}^\iota\kappa^{\nu_{w_j}}, {}^\iota\kappa^{\nu_{w_j}})$. Once again, for brevity, denoting 
$\J({}^\iota\kappa^{\nu_{w_j}}, {}^\iota\kappa^{\nu_{w_j}})$ by $\J_{w_j}$, one concludes that the constituents of ${}^\iota\pi_\infty$ is the multi-set 
$\{\J_{w_1},\dots, \J_{w_1}, \  \dots \ ,\J_{w_{{\sf r}_1}}, \dots, \J_{w_{{\sf r}_1}} \},$
with each $\J_{w_j}$ appearing $k_1 = [F:F_1]/2$-many times; 
putting these together one has
\begin{multline}
\label{eqn:arch-components-TR}
{}^\iota\pi_\infty \ = \ \bigotimes_{v \in \place_\infty(F)} {}^\iota\pi_v
 \ = \ \bigotimes_{v \in \place_\infty(F)} \J({}^\iota\lambda^{\eta_{v}}, {}^\iota\lambda^{\eta_{v}}) \\
  \  = \ \bigotimes_{w \in \place_\infty(F_1)} \bigotimes_{v|w} \J({}^\iota\kappa^{\nu_{w}}, {}^\iota\kappa^{\nu_{w}}) 
   \ = \ \bigotimes_{w \in \place_\infty(F_1)} \bigotimes_{v|w} 
 \J(\kappa^{\iota^{-1} \circ \nu_{w}}, \kappa^{\iota^{-1} \circ \nu_{w}}) 
 \ = \ \bigotimes_{w \in \place_\infty(F_1)} \bigotimes_{v|w} \J_w, 
\end{multline}
and each of these $\J_w$ is self-conjugate from Rem.\,\ref{rem:interchange-alpha-beta}.

\medskip
\subsubsection{\bf Galois action on archimedean constituents}
\label{sec:galois-arch-constituents}
Let $\gamma \in \Gal(\bar\Q/\Q).$
The archimedean constituents of ${}^{\gamma \circ \iota}\pi$ is a permutation of the archimedean constituents of ${}^\iota\pi$, possibly up to replacing a local component 
by its conjugate which will only be relevant when $F$ is in the {\bf CM}-case. 
This is made more precise in the following paragraphs. 

\medskip
\paragraph{\bf The case when $F$ is itself a CM field}
\label{sec:galois-action-CM-field}
If $F = F_1$ is a CM field, and $F_0$ its maximal totally real quadratic subfield, 
for $\gamma \in \Gal(\bar\Q/\Q)$ and $\nu \in \Sigma_{F_1}$ from Lem.\,\ref{sec:CM-field-embeddings} one has $\gamma \circ \bar{\nu} = \overline{ \gamma \circ \nu};$ 
this means that $\gamma$ permutes 
the set of pairs of conjugate embeddings $\left\{\{\nu_w, \bar\nu_w\} \ | \ w \in \place_\infty(F_1)\right\}$, giving an action of $\gamma$ on $\place_\infty(F_1)$; 
if we identify $\place_\infty(F_1) = \place_\infty(F_0) = \Sigma_{F_0}$ then the action of $\gamma$ on $\place_\infty(F_1)$ is the same as action of 
$\gamma$ on $\Sigma_{F_0}$ via composition. It is important to note that $\gamma$ need not map the distinguished embedding corresponding to $w$ to the distinguished embedding 
corresponding to $\gamma \cdot w$; all one can say is that $\gamma \circ \nu_w \in \{\nu_{\gamma \cdot w}, \bar{\nu}_{\gamma \cdot w}\}.$ 
Suppose $\kappa \in X^+_{00}(\Res_{F_1/\Q}(T_{N,0}) \times E)$, $\pi_{1,f} \in \Coh_{!!}(\Res_{F_1/\Q}(\GL_N/F_1,\kappa))$, $\iota : E \to \C$, 
and ${}^\iota\pi_1$ the corresponding cuspidal automorphic representation of $\GL_N(\A_{F_1})$, then 
${}^\iota\pi_{1,\infty} = \otimes_{w \in \place_\infty(F_1)} {}^\iota\pi_{1,w}$ where 
$$
{}^\iota\pi_{1,w} \ = \ \J({}^\iota\kappa^{\nu_w}, {}^\iota\kappa^{\bar\nu_w})  \ = \ \J(\kappa^{\iota^{-1} \circ \nu_w}, \kappa^{\iota^{-1} \circ \bar\nu_w}).
$$ 
By the same token, replacing $\iota$ by $\gamma \circ \iota$ one has
$$
{}^{\gamma \circ \iota} \pi_{1,w} \ = \ \J(\kappa^{\iota^{-1} \circ \gamma^{-1} \circ \nu_w}, \kappa^{\iota^{-1} \circ \gamma^{-1} \circ \bar\nu_w}) 
 \ = \ \J(\kappa^{\iota^{-1} \circ \gamma^{-1} \circ \nu_w}, \kappa^{\iota^{-1} \circ \overline{\gamma^{-1} \circ \nu_w}})
$$ 
Depending on whether $\gamma^{-1} \circ \nu_w = \nu_{\gamma^{-1} \cdot w}$ or $\overline{\nu_{\gamma^{-1} \cdot w}}$, from 
Rem.\,\ref{rem:interchange-alpha-beta} it follows that 
\begin{equation}\label{eqn:galois-CM}
{}^{\gamma \circ \iota} \pi_{1,w} \ = \ 
\begin{cases}
{}^\iota\pi_{1, \gamma^{-1} \cdot w} & \mbox{if $\gamma^{-1} \circ \nu_w = \nu_{\gamma^{-1} \cdot w}$,} \\ \\  
\overline{{}^\iota\pi_{1, \gamma^{-1} \cdot w}}  & \mbox{if $\gamma^{-1} \circ \nu_w = \overline{\nu_{\gamma^{-1} \cdot w}}.$} 
\end{cases}
\end{equation}
Hence the archimedean components of ${}^{\gamma \circ \iota} \pi_1$ is a permutation of the archimedean components of ${}^{\iota} \pi_1$ up to taking conjugates; this paragraph fixes a mistake in \cite[Prop.\,3.2, (i)]{gan-raghuram}.

\medskip
\paragraph{\bf When $F$ is totally imaginary in the {\bf CM}-case.}
Let $\lambda \in X^+_{00}(\Res_{F/\Q}(T_{N,0}) \times E)$, $\lambda = {\rm BC}_{F/F_1}(\kappa),$ $\pi_f \in \Coh_{!!}(G,\lambda),$ $\iota : E \to \C$, and $\gamma \in \Gal(\bar\Q/\Q)$. 
The Galois action on $\Sigma_F$ and $\Sigma_{F_1}$ preserves the fibers of the restriction map $\Sigma_F \to \Sigma_{F_1}$. Suppose, $w_1, w_j \in \place_\infty(F_1)$ and 
$\nu_1, \nu_j \in \Sigma_{F_1}$ are the corresponding distinguished elements, and suppose $\gamma \circ \{\nu_1, \bar\nu_1\} = \{\nu_j, \bar\nu_j\}$. 
Suppose the fiber over $\nu_1$ is $\{\eta_{11}, \eta_{12}, \dots, \eta_{1 k}\}$ (recall $k = [F:F_1]$)
then the fiber over $\bar\nu_1$ is $\{\bar\eta_{11}, \bar\eta_{12}, \dots, \bar\eta_{1 k}\}$; and similarly, if 
the fiber over $\nu_j$ is $\{\eta_{j1}, \eta_{j2}, \dots, \eta_{j k}\}$ and then the fiber over $\bar\nu_j$ is $\{\bar\eta_{j1}, \bar\eta_{j2}, \dots, \bar\eta_{j k}\}.$ There are two cases: 
\medskip
\begin{enumerate}
\item $\gamma \circ \nu_1 = \nu_j$. Then necessarily, $\gamma \circ \bar\nu_1 = \bar\nu_j$, $\gamma \circ \{\eta_{11},  \dots, \eta_{1 k}\} = \{\eta_{j1},  \dots, \eta_{j k}\}$ and 
$\gamma \circ \{\bar\eta_{11},  \dots, \bar\eta_{1 k}\} = \{\bar\eta_{j1},  \dots, \bar\eta_{j k}\}.$ 

\medskip
\item $\gamma \circ \nu_1 = \bar\nu_j$. Then necessarily, $\gamma \circ \bar\nu_1 = \nu_j$, $\gamma \circ \{\eta_{11},  \dots, \eta_{1 k}\} = \{\bar\eta_{j1},  \dots, \bar\eta_{j k}\} $ and 
$\gamma \circ \{\bar\eta_{11},  \dots, \bar\eta_{1 k}\} =  \{\eta_{j1},  \dots, \eta_{j k}\}.$
\end{enumerate}

Since $F = F_1$ is already discussed in \ref{sec:galois-action-CM-field} above, suppose that $k > 1$; suppose $\gamma \circ \eta_{11} = \eta_{j1}$ then it is possible that $\gamma \circ \{\eta_{11}, \bar\eta_{11}\} \neq \{\eta_{j1}, \bar\eta_{j1}\}.$ In particular, the Galois action on $\Sigma_F$ does not descend to give a Galois action of $\gamma$ on $S_\infty(F)$. Similarly, also in case (2). Nevertheless, using 
\eqref{eqn:arch-components-CM} it follows that 
\begin{equation}
\label{eqn:galois-tot-imag}
{}^{\gamma\circ\iota}\pi_\infty \ = \ 
\bigotimes_{w \in \place_\infty(F_1)} \bigotimes_{v|w} 
 \J(\kappa^{\iota^{-1} \circ \gamma^{-1} \circ \nu_{w}}, \kappa^{\iota^{-1} \circ \gamma^{-1} \circ \bar\nu_{w}}), 
\end{equation}
and as in \eqref{eqn:galois-CM}, the inner constituent is given by: 
\begin{equation}
\label{eqn:galois-tot-imag-inner}
 \J(\kappa^{\iota^{-1} \circ \gamma^{-1} \circ \nu_{w}}, \kappa^{\iota^{-1} \circ \gamma^{-1} \circ \bar\nu_{w}})
\ = \ 
\begin{cases}
 \J(\kappa^{\iota^{-1} \circ \nu_{\gamma^{-1} \cdot w}}, \kappa^{\iota^{-1} \circ \bar\nu_{\gamma^{-1} \cdot w}}) & \mbox{if $\gamma^{-1} \circ \nu_w = \nu_{\gamma^{-1} \cdot w}$,} \\ \\  
\overline{ \J(\kappa^{\iota^{-1} \circ \nu_{\gamma^{-1} \cdot w}}, \kappa^{\iota^{-1} \circ \bar\nu_{\gamma^{-1} \cdot w}})}  & \mbox{if $\gamma^{-1} \circ \nu_w = \overline{\nu_{\gamma^{-1} \cdot w}}.$} 
\end{cases}
\end{equation}
Hence the archimedean components of ${}^{\gamma \circ \iota} \pi$ is a permutation of the archimedean components of ${}^{\iota} \pi$ up to taking conjugates.

\medskip
\paragraph{\bf When $F$ is totally imaginary in the {\bf TR}-case.}
The Galois action on $\Sigma_F$ and $\Sigma_{F_1}$ preserves the fibers of the restriction map $\Sigma_F \to \Sigma_{F_1}$, and since $F_1$ is totally real, 
identify the Galois-sets $\Sigma_{F_1} = \place_\infty(F_1).$ Using the notations of \eqref{eqn:arch-components-TR}, if 
$$
{}^\iota\pi_\infty \ = \ \bigotimes_{w \in \place_\infty(F_1)} \bigotimes_{v|w} \ \J(\kappa^{\iota^{-1} \circ \nu_{w}}, \kappa^{\iota^{-1} \circ \nu_{w}}), 
$$
then for $\gamma \in \Gal(\bar\Q/\Q)$ one has:
$$
({}^{\gamma \circ \iota} \pi)_\infty 
\ = \ \bigotimes_{w \in \place_\infty(F_1)} \bigotimes_{v|w} \ \J(\kappa^{\iota^{-1} \circ \nu_{\gamma^{-1}\circ w}}, \kappa^{\iota^{-1} \circ \nu_{\gamma^{-1}\circ w}}).
$$
Hence the archimedean components of ${}^{\gamma \circ \iota} \pi$ is a permutation of the archimedean components of ${}^{\iota} \pi$.

\medskip
\subsection{Boundary Cohomology}
\label{sec:coh-of-bdry}

The cohomology $H^\bullet(\PBSC, \tM_{\lambda,E})$ 
of the boundary of the Borel--Serre compactification of the locally symmetric space $\SGK$ is briefly discussed here, and the reader is referred to  
\cite[Chap.\,4]{harder-raghuram-book} for more details and proofs. 
There is a spectral sequence built from the cohomology of the boundary strata $ \PPBSC$ that converges to the cohomology of the boundary.  
To understand the cohomology of a single stratum $\PPBSC,$ note that
$$
H^\bullet(\partial_P\SGK, \tM_{\lambda,E}) \ = \ 
H^\bullet(P(\Q)\backslash G(\A)/ K_\infty K_f, \tM_{\lambda,E}).
$$
The space $P(\Q)\backslash G(\A)/ K_\infty K_f$ fibers over locally symmetric spaces of $M_P.$
Let $\Xi_{K_f}$ be a complete set of representatives for $P(\A_f)\backslash G(\A_f)/K_f.$
Let $K^P_\infty = K_\infty \cap P(\R),$ and for $\xi_f \in \Xi_{K_f}$, let $K_f^P(\xi_f) = P(\A_f) \cap \xi_f K_f \xi_f^{-1}.$ 
Then
$$
P(\Q)\backslash G(\A)/ K_\infty^0K_f  \ = \ \coprod_{\xi_f \, \in \, \Xi_{K_f}} P(\Q)\backslash P(\A)/ K^P_\infty K_f^P(\xi_f). 
$$
Let $\kappa_P : P \to P/U_P = M_P$ be the canonical map, and define 
$K_\infty^{M_P} = \kappa_P(K^P_\infty)$, and for $\xi_f \in \Xi_{K_f}$, let 
$K_f^{M_P}(\xi_f) = \kappa_P(K_f^P(\xi_f)).$ 
Define 
$$
\uSMP_{K_f^{M_P}(\xi_f)} \ := \ M_P(\Q)\backslash M_P(\A) / K_\infty^{M_P}K_f^{M_P}(\xi_f).
$$ 
The underline is to emphasize that we have divided by $K_\infty^{M_P}$ that may be explicated as follows: for the maximal 
parabolic $P = P_{(n,n')},$ whose Levi quotient $M_P$  may be identified with  
the block diagonal subgroup $G_n \times G_{n'}$ where $G_n = R_{F/\Q}(\GL_n)$ and $G_{n'} = R_{F/\Q}(\GL_{n'}),$ one has 
\begin{multline*}
K_\infty^{M_P} = \kappa_P(P(\R) \cap K_\infty) = M_P(\R) \cap K_\infty = \\
\prod_{v \in \place_\infty}
\left(\begin{bmatrix} \GL_n(\C) & \\ & \GL_{n'}(\C) \end{bmatrix} 
\cap \U(N) \right) S(\R) \ = \ 
\prod_{v \in \place_\infty}
\left(\begin{bmatrix} \U(n) & \\ & \U(n') \end{bmatrix} \right) S(\R). 
\end{multline*}
Note that $K_\infty^{M_P}$ is connected. 
 Let $K_f^{U_P}(\xi_f) = U_P(\A_f) \cap \xi_f K_f \xi_f^{-1}.$ We have the fibration: 
$$
U_P(\Q)\backslash U_P(\A)/K_f^{U_P}(\xi_f) \ \hookrightarrow \ 
P(\Q)\backslash P(\A)/ K^P_\infty K_f^P(\xi_f)  \ \twoheadrightarrow \ 
\uSMP_{K_f^{M_P}(\xi_f)}.
$$
The corresponding Leray--Serre spectral sequence is known to degenerate at the $E_2$-level. The cohomology of the total space is given in terms of the cohomology of the base with coefficients in the cohomology of the fiber. For the cohomology of the fiber, if $\u_P$ is the Lie algebra of $U_P$ then 
the cohomology of the fiber is the same as the Lie algebra cohomology group 
$H^\bullet(\u_P, \M_{\lambda, E})$--by a classical theorem due to van Est, which is naturally an algebraic representation of $M_P$; the associated sheaf on $\uSMP_{K_f^{M_P}(\xi_f)}$ is denoted by putting a tilde on top. One has:
\begin{equation}
\label{eqn:bdry-P-coh}
H^\bullet(\partial_P\SGK, \tM_{\lambda,E})  \ = \ 
\bigoplus_{\xi_f \, \in \, \Xi_{K_f}} 
H^\bullet \left(\uSMP_{K_f^{M_P}(\xi_f)}, \widetilde{H^\bullet(\u_P, \M_{\lambda,E})}\right). 
\end{equation}
Pass to the limit over all open compact subgroups $K_f$ and define:
$H^\bullet(\ppBSC, \tM_{\lambda,E}) := \varinjlim_{K_f} \, 
H^\bullet(\partial_P\SGK, \tM_{\lambda,E}). 
$
Let $\uSMP \ := \ M_P(\Q)\backslash M_P(\A) / K_\infty^{M_P};$ (\ref{eqn:bdry-P-coh}) can be rewritten as: 
$$
H^\bullet(\ppBSC, \tM_{\lambda,E})^{K_f} \ = \ 
\bigoplus_{\xi_f \, \in \, \Xi_{K_f}} 
H^\bullet \left(\uSMP, \widetilde{H^\bullet(\u_P, \M_{\lambda,E})}\right)^{K_f^{M_P}(\xi_f)}. 
$$
It is clear using Mackey theory that the right hand side is the $K_f$-invariants of an algebraically induced representation; hence one has the following  

\begin{prop}
\label{prop:bdry-coh-1}
The cohomology of $\ppBSC$ is given by:
$$
H^\bullet(\ppBSC, \tM_{\lambda,E}) \ = \ 
\aInd_{P(\A_f)}^{G(\A_f)}
\left( H^\bullet(\uSMP, \widetilde{H^\bullet(\u_P, \M_{\lambda,E})}) \right).
$$
The notation $\aInd$ stands for algebraic, or un-normalized, induction. 
\end{prop}

The following is a brief review of well-known results of Kostant \cite{kostant} on the structure of $H^{\bullet}(\u_P, \M_{\lambda, E}).$ 
The calculation of the unipotent cohomology group is over the field $E$. Recall that $G \times E = \prod_{\tau : F \to E} G_0^\tau$ where $G_0^\tau = G_0 \times_{F,\tau} E = \GL_N/E.$  
Let $\bfdelta_{G_0}$ stand for the set of roots of $G_0$ with respect to $T_{N,0}$, 
$\bfdelta_{G_0}^+$ the subset of positive roots (for choice of Borel subgroup being the upper triangular subgroup), and $\bfpi_{G_0}$ the set of simple roots.  
The notations $\bfdelta_{G_0^\tau}$, $\bfdelta_{G_0^\tau}^+$ and $\bfpi_{G_0^\tau}$ are clear. 
Let $P = \Res_{F/\Q}(P_0)$ be the parabolic subgroup of $G$ as above, and let $P_0^\tau := P_0 \times_{\tau} E.$
The Weyl group factors as $W = \prod_{\tau : F \to E} W_0^\tau$ with each $W_0^\tau$ isomorphic to the permutation group $\perm_N$ on $N$-letters. 
Let $W^P$ be the set of Kostant representatives in the Weyl group $W$ of $G$ corresponding to the parabolic subgroup $P$ defined as: $W^P = \{ w = (w^\tau) : w^\tau \in W^\tau_0{}^{P^\tau_0} \},$ where,  
$$
W^\tau_0{}^{P^\tau_0} \ := \{w^\tau \in W^\tau_0  : \ (w^\tau)^{-1}\alpha > 0, \ \forall \alpha \in \bfpi_{M_{P_0^\tau}} \}. 
$$
Here $\bfpi_{M_{P_0^\tau}} \subset \bfpi_{G_0^\tau}$ denotes the set of simple roots in the Levi quotient $M_{P_0^\tau}$ of $P_0^\tau.$ 
The twisted action of $w \in W$ on $\lambda \in X^*(T)$ is 
$w \cdot \lambda = (w^\tau \cdot \lambda^\tau)_{\tau : F \to E}$ and 
$w^\tau\cdot \lambda^\tau = w^\tau(\lambda^\tau + \bfgreek{rho}_N) - \bfgreek{rho}_N,$ where 
$\bfgreek{rho}_N = \tfrac12 \sum_{\alpha \in \bfdelta_{G_0}^+} \alpha.$ 
For $w \in W^P$, the irreducible finite-dimensional representation of $M_P \times E$ with extremal weight $w \cdot \lambda$ is denoted $\M_{w\cdot \lambda, E}$ . 
Kostant's theorem asserts that one has a multiplicity-free decomposition of $M_P \times E$-modules: 
\begin{equation}
\label{eqn:kostant}
H^q(\u_P, \M_{\lambda, E}) \ \simeq \  \bigoplus_{\substack{w \in W^P \\ l(w) = q}} \M_{w \cdot \lambda, E}.
\end{equation}
As explained in \cite{harder-raghuram-book}, the above result of Kostant can be parsed over the set of embeddings $\tau : F \to E.$  
Denote by $H^{l(w)}(\u_P, \M_{\lambda, E})(w)$ the summand of $H^q(\u_P, \M_{\lambda, E})$ corresponding to the Kostant representative $w$ which is nonzero 
for $q = l(w)$ and isomorphic to $\M_{w \cdot \lambda, E}$ . 
Applying \eqref{eqn:kostant} to the boundary cohomology as in Prop.\,\ref{prop:bdry-coh-1} gives the following 

\begin{prop}
\label{prop:bdry-coh-2}
The cohomology of $\ppBSC$ is given by:
$$
H^q(\ppBSC, \tM_{\lambda, E})  \ = \ 
\bigoplus_{w \in W^P}
{}^{\rm a}{\rm Ind}_{P(\A_f)}^{G(\A_f)}
\left(H^{q - l(w)}(\uSMP,  \widetilde{H^q(\u_P, \M_{\lambda, E})(w)} ) \right).
$$
\end{prop}

There is a canonical surjection $\uSMP \to \SMP$, using which we may inflate up the cohomology of $\SMP$ to the cohomology of $\uSMP$; 
this will be especially relevant to strongly inner cohomology classes of $\SMP$, which after inducing up to $G(\A_f)$ will contribute to boundary cohomology; see Sect.\,\ref{sec:ind-rep-bdry-coh}.

\medskip

\subsection{Galois action and local systems in boundary cohomology}
\label{sec:galois-infty-bdry}
For an embedding $\iota : E \to \C$, the map $\gamma_*$ induced by a Galois element $\gamma \in \Gal(\bar\Q/\Q)$ in unipotent cohomology  
$$
H^q(\u_P, \M_{{}^\iota\!\lambda, \C})({}^\iota w) \ \to \ H^q(\u_P, \M_{{}^{\gamma \circ \iota} \! \lambda, \C})({}^{\gamma \circ \iota} w), 
$$
where, $q = l(w) = l({}^\iota w) = l({}^{\gamma\circ\iota} w)$, will play a role in the proof of the reciprocity law of the main theorem. By Schur's lemma, this can be understood by its effect 
on the highest weight vector for the irreducible representation 
$H^q(\u_P, \M_{{}^\iota\!\lambda, \C})({}^\iota w) \simeq  \M_{{}^\iota w \cdot {}^\iota\!\lambda, \C}.$
Such a highest weight vector ${\sf h}(\lambda, w, \iota)$ 
 will be fixed by fixing a harmonic representative the corresponding cohomology class as in Kostant \cite[Thm.\,5.14]{kostant}. To explicate this vector note that: 
$$
\u_P \otimes E \ = \ \bigoplus_{\tau : F \to E} \u_{P_0}^\tau; \quad \u_{P_0}^\tau := \u_{P_0} \otimes_{F, \tau} E.
$$
Fix an ordering: 
$$
\Hom(F, E) = \{\tau_1,\tau_2,\dots,\tau_d\}.
$$ 
Let $\bfdelta(\u_{P_0})$ denote the subset of $\bfdelta^+$ of those positive roots $\varphi$ whose root space $X_\varphi$ is in $\u_{P_0}$. Fix an ordering
$$
\bfdelta(\u_{P_0}) = \{\varphi_1,\varphi_2,\dots,\varphi_{nn'}\}.
$$ 
For example, thinking in terms of upper triangular matrices, 
this ordering could be taken as the lexicographic ordering on the set of pairs of indices $\{(i,j): 1 \leq i \leq n, \ 1 \leq j \leq n'\}$. Fix a generator $e_\varphi$ for $X_\varphi$ for each 
$\varphi \in \bfdelta(\u_{P_0});$ note that $e_\varphi$ is well-defined up to $\Q^\times.$  Let $\{e_\varphi^*\}$ denote the basis of $\u_{P_0}^*$ that is dual 
to $\{e_\varphi\}.$ 
For a Kostant representative $w_0 \in W^{P_0} \subset W_{G_0}$, 
define $\Phi_{w_0} = \{\varphi > 0 : w_0^{-1}\varphi < 0\};$ then $\Phi_{w_0} \subset \bfdelta(\u_{P_0})$; with respect to the ordering that it inherits from $\bfdelta(\u_{P_0})$, denote 
$\Phi_{w_0} = \{\varphi^{w_0}_1, \dots,\varphi^{w_0}_{l}\}$ as an ordered set, where $l = l(w_0^{-1}) = l(w_0).$ Define 
$$
e_{\Phi_{w_0}}^* \ := \ e_{\varphi^{w_0}_1}^* \wedge \cdots \wedge e_{\varphi^{w_0}_l}^* \ \in \ \wedge^{q_0}(\u_{P_0}^*); \quad q_0 := l(w_0).
$$
Let $e_\varphi^\tau$ denote the image $e_\varphi \otimes 1$ of $e_\varphi$ under the canonical map $X_\varphi \to X_\varphi^\tau = X_\varphi \otimes_{F,\tau} E$. For  
$w = (w^\tau)_{\tau: F \to E} \in W_G = \prod_{\tau: F \to E} W_{G_0} \times_{F, \tau} E$, written using the ordering on $\Hom(F, E)$ as $w = \{w^{\tau_1}, \dots, w^{\tau_d}\}$, define: 
$$
e_{\Phi_w}^* \ := \ e_{\Phi_{w^{\tau_1}}}^* \wedge \cdots \wedge e_{\Phi_{w^{\tau_d}}}^* \ \in \ \wedge^q(\u_{P}^*\otimes E); \quad q := l(w). 
$$
Changing the base to $\C$ via $\iota : E \to \C$ gives
\begin{equation}
\label{eqn:e-Phi}
e_{\Phi_{^\iota \! w}}^* \ := \ e_{\Phi_{w^{\iota \circ \tau_1}}}^* \wedge \cdots \wedge e_{\Phi_{w^{\iota \circ \tau_d}}}^* \ \in \ \wedge^q(\u_{P}^* \otimes_\Q \C). 
\end{equation}
Fix a weight vector ${\sf s}(\lambda^\tau) \in \M_{\lambda^\tau, E}$ for the highest weight $\lambda^\tau;$ then 
${\sf s}(\lambda) = {\sf s}(\lambda^\tau_1) \otimes \cdots \otimes {\sf s}(\lambda^\tau_d)$ is the highest weight vector for  $\M_{\lambda, E}$. 
For each $w \in W,$ fix its representative in $G(E)$, which amounts to fixing a permutation matrix representing $w^\tau$ in $\GL_n(E)$ for 
each embedding $\tau : F \to E$. Let 
\begin{equation}
\label{eqn:extremal-w-vector}
{\sf s}(w \lambda) \ :=  \ \rho_{\lambda^{\tau_1}}(w^{\tau_1}){\sf s}(\lambda^\tau_1) \otimes \cdots \otimes  \rho_{\lambda^{\tau_d}}(w^{\tau_d}){\sf s}(\lambda^\tau_d)
\end{equation}
be the weight vector of extremal weight $w\lambda.$ These vectors can be composed via $\iota$: 
${\sf s}({}^\iota \!w \, {}^\iota\!\lambda)$ is the weight vector in $\M_{{}^\iota\!\lambda, \C}$ of extremal weight ${}^\iota \!w \, {}^\iota\!\lambda.$ Theorem 5.14 of \cite{kostant} asserts that 
\begin{equation}
\label{eqn:h-w-vector}
{\sf h}(\lambda, w, \iota) \ = \ e_{\Phi_{^\iota \! w}}^* \otimes {\sf s}({}^\iota \!w \, {}^\iota\!\lambda)  
\end{equation} 
is the highest weight vector for $H^q(\u_P, \M_{{}^\iota\!\lambda, \C})({}^\iota w).$ The image of ${\sf h}(\lambda, w, \iota)$ under the 
map $\gamma_*$ induced by $\gamma \in \Gal(\bar\Q/\Q)$ is a multiple of ${\sf h}(\lambda, w, \gamma \circ \iota)$; the scaling factor is captured by what 
$\gamma$ does to the wedge-products $e_{\Phi_{^\iota \! w}}^*$, motivating the following 

\begin{defn}
\label{def:signature}
Let $\iota : E \to \C$ and $\gamma \in \Gal(\bar\Q/\Q)$ then we have 
$$
e_{\Phi_{^{\gamma \circ \iota} \! w}}^*  \ = \ \varepsilon_{\iota, w}(\gamma) e_{\Phi_{^\iota \! w}}^* 
$$
for a signature $\varepsilon_{\iota, w}(\gamma) \in \{\pm 1\}$.
\end{defn}
From \eqref{eqn:extremal-w-vector}, \eqref{eqn:h-w-vector}, and the above definition one has:  
\begin{equation}
\label{eqn:galois-signature}
\gamma_*({\sf h}(\lambda, w, \iota)) \ = \ \varepsilon_{\iota, w}(\gamma) \cdot  {\sf h}(\lambda, w, \gamma \circ \iota). 
\end{equation}

\bigskip
\section{\bf The critical set and a combinatorial lemma}
\label{sec:crit-comb-lemma}

\medskip
\subsection{The critical set for $L(s, \sigma \times \sigma'^\v)$}
\label{sec:crit-set-description}

Let $n$ and $n'$ be two positive integers, and consider weights $\mu \in X^+_{00}(T_n \times \C)$ and 
$\mu' \in X^+_{00}(T_{n'} \times \C)$ given by:  
\begin{equation}
\label{eqn:mu}
\mu = (\mu^\eta)_{\eta : F \to \C}, \quad 
\mu^\eta = \sum_{i=1}^{n-1} (a^\eta_i-1) \bfgreek{gamma}_i + d^\eta \cdot  \bfgreek{delta} \ = \ (b^\eta_1,\dots, b^\eta_n), 
\end{equation}
and similarly, 
\begin{equation}
\label{eqn:mu'}
\mu' = (\mu'^\eta)_{\eta : F \to \C}, \quad 
\mu'^\eta = \sum_{j=1}^{n'-1} (a'^\eta_i-1) \bfgreek{gamma}_j + d'^\eta \cdot  \bfgreek{delta} \ = \ (b'^\eta_1,\dots, b'^\eta_{n'}). 
\end{equation}
Let $\sigma_f \in \Coh_{!!}(G_n, \mu)$ and $\sigma'_f \in \Coh_{!!}(G_{n'}, \mu')$ be strongly inner Hecke-summands; these Hecke-summands, take a unique representation at infinity to contribute to the respective cuspidal spectrum cohomology. Denote $\sigma_\infty = \J_\mu$ and 
$\sigma'_\infty = \J_{\mu'}$ then $\sigma = \sigma_\infty \otimes \sigma_f$ and $\sigma' = \sigma'_\infty \otimes \sigma'_f$ are cuspidal automorphic representations. We let $L(s, \sigma \times \sigma')$ stand for the completed standard Rankin--Selberg $L$-function of 
degree $nn'.$ We refer the reader to \cite[Sect.\,10.1]{shahidi-book} for a summary of the basic analytic properties of these $L$-functions. 
The purpose of this section is to identify the set of integers or possibly half-integers $m$ which are critical for  
$L(s, \sigma \times \sigma'^\v).$ (Note that we have dualized $\sigma'.$)

\medskip
\subsubsection{\bf Definition of the critical set}
\label{sec:defn-crit-set}
For any two half-integers $\alpha$ and $\beta,$ the local $L$-factor (see \cite{knapp}) of the character $z \mapsto z^\alpha \bar{z}^\beta$ of $\C^\times$ is given by: 
\begin{multline}
\label{eqn:abelian-local-l-factor}
L(s, z^\alpha \bar{z}^\beta)
\ = \ 
2(2\pi)^{-\left(s + \frac{\alpha+\beta}{2} + \frac{|\alpha - \beta|}{2}\right)}
\Gamma\left(s + \frac{\alpha+\beta}{2} + \frac{|\alpha - \beta|}{2}\right) \\ \sim \ 
\Gamma\left(s + \frac{\alpha+\beta}{2} + \frac{|\alpha - \beta|}{2}\right), 
\end{multline}
where, by $\sim$, we mean up to nonzero constants and exponential functions, which are entire and nonvanishing everywhere and hence are irrelevant to the computation of critical points; see Def.\,\ref{def:crit} below. For any 
$v \in \place_\infty$, let $\{\eta_v, \bar{\eta}_v\}$ be the pair of conjugate embeddings of $F$ to $\C$ as before. Let 
$$
\alpha^v \ =  \ -w_0\mu^{\eta_v} + \bfgreek{rho}_n \ = \ (\alpha^v_1,\dots, \alpha^v_n) 
\ \ \ {\rm and} \ \ \ 
\beta^v  \ = \ - \mu^{\bar\eta_v} - \bfgreek{rho}_n \ = \ (\beta^v_1,\dots, \beta^v_n)
$$
be the cuspidal parameters of $\mu$ at $v$; see  
\eqref{eqn:cuspidal-parameters-alpha} and \eqref{eqn:cuspidal-parameters-beta}. Similarly, let 
$$
\alpha'^v \ =  \ -w_0\mu'^{\eta_v} + \bfgreek{rho}_{n'} \ = \ (\alpha'^v_1,\dots, \alpha'^v_{n'}) 
\ \ \ {\rm and} \ \ \ 
\beta'^v  \ = \ - \mu'^{\bar\eta_v} - \bfgreek{rho}_{n'} \ = \ (\beta'^v_1,\dots, \beta'^v_{n'})
$$
be the cuspidal parameters of $\mu'$ at $v.$ 
Note that 
$$
\alpha := \alpha_i^v + \alpha'^v_j \ \in \ 
\frac{n-1}{2} + \Z + \frac{n'-1}{2} + \Z \ = \ 
\frac{N}{2} + \Z, \quad {\rm and} \quad 
\beta := \beta_i^v + \beta'^v_j \ \in \ \frac{N}{2} + \Z. 
$$
Then, it is clear that the quantity $\tfrac{\alpha+\beta}{2} + \tfrac{|\alpha - \beta|}{2}$ 
inside the argument of the $\Gamma$-function above is in 
$\tfrac{N}{2} + \Z.$ 
This tells us that the critical set for $L(s, \sigma \times \sigma')$ will be a subset of $\tfrac{N}{2} + \Z.$ 

\medskip

Let $\sigma$ and $\sigma'$ be cuspidal automorphic representations of $G_n(\A)$ and $G_{n'}(\A)$, respectively. The set of critical points for 
$L(s, \sigma \times \sigma'^\v)$ is defined to be: 
\begin{multline}\label{def:crit} 
\Crit(L(s, \sigma \times \sigma'^\v)) \ := \ \\
\left\{m \in \tfrac{N}{2} + \Z \ : 
\mbox{both $L_\infty(s, \sigma \times \sigma'^\v)$ and $L_\infty(1-s, \sigma^\v \times \sigma')$ are finite at $s = m$} \right\}. 
\end{multline}
If $\sigma$ and $\sigma'$ are cohomological with respect to $\mu$ and $\mu'$, then we denote
\begin{equation}
\label{eqn:crit-mu-mu'}
\Crit(\mu, \mu') \ := \ \Crit(L(s, \sigma \times \sigma'^\v)). 
\end{equation}

\medskip
\subsubsection{\bf Computing the critical set}
\label{sec:compute-crit-set}
Recall, the purity conditions:
$$
\alpha^v_i + \beta^v_i = - \w, \ \ {\rm and} \ \ 
\alpha'^v_i + \beta'^v_i = - \w'. 
$$ 
We define a quantity $a(\mu,\mu')$, and call it the {\it abelian width} between $\mu$ and $\mu'$, as: 
\begin{equation}
\label{eqn:abelian-width}
a(\mu,\mu') \ := \ 
\frac{\w - \w'}{2} \ = \ \frac{(d^\eta + d^{\bar\eta}) - (d'^\eta + d'^{\bar\eta})}{2}. 
\end{equation}

From the local Langlands correspondence and \eqref{eqn:abelian-local-l-factor} on abelian local $L$-factors, we get 
\begin{equation}
\label{eqn:l-factor-infinity}
L_\infty(s, \sigma \times \sigma'^\v) \ \sim \ 
\prod_{v \in \place_\infty} \, \prod_{i=1}^n \, \prod_{j=1}^{n'} \  
\Gamma\left(
s - a(\mu,\mu') + \frac{|\alpha^v_i - \alpha'^v_j - \beta^v_i + \beta'^v_j|}{2} \right).
\end{equation}
And, similarly, 
\begin{equation}
\label{eqn:l-factor-infinity-dual}
L_\infty(1-s, \sigma^\v \times \sigma') \ \sim \ 
\prod_{v \in \place_\infty} \, \prod_{i=1}^n \, \prod_{j=1}^{n'} \ 
\Gamma\left(
1-s + a(\mu,\mu') + \frac{|\alpha^v_i - \alpha'^v_j - \beta^v_i + \beta'^v_j|}{2} \right).
\end{equation}

Let $m \in \tfrac{N}{2} + \Z.$ Then $m \in \Crit(\mu,\mu')$ if and only if 
\begin{equation}
\label{eqn:crit-m-1}
m - a(\mu,\mu') + \frac{|\alpha^v_i - \alpha'^v_j - \beta^v_i + \beta'^v_j|}{2}  \ \geq \ 1, \quad \forall v \in \place_\infty, \forall i, \forall j, 
\end{equation}
which is the condition that $L_\infty(m, \sigma \times \sigma'^\v)$ is finite from \eqref{eqn:l-factor-infinity}, and 
\begin{equation}
\label{eqn:crit-m-2}
1-m + a(\mu,\mu') + \frac{|\alpha^v_i - \alpha'^v_j - \beta^v_i + \beta'^v_j|}{2}  \ \geq \ 1, \quad \forall v \in \place_\infty, \forall i, \forall j, 
\end{equation}
which is the condition that $L_\infty(1-m, \sigma^\v \times \sigma')$ is finite from \eqref{eqn:l-factor-infinity-dual}. 
Define the {\it cuspidal width} $\ell(\mu,\mu')$ between $\mu$ and $\mu'$ as: 
\begin{equation}
\label{eqn:cuspidal-width}
\ell(\mu,\mu') \ := \ 
{\rm min}
\left\{|\alpha^v_i - \alpha'^v_j - \beta^v_i + \beta'^v_j| \ : \ v \in \place_\infty, 1 \leq i \leq n, \ 1\leq j \leq n' \right\}. 
\end{equation}
Then \eqref{eqn:crit-m-1} and \eqref{eqn:crit-m-2} together gives us the following

\medskip
\begin{prop}
\label{prop-crit-mu-mu'}
Let $\mu \in X^+_{00}(T_n \times \C)$ and $\mu' \in X^+_{00}(T_{n'} \times \C).$ For  
$\sigma_f \in \Coh_{!!}(G_n, \mu)$ and $\sigma'_f \in \Coh_{!!}(G_{n'}, \mu')$, the critical set for the Rankin--Selberg $L$-function 
$L(s, \sigma \times \sigma'^\v)$ is given by: 
$$
\Crit(\mu, \mu') \ = \ 
\left\{m \in \tfrac{N}{2} + \Z \ : \ 
1 - \frac{\ell(\mu,\mu')}{2} + a(\mu, \mu') \ \leq \ m \ \leq \ \frac{\ell(\mu,\mu')}{2} + a(\mu, \mu') \right\}.
$$
This is contiguous string of integers or half-integers (depending on whether $N$ is even or odd), 
centered around $\tfrac12+ a(\mu, \mu')$, of length $\ell(\mu,\mu').$ 
\end{prop}

\medskip
\begin{cor}
\label{cor-crit-N/2}
With notations as in Prop.\,\ref{prop-crit-mu-mu'}, the points $s = -N/2$ and $s = 1-N/2$ are both critical for $L(s, \sigma \times \sigma'^\v)$ if and only if 
$$
- \frac{N}{2} + 1 - \frac{\ell(\mu,\mu')}{2}
 \ \leq \ a(\mu, \mu')  \ \leq \  
- \frac{N}{2} -1 + \frac{\ell(\mu,\mu')}{2}. 
$$
\end{cor} 
Of course, for this to be possible one needs $\ell(\mu,\mu') \geq 2$, i.e., that there at least two critical points. 
The corollary, which is one part of a {\it combinatorial lemma} below, is to be viewed like this: the two successive $L$-values at $s = -N/2$ and $s = 1-N/2$ are critical if and only if the abelian width is bounded in absolute value in terms of the cuspidal width.

\medskip
\begin{cor}
\label{cor:F-tot-real-nn'-even}
Suppose $F$ is in the {\bf TR}-case and $F_1 = F_0$ is the maximal totally real subfield of $F$. Given $\mu \in X^+_{00}(T_n \times \C)$ and $\mu' \in X^+_{00}(T_{n'} \times \C),$ 
if $n$ and $n'$ are both odd, then $\ell(\mu,\mu') = 0$; in particular, the Rankin--Selberg $L$-function 
$L(s, \sigma \times \sigma'^\v)$ has no critical points. 
\end{cor}

\begin{proof}
Recall from Prop.\,\ref{prop:strong-pure-weights-base-change} that $\mu$ is the base change of a strongly-pure weight over $F_1$. For $v \in \place_\infty(F)$, one has 
$\eta_v|_{F_1} = \bar\eta_v|_{F_1}$, hence 
$\mu^{\eta_v} = \mu^{\bar\eta_v}$. Hence, for the cuspidal parameters, one has $\alpha^v = w_0 \beta^v$, i.e., $\alpha^v_i = \beta^v_{n+1-i}.$ If $n$ is odd, then 
$\alpha^v_{(n+1)/2} = \beta^v_{(n+1)/2}.$ Similarly, if $n'$ is odd, then $\alpha'^v_{(n'+1)/2} = \beta'^v_{(n'+1)/2}.$ From \eqref{eqn:cuspidal-width} it follows 
that $\ell(\mu,\mu') = 0$, as $0$ is realized as the minimum by taking $i = (n+1)/2$ and $j = (n'+1)/2.$ 
\end{proof}

\medskip
\subsubsection{\bf Critical set at an arithmetic level} 
\label{sec:crit-set-arith-level}
Let $\mu \in X^+_{00}(T_n \times E)$ and $\mu' \in X^+_{00}(T_{n'} \times E),$ and take  
$\sigma_f \in \Coh_{!!}(G_n, \mu)$ and $\sigma'_f \in \Coh_{!!}(G_{n'}, \mu').$ For any $\iota : E \to \C$, 
Prop.\,\ref{prop-crit-mu-mu'} gives 
the critical set $\Crit({}^\iota\mu, {}^\iota\mu')$  for the Rankin--Selberg $L$-function 
$L(s, {}^\iota\sigma \times {}^\iota\sigma'^\v).$ 

\medskip
\begin{cor}
\label{cor:crit-set-ind-iota}
The critical set $\Crit({}^\iota\mu, {}^\iota\mu') 
= \Crit(L(s, {}^\iota\sigma \times {}^\iota\sigma'^\v))$ is independent of $\iota:$  
$$
\Crit(L(s, {}^\iota\sigma \times {}^\iota\sigma'^\v)) \ = \ 
\Crit(L(s, {}^{\gamma \circ\iota}\sigma \times {}^{\gamma \circ\iota}\sigma'^\v)), \quad \forall \ \iota : E \to \C, \ \ \forall \gamma \in \Gal(\bar\Q/\Q).
$$
\end{cor} 
\begin{proof}
From Rem.\,\ref{rem:interchange-alpha-beta}, one can deduce 
$\ell({}^\iota\mu, {}^\iota\mu') = \ell({}^{\gamma \circ \iota}\mu, {}^{\gamma \circ \iota}\mu')$ and $a({}^\iota\mu, {}^\iota\mu') = a({}^{\gamma \circ \iota}\mu, {}^{\gamma \circ \iota}\mu').$ 
One can also see this directly, since by the results of \ref{sec:galois-arch-constituents}, %\eqref{eqn:galois-tot-imag} and \eqref{eqn:galois-tot-imag-inner}, 
the archimedean components of ${}^{\gamma \circ\iota}\sigma$ are a permutation 
of those of ${}^\iota\sigma$ up to conjugates; similarly, for ${}^\iota\sigma'$; since $L(s, z^\alpha \bar{z}^\beta) = L(s, z^\beta\bar{z}^\alpha)$ one gets $L_\infty(s, {}^\iota\sigma \times {}^\iota\sigma'^\v) = L_\infty(s, {}^{\gamma \circ\iota}\sigma \times {}^{\gamma \circ\iota}\sigma'^\v).$ 
\end{proof}

\medskip
\subsection{Combinatorial Lemma}
\label{sec:comb-lemma}

\medskip
\subsubsection{\bf Statement of the lemma}

\begin{lemma}
\label{lem:comb-lemma}
For strongly-pure weights $\mu \in X^+_{00}(T_n \times \C)$ and $\mu' \in X^+_{00}(T_{n'} \times \C)$, and cuspidal Hecke summands  
$\sigma_f \in \Coh_{!!}(G_n, \mu),$  $\sigma'_f \in \Coh_{!!}(G_{n'}, \mu'),$ the following are equivalent: 
\begin{enumerate}
\item The points $s = -N/2$ and $s = 1-N/2$ are both critical for $L(s, \sigma \times \sigma'^\v)$. 
\medskip
\item The abelian width is bounded in terms of the cuspidal width as: 
$$
- \frac{N}{2} + 1 - \frac{\ell(\mu,\mu')}{2}
 \ \leq \ a(\mu, \mu')  \ \leq \  
- \frac{N}{2} -1 + \frac{\ell(\mu,\mu')}{2}. 
$$
%\medskip
\item There exists $w \in W^P$ such that $w^{-1}\cdot(\mu+\mu')$ is dominant and $l(w^\eta) + l(w^{\bar\eta}) = \dim(U_{P_0})$ for all $\eta : F \to \C.$ 
(Recall: $w = (w^\eta)_{\eta : F \to \C}$ with $w^\eta \in W^{P_0 \times_\eta \C} \subset W_{G_0} \times_\eta \C$.)

\end{enumerate}
\end{lemma}

We have already proved $(1) \iff (2)$. It remains to prove $(2) \iff (3).$ 
It is clear that 
$$
l(w^\eta) + l(w^{\bar\eta}) = \dim(U_{P_0}), \ \forall \eta : F \to \C \ \implies \ 
l(w) = \tfrac12 \dim(U_P). 
$$
However, if the degree of $F$ is greater than $2$ (i.e., if ${\sf r} > 1$) then the converse is not true in general. 

\begin{defn}
\label{def:strongly-balanced}
A Kostant representative $w \in W^P$ is said to be \underline{balanced} if 
$$l(w^\eta) + l(w^{\bar\eta}) = \dim(U_{P_0}), \ \forall \eta : F \to \C.$$
\end{defn}

\smallskip

For the benefit of the reader, we will make two passes over the proof of (2) $\iff$ (3) in simpler situations, because the proof in the general case is intricate in details and somewhat tedious; it is the 
sort of proof that makes one believe the dictum ``{\it der Teufel steckt im Detail.}''

\medskip
\subsubsection{\bf Explicating (2) $\iff$ (3) in the simplest nontrivial example}
\label{sec:example-n=n'=r=1}
\begin{proof}
Let us consider the case of $n = n' = 1$ and so $N = 2.$ Take $F$ to be an imaginary quadratic field with  
$\Hom(F,\C) = \{\eta, \bar\eta\}.$ 
The weights $\mu$ and $\mu'$ are both a pair of integers indexed by $\Hom(F,\C);$ we will write
$
\mu \ = \ ((a), (a^*)), \quad \mu' \ = \ ((b), (b^*)), 
$ 
with $a,a^*, b, b^* \in \Z,$ with the convention that $\mu^\eta = (a), \mu^{\bar\eta} = (a^*)$ and similarly for $\mu'.$  Note that purity of $\mu$ and 
$\mu'$ is automatic, and the purity weights are
$\w = a+a^*, \ \w' = b+b^*.$
The abelian width is
$a(\mu, \mu') = \frac{a+a^* - b-b^*}{2}.$
The cuspidal parameters at the only complex place $v$ of $F$ are: 
$\alpha^v = (-a), \ 
\beta^v = (-a^*), \ 
\alpha'^v = (-b), \ 
\beta'^v = (-b^*).$
The cuspidal width is
$\ell(\mu, \mu') = |-a+a^* +b-b^*|.$
The weight $\mu+\mu'$ which we would like make dominant using a balanced Kostant representative has the shape
$\mu+\mu' \ = \ ((a,b), \, (a^*, b^*)).$
For simplicity, let us denote: 
$p := a-b, \ p^* := a^* - b^*.$
Hence, $\mu+\mu'$ is dominant if and only if $p \geq 0$ and $p^* \geq 0.$ 
The inequalities in (2) now take the shape: 
\begin{equation}
\label{eqn:p-p*}
- \frac{|p^*-p|}{2} \ \leq \ 
\frac{p+p^*}{2}  \ \leq \ 
\frac{|p^*-p|}{2} - 2. 
\end{equation}

Since, $P_0 = B_0$ is the Borel subgroup, the Levi subgroup $M_P$ is a torus; hence, $W_{M_P}$ is trivial and 
$W^P = W_G.$ If $W_{G_0}$ is written as $\{1,s\}$ with $s$ the nontrivial element, then 
the elements of $W^P$ may be written as: 
$
W_G = (\{1,s\}, \{1^*, s^*\}). $
The dimension of $U_P$ is $2$, hence the balanced elements (of length $1$) of $W^P$ are $(1,s^*)$ and $(s,1)$. 
Now, consider three cases depending on the sign of $p - p^*$: 
\begin{itemize}
\item $p=p^*.$ In this case, \eqref{eqn:p-p*} reads $0 \leq p \leq -2,$ which is absurd; hence (2) is violated. If $p \geq 0$ then the only $w \in W^P$ such that $w^{-1}\cdot(\mu+\mu')$ is dominant is $w = (1,1^*)$ which has length $0$, hence (3) is violated. Similarly, if $p < 0$,  then the only $w \in W^P$ such that $w^{-1}\cdot(\mu+\mu')$ is dominant is $w = (s,s^*),$ which has length $2$, hence (3) is violated again. So, both (2) and (3) are false. 

\medskip
\item $p>p^*.$ In this case, \eqref{eqn:p-p*} simplifies to $p^* - p \leq p + p^* \leq p - p^* - 4,$ which implies that 
$p \geq 0 > -2 \geq p^*.$  The only $w \in W^P$ such that $w^{-1}\cdot(\mu+\mu')$ is dominant is $w = (1,s^*)$ which has length $1$, hence (3) is satisfied. 

\medskip
\item $p<p^*.$ In this case, $p^* \geq 0 > -2 \geq p$ and the only $w \in W^P$ that works is $(s,1)$ which is of length $1$. 
\end{itemize}
In all cases, either both (2) and (3) are satisfied, or both are violated. Hence, $(2) \iff (3).$ 
\end{proof}

\medskip
In the second case ($p>p^*$), one might ask what happens in the degenerate case of $p=0$ and $p^*=-1.$ (So we are violating (2) but keeping $p>p^*$.) This means that $\mu+\mu'$ has the shape
$((a,a), (b^*-1,b^*))$. The $\eta$ component $(a,a)$ is dominant, but one has to make the $\bar\eta$-component $(b^*-1,b^*)$ dominant. This can only be done using $s^*$, however, the reader can easily check that $s^* \cdot (b^*-1,b^*) = (b^*-1,b^*)$. In other words, there is no element $w$ such that
$w^{-1}\cdot(\mu+\mu')$ is dominant.

\medskip
\subsubsection{\bf Proof of (2) $\iff$ (3) for $\GL_n \times \GL_1$}
\label{sec:example-n'=1}

It is most convenient to first understand the case when ${\sf r} = 1,$ i.e., when $F$ is an imaginary quadratic field. Then 
$\Sigma_F = \{\eta, \bar\eta\}$ (for a non-canonical choice of 
$\eta : F \to \C$ that is fixed once and for all).  As above, we will follow a notational artifice that all quantities indexed by $\bar\eta$ will be designated with a $*$. A weight 
$\mu \in X^+_{0}(T_n \times \C)$ may be written as:
$\mu = \{ \mu^\eta, \mu^{\bar{\eta}} \}$ with $\mu^\eta = (\mu_1 \geq \mu_2 \geq \cdots \geq \mu_n)$ and  
$\mu^{\bar\eta} = (\mu_1^* \geq \mu_2^* \geq \cdots \geq \mu_n^*),$
with $\mu_i, \mu_j^* \in \Z,$ and purity implies
$\w \ = \ \mu_i + \mu_{n-i+1}^*.$ 
A weight $\mu' \in X^+_{0}(T_1 \times \C)$ is simply a pair of integers 
$\mu' = \{b, b^*\}$ with purity weight $\w' = b+b^*$. 
The weight $\mu+\mu'$ is given by:
$$
\mu + \mu' \ = \ \{(\mu_1, \mu_2, \dots, \mu_n, b), \quad (\mu_1^*, \mu_2^*,  \dots, \mu_n^*, b^*) \}. 
$$

\smallskip

We are seeking to understand when we can find a Kostant representative $w \in W^P$ which is balanced ($l(w^\eta) + l(w^{\bar{\eta}})= \dim(U_{P_0}) = n$) and such that $w^{-1}\cdot(\mu + \mu')$ is a dominant weight. For this, first identify the Kostant representatives for $P_0$ in $G_0$; the simple roots of $M_{P_0}$ are
${\bf \Pi}_{M_{P_0}} = \{e_1-e_2, \, e_2-e_3, \, \dots, e_{n-1} - e_n\}.$
The Weyl group of $G_0$ is $W_{G_0} = \perm_{n+1}$ the symmetric group on $n+1$ letters. We have 
\begin{eqnarray*}
w \in W^{P_0} & \iff &   w^{-1}(e_1-e_2) > 0, \ w^{-1}(e_2-e_3) > 0, \ \dots, w^{-1}(e_{n-1} - e_n) > 0 \\ 
& \iff &  w^{-1}(1) < w^{-1}(2) < \cdots < w^{-1}(n).
\end{eqnarray*}
The elements of $W^{P_0}$ and their lengths are listed below: 
$$
\begin{array}{|l||c||}
\hline
w^{-1} \ \ (w\in W^{P_0})      & l(w)  \\ \hline
s_0 := 1 & 0 \\ 
s_1 := (n, n+1) & 1 \\
s_2 := (n-1, n, n+1) & 2 \\ 
\quad \quad \vdots & \vdots \\ 
s_{n-1} := (2,3,\dots, n+1) & n-1 \\ 
s_n := (1,2,3,\dots, n+1) & n \\
\hline
\hline
\end{array}
$$
Note that the $(n+1)$-cycle $(1,2,\dots,n+1) = (1,2)(2,3)\cdots (n,n+1)$ as a product of $n$ simple transpositions giving its length which applies to the last row and a similar calculation gives all the other lengths. The Kostant representatives for $P$ are:
$W^P \ = \ \{(w, w^*) \ : \ w , w^* \in W^{P_0}\},$ where $l(w,w^*) = l(w) + l(w^*).$
Hence the inverses of the balanced Kostant representatives are: 
$$
\{(s_0, s_n^*), \ (s_1, s_{n-1}^*), \dots, (s_n, s_0^*) \}.
$$
The twisted action of the Kostant representatives on the weight are given in the table below: 
\begin{equation}
\label{eqn:twisted-action-n'=1}
\begin{array}{||c|l||}
\hline
w^{-1} \ \ (w\in W^{P_0})      & \quad w^{-1} \cdot (\mu_1,\mu_2,\dots, \mu_n,b)  \\ \hline
1 & \quad (\mu_1,\mu_2,\dots, \mu_n, b) \\ 
(n, n+1) & \quad (\mu_1,\mu_2,\dots, \mu_{n-1}, b-1, \mu_n +1) \\
 (n-1, n, n+1) & \quad (\mu_1,\mu_2,\dots, \mu_{n-2}, b-2, \mu_{n-1} +1, \mu_n +1) \\ 
\quad \quad \vdots & \quad\quad\quad \quad\quad\vdots \\ 
(2,3,\dots, n+1) & \quad (\mu_1, b-n+1, \mu_2 +1 ,\dots, \mu_{n-1} +1, \mu_n +1) \\ 
(1,2,3,\dots, n+1) &  \quad (b-n, \mu_1+1, \dots, \mu_{n-1} +1, \mu_n +1)\\
\hline
\hline
\end{array}
\end{equation}

\smallskip

For the combinatorial lemma, the abelian width is given by:
$$
a(\mu, \mu') \ = \ \frac{\w - \w'}{2} \ = \ \frac{\mu_i + \mu_{n-i+1}^* - b - b^*}{2} \ = \ \frac{(\mu_i - b) + (\mu_{n-i+1}^* - b^*)}{2},  
$$
and for the cuspidal width, the cuspidal parameters are given by: 
\begin{align*}
 \alpha  = \ (\alpha_1,\dots, \alpha_n) \ =  & \ (-\mu_n+ \tfrac{(n-1)}{2}, \ -\mu_{n-1}+\tfrac{(n-3)}{2},\ \dots, \ -\mu_1-\tfrac{(n-1)}{2}), \\
 \beta  = \ (\alpha_1,\dots, \alpha_n)  \ =  & \ (-\mu_1^*-\tfrac{(n-1)}{2}, \ -\mu_2^*-\tfrac{(n-3)}{2},\ \dots, \ -\mu_n^*+\tfrac{(n-1)}{2}), 
\end{align*}
and similarly, $ \alpha' = -b, \  \beta' = -b^*,$ from which the cuspidal width is:
$$
\ell(\mu,\mu') = 
\min\left\{
\begin{array}{l}
|\mu_1^*-\mu_n+(n-1) +b -b^*|, \\
 |\mu_2^*-\mu_{n-1}+(n-3) +b -b^*|, \\ 
\quad \quad \quad \vdots \\  
|\mu_n^*-\mu_1-(n-1) +b -b^*|
\end{array}\right\}.
$$ 
From the shape of $a(\mu,\mu')$ and $\ell(\mu, \mu')$ it is convenient to introduce the quantities
$c_i :=  \mu_i - b$ and $c_i^* := \mu_i^*-b^*.$
(These are the $p$ and $p^*$ when $n=1$.) Then we have: 
$a(\mu, \mu') = \frac{c_i + c_{n-i+1}^*}{2},$  and 
$$
\ell(\mu,\mu') \ = \ 
\min\{ \ |c_1^*-c_n + (n-1)|, \ |c_2^*-c_{n-1}+(n-3)|, \ |c_n^* - c_1 - (n-1)| \ \}.
$$
From the dominance of the weights $\mu$ and $\mu'$ we have the inequalities
$$
c_1^*-c_n + (n-1) \ > \ c_2^*-c_{n-1}+(n-3) \ > \ \cdots \ > c_n^* - c_1 - (n-1). 
$$
The proof conveniently breaks into $(n+1)$ disjoint cases depending on the relative position of $0$ in the above decreasing sequence. 

\smallskip
\noindent
{\bf Case 0:} $0 > c_1^*-c_n + (n-1) \ > \ c_2^*-c_{n-1}+(n-3) \ > \ \cdots \ > c_n^* - c_1 - (n-1),$ 

\smallskip
\noindent
{\bf Case j ($1\ \leq j \leq n-1$):} $c_j^*-c_{n-j+1} + (n-2j+1) \ > 0 \ > \ c_{j+1}^*-c_{n-j}+(n-2j-1),$ 

\smallskip
\noindent
{\bf Case n:} $c_1^*-c_n + (n-1) \ > \ c_2^*-c_{n-1}+(n-3) \ > \ \cdots \ > c_n^* - c_1 - (n-1) > 0.$

\bigskip
In {\bf Case 0}, we have $\ell(\mu,\mu') = -c_1^*+c_n - (n-1).$ Keeping in mind that $N = n+1$, the inequalities in (2) of the lemma read:
$$
- \frac{(n+1)}{2} + 1 - \frac{(-c_1^*+c_n-(n-1))}{2} 
\ \leq \ 
\frac{c_1^*+c_n}{2}
\ \leq \ 
- \frac{(n+1)}{2} -1 + \frac{(-c_1^*+c_n-(n-1))}{2} 
$$
This simplifies to 
$$
c_1^*-c_n \ \leq \ c_1^* + c_n \ \leq \ -c_1^* + c_n - 2n - 2. 
$$
Whence we get: 
$$
c_n \geq 0, \quad c_1^* \leq -n-1.
$$
This is exactly the condition that $w^{-1} = (1, s_n^*)$ under the twisted action makes $\mu+\mu'$ dominant. (See the last row of \eqref{eqn:twisted-action-n'=1}.)

\bigskip
{\bf Case n} is similar; we have $\ell(\mu,\mu') = c_n^*-c_1 - (n-1).$ The inequalities in (2) of the lemma read:
$$
- \frac{(n+1)}{2} + 1 - \frac{(c_n^*-c_1 - (n-1))}{2} 
\ \leq \ 
\frac{c_n^*+c_1}{2}
\ \leq \ 
- \frac{(n+1)}{2} -1 + \frac{(c_n^*-c_1 - (n-1))}{2} 
$$
This simplifies to 
$$
-c_n^*+c_1  \ \leq \ c_n^* + c_1 \ \leq \ c_n^* - c_1 - 2n - 2. 
$$
Whence we get: 
$$
c_n^* \geq 0, \quad c_1 \leq -n-1.
$$
This is exactly the condition that $w^{-1} = (s_n, 1^*)$ makes $\mu+\mu'$ dominant. 

\bigskip
{\bf Case j} breaks up into two sub-cases: 
\begin{enumerate}
\item[{\bf Case j1}:] $c_j^*-c_{n-j+1} + (n-2j+1) \geq c_{n-j} - c_{j+1}^* - (n-2j-1).$ 
\smallskip
\item[{\bf Case j2}:] $c_j^*-c_{n-j+1} + (n-2j+1) < c_{n-j} - c_{j+1}^* - (n-2j-1).$
\end{enumerate}

\medskip
For {\bf j1}, we have $\ell(\mu,\mu') = c_{n-j} - c_{j+1}^* - (n-2j-1)$ and the inequalities of (2) read: 
\begin{multline*}
- \frac{(n+1)}{2} + 1 - \frac{(c_{n-j} - c_{j+1}^* - (n-2j-1))}{2} 
\ \leq 
\frac{c_{n-j} + c_{j+1}^*}{2}, \ {\rm and} \\
\frac{c_{n-j} + c_{j+1}^*}{2} \ \leq \ 
- \frac{(n+1)}{2} -1 + \frac{(c_{n-j} - c_{j+1}^* - (n-2j-1))}{2}. 
\end{multline*}
These simplify to: 
$$
- c_{n-j} + c_{j+1}^* - 2j
\ \leq \ 
c_{n-j} + c_{j+1}^*
\ \leq \ 
c_{n-j} - c_{j+1}^* - 2n+2j-2. 
$$
This in turn implies that: 
$$
c_{n-j} \ \geq \ -j, \quad c_{j+1}^* \ \leq \ -n+j-1.
$$
Next, we see that the defining inequalities of {\bf j1} gives in particular that
\begin{equation}
\label{eqn:inequalities-case-j}
c_j^* + c_{j+1}^*  + 2n-4j \ \geq \ c_{n-j} + c_{n-j+1}. 
\end{equation}
Add $c_{n-j+1}$ on both sides of \eqref{eqn:inequalities-case-j} to get 
$$
c_{n-j+1} + c_j^* + c_{j+1}^*  + 2n-4j \ \geq \ c_{n-j} + 2c_{n-j+1}, 
$$
and applying purity, we can rewrite this as
$$
c_{n-j} + 2c_{j+1}^*+ 2n-4j \ \geq \ c_{n-j} + 2c_{n-j+1},
$$
whence
$$
c_{n-j+1} \ \leq \ c_{j+1}^*+ n-2j \ \leq \ -j-1. 
$$
Next, add $c_j^*$ to both sides of \eqref{eqn:inequalities-case-j} to get 
$$
2c_j^* + c_{j+1}^* + 2n-4j \ \geq \ c_{n-j} + c_{n-j+1} + c_j^*, 
$$
and applying purity, we can rewrite this as
$$
2c_j^* + c_{j+1}^*+ 2n-4j \ \geq \ 2c_{n-j} + c_{j+1}^*
$$
whence,
$$
c_j^* \ \geq \ c_{n-j} - n + 2j \ \geq \ -n+j. 
$$
Putting all this together, we get the following inequalities:
$$
c_{n-j} \geq -j, \quad c_{n-j+1} \leq -j-1, \quad {\rm and} \quad c_j^* \geq j-n, \quad c_{j+1}^* \leq -n+j-1.
$$

\medskip

For {\bf j2}, we have $\ell(\mu,\mu') = c_j^* - c_{n-j+1} + (n-2j+1)$ and the inequalities of (2) simplifying to:
$$
c_j^* \geq -n+j  \quad {\rm and} \quad c_{n-j+1} \leq -j-1. 
$$
The defining inequalities of {\bf j2} may be written as: 
\begin{equation}
\label{eqn:inequalities-case-j2}
c_j^* + c_{j+1}^* \ \leq \  c_{n-j} + c_{n-j+1} -2n+4j. 
\end{equation}
Add $c_{j+1}^*$ to both sides of \eqref{eqn:inequalities-case-j2}, apply purity to right hand side, and simplify to get:
$$
c_{j+1}^* \ \leq \ -n+j-1. 
$$
Next, add $c_{n-j}$ to both sides of \eqref{eqn:inequalities-case-j2}, apply purity to left hand side, and simplify to get:
$$
c_{n-j} \ \geq \ -j. 
$$
Putting all this together, we see exactly as in {\bf Case j1} that 
$$
c_{n-j} \geq -j, \quad c_{n-j+1} \leq -j-1, \quad {\rm and} \ 
c_j^* \geq -n+j, \quad c_{j+1}^* \leq -n+j-1. 
$$

\medskip 

Using the table \eqref{eqn:twisted-action-n'=1} we see that 
$$
c_{n-j} \geq -j, \quad c_{n-j+1} \leq -j-1
 \ \iff \ 
s_j \cdot (\mu_1,\dots,\mu_n,b) \ \mbox{is dominant}. 
$$
$$
c_j^* \geq j-n, \quad c_{j+1}^* \leq -n+j-1
 \ \iff \ 
s_{n-j}^* \cdot (\mu_1^*,\dots,\mu_n^*,b^*) \ \mbox{is dominant}. 
$$
So, in {\bf Case j}, the required balanced Kostant representative is the inverse of $(s_j, s_{n-j}^*).$

\medskip

Conversely, if $w^{-1} = (s_j, s_{n-j}^*)$ makes $(\mu+\mu')$ dominant then we just argue backwards in the above paragraphs to see that inequalities of (2) are satisfied. Thus far, we have proved (2) $\iff$ (3) when $F$ is imaginary quadratic. 

\bigskip

\paragraph{\bf A general totally imaginary field}
\label{sec:general-cm-field}

Now let $F$ be any totally imaginary field. For each $v \in \place_\infty$ we have a pair of complex embeddings 
$\{\eta_v, \bar\eta_v\}$ of $F$. For any such embedding $\eta$, the weight $\mu$, has a $\eta$-component $\mu^\eta = (\mu^\eta_1,\dots,\mu^\eta_n)$ which is a non-increasing sequence of integers, and similarly, $\mu'^\eta = (b^\eta)$ is just an integer. Define $c_j^\eta = \mu^\eta_j - b^\eta$. The abelian width is given by:
$a(\mu, \mu') \ = \ \frac{c_j^\eta + c_{n-j+1}^{\bar\eta}}{2},$ 
for any $j$ and any $\eta.$ 
For $v \in \place_\infty$, define $\ell_v(\mu,\mu')$ as the minimum of the absolute 
values of the following $n$ integers: 
$$
c^{\bar\eta_v}_1- c^{\eta_v}_n + (n-1) \  > \ c^{\bar\eta_v}_2-c_{n-1}^{\eta_v}+(n-3) \ > \ \cdots \  > c^{\bar\eta_v}_n - c_1^{\eta_v} - (n-1). 
$$
Then
$\ell(\mu,\mu') = \min\{\ell_v(\mu,\mu') \, : \, v \in \place_\infty \}.$
The inequalities of (2) imply that for each $v \in \place_\infty$ we have 
\begin{equation}
\label{eqn:(2)-v}
- \frac{N}{2} + 1 - \frac{\ell_v(\mu,\mu')}{2}
 \ \leq \ a(\mu, \mu')  \ \leq \  
- \frac{N}{2} -1 + \frac{\ell_v(\mu,\mu')}{2}. 
\end{equation}
Using the same argument as in the imaginary quadratic case, we see that there exists  
$
w_v \ = \ (w_{\eta_v}, w_{\bar{\eta}_v}) \ \in \ W^{(P_0 \times_{\eta_v} \C) \times  (P_0 \times_{\bar\eta_v} \C)}
$ 
such that 
$w_v^{-1}\cdot((\mu^{\eta_v}, \mu'^{\eta_v}), (\mu^{\bar\eta_v}, \mu'^{\bar\eta_v}))$ is dominant and $l(w_{\eta_v}) + l(w_{\bar{\eta}_v}) =  n.$ The required balanced Kostant representative then is
$w = (w_v)_{v \in \place_\infty}$; hence (3) is satisfied. 
Conversely, if (3) holds, then writing $w = (w^\eta)$ as $w = (w_v)$ with $w_v = (w_{\eta_v}, w_{\bar\eta_v})$, we see that 
$w_v^{-1}\cdot((\mu^{\eta_v}, \mu'^{\eta_v}), (\mu^{\bar\eta_v}, \mu'^{\bar\eta_v}))$ is dominant, and working backwards as in the imaginary quadratic case, we deduce
\eqref{eqn:(2)-v} holds for each $v$, and hence (2) holds. 
\hfill$\Box$

\medskip
\subsubsection{\bf Proof of (2) $\iff$ (3) in the general case}

First of all, we will prove it in the special case when $F$ is imaginary quadratic, i.e., $\r = 1.$ 

\medskip
\paragraph{\bf Parametrizing Kostant representatives} 
We will need explicit Kostant representatives. Recall, that $G_0 = \GL_N$ and 
$P_0 = M_{P_0} U_{P_0}$ the standard $(n, n')$-parabolic subgroup of $G_0$, where $N = n+n';$ clearly, $\dim(U_{P_0}) = nn'.$ 
Then $W_{G_0} = \perm_N$ the permutation group on $N$ letters, and $W_{M_{P_0}} = \perm_n \times \perm_{n'}.$ The set of Kostant representatives $W^{P_0}$ may be described as: 
\begin{equation}
\label{eqn:W^P-defn}
W^{P_0} \ = \ \{w \in W_{G_0} \ : \ w^{-1}(1) < \cdots < w^{-1}(n) \ {\rm and} \ w^{-1}(n+1) < \cdots < w^{-1}(N) \}.
\end{equation}
The set $W^{P_0}$ is in bijection with the set of all $n$-tuples $\kappa = (k_1,\dots, k_n)$ where 
$1 \leq k_1 < \cdots < k_n \leq N$. Any such $\kappa$ corresponds to $w_\kappa \in W^{P_0}$ which is uniquely defined by the conditions: 
\begin{equation}
\label{eqn:w-kappa}
w_\kappa^{-1}(1) = k_1, \, \dots, w_\kappa^{-1}(n) = k_n. 
\end{equation}
If $\kappa = (1,2,\dots, n)$ then $w_\kappa$ is the identity element. 
There is a self-bijection $W^{P_0} \to W^{P_0}$ defined by $w_\kappa \mapsto  w_{\kappa^\v}$, where 
 \begin{equation}
\label{eqn:kappa-v}
\kappa^\v \ :=  \ N+1-k_n \ < \ \cdots \ < N+1 - k_1; \quad \kappa^\v_j = N+1 - k_{n-j+1}.
\end{equation}
Let $w_N = w_{G_0} \in W_{G_0}$ denote the element of longest length, which is given by $w_N(j) = N+1-j$ for any $1 \leq j \leq N$; clearly, $w_{G_0}^2 = 1.$ 
Similarly, $w_n$ and $w_{n'}$ are defined, and  we have $w_{M_{P_0}} = w_n \times w_{n'}$.

\smallskip

\begin{lemma}
\label{lem:length-w-kappa}
With the notations as above, we have: 
\begin{enumerate}
\item $l(w_\kappa) = (k_1-1) + (k_2-2) + \dots + (k_n-n).$
\smallskip
\item $l(w_\kappa) + l(w_{\kappa^\v}) = nn' = \dim(U_{P_0}).$
\smallskip
\item $w_{\kappa^\v} = w_{M_{P_0}} w_\kappa w_{G_0}.$
\end{enumerate}
\end{lemma}

\begin{proof}
Clearly, $l(w_\kappa) = l(w_\kappa^{-1})$, and 
for counting the length of $w_\kappa^{-1}$, count the number of its shuffles, i.e., count the number of pairs $(i,j)$ with $1 \leq i < j \leq N$ with $w_\kappa^{-1}(i) > w_\kappa^{-1}(j).$ But for 
any such shuffle, by \eqref{eqn:W^P-defn}, it is clear that $1 \leq i \leq n$ and $n+1 \leq j \leq N$. We leave it to the reader to see that for a fixed $i \leq n$, the number of shuffles $(i,j)$ is $k_i - i$. 
Also, (2) follows from Statement (1) and Equation \eqref{eqn:kappa-v}. To see the validity of (3), compute the inverses of both sides on any $1 \leq j \leq n$:
$$
(w_{G_0}  w_\kappa^{-1} w_{M_{P_0}})(j) = ( w_{G_0} w_\kappa^{-1})(n+1-j) = w_{G_0}(k_{n+1-j}) = N+1-k_{n+1-j} = \kappa^\v_j = w_{\kappa^\v}^{-1}(j). 
$$
\end{proof}

\medskip
\paragraph{\bf Twisted action of $W^{P_0}$ on weights} 

The usual permutation action of $\sigma \in S_m$ on an $m$-tuple is given by:  
$\sigma(t_1,\dots, t_m) \ = \ (t_{\sigma^{-1}(1)}, \dots, t_{\sigma^{-1}(m)}).$ If $\underline{t} := (t_1,\dots, t_m)$, then 
the twisted action of $\sigma$ on $\underline{t}$ is defined by: 
$\sigma \cdot \underline{t} = \sigma(\underline{t}+ \bfgreek{rho}_m) - \bfgreek{rho}_m,$ which unravels to 
$$
\sigma \cdot (t_1,\dots, t_m) \ = \ 
(t_{\sigma^{-1}(1)} + 1 - \sigma^{-1}(1), \, t_{\sigma^{-1}(2)} + 2 - \sigma^{-1}(2), \, \dots, \, t_{\sigma^{-1}(m)}+ m - \sigma^{-1}(m)). 
$$ 
Now, keeping the combinatorial lemma in mind, suppose 
$$
\mu = ((b_1,\dots,b_n), \, (c_1,\dots, c_n)), \quad \mu' = ((b'_1,\dots,b'_{n'}), \, (c'_1,\dots, c'_{n'})), 
$$
where each $n$-tuple or $n'$-tuple is a non-increasing string of integers satisfying the purity condition: 
$\w = b_i + c_{n-i+1}, \ \w' = b'_j + c'_{n'-j+1}.$
We are seeking a Kostant representative of optimal length that `straightens out' 
$$
\mu+\mu' \ = \ ((b_1,\dots,b_n, b'_1,\dots,b'_{n'}), \, (c_1,\dots, c_n, c'_1,\dots, c'_{n'})). 
$$
For this, we need the twisted action of $w_\kappa^{-1}$ on an $(n+n')$-tuple like $(b_1,\dots,b_n, b'_1,\dots,b'_{n'})$. Given $\kappa$, let us define its complement $\kappa^c$ as 
the ordered string of integers: 
$$
\kappa^c \ := \ k^c_1 < \cdots < k^c_{n'} \ := \ \{1,2,\dots, N\} \setminus \{k_1, k_2, \dots, k_n\}.
$$
It is  useful to note that 
$$
\kappa^c \ = \ 
\{1, 2, \dots, k_1-1,  k_1+1, \dots , k_2-1, k_2+1, \dots, k_n-1,  k_n+1, \dots, N \}. 
$$
The element $w_\kappa^{-1} \in W^{P_0}$ is the permutation that may be written as: 
$$
\begin{pmatrix}%{lllllll}
1 & 2 & \dots & n & n+1 & \dots & N \\ 
k_1 & k_2 & \dots & k_n & k^c_1 & \dots & k^c_{n'}
\end{pmatrix}, 
$$
and the permutation $w_\kappa$ is: 
\begin{multline*}
\left(\begin{matrix}
1    & \dots  & k_1-1            &  k_1  &  k_1+1  & \dots   &  k_2-1       &  k_2  & k_2+1  &\dots \\
n+1 & \dots  & n+ k_{1}-1    &  1      &  n+k_1  & \dots   &  n+k_2-2    &  2     & n+k_2-1 & \dots  
\end{matrix}\right. \\ 
\left.
\begin{matrix}
\dots & k_{n-1} & k_{n-1}+1 & \dots   &  k_n-1       &  k_n  & k_n+1  & \dots &  N\\
\dots & n-1 &  k_{n-1}+2 & \dots   &  k_n    &  n     & k_n+1 & \dots & N
\end{matrix} \right). 
\end{multline*}
(The reader should pay some attention to the special cases $k_1 = 1$ and $k_n = N$.) 
Denoting 
$$
(b_1,\dots,b_n, b'_1,\dots,b'_{n'}) \ = \ (d_1,\dots,d_n, d_{n+1}, \dots, d_N),
$$ 
we have: 
\begin{multline}
\label{eqn:twisted-action-w-kappa}
w_\kappa^{-1} \cdot (d_1,\dots,d_n, d_{n+1}, \dots, d_N) \ = \\ 
(d_{w_\kappa(1)}+1- w_\kappa(1), \ d_{w_\kappa(2)}+2- w_\kappa(2), \ \dots \ ,  d_{w_\kappa(N)}+N- w_\kappa(N)). 
\end{multline}

\medskip
\paragraph{\bf Dominance of $w_\kappa^{-1} \cdot (d_1,\dots,d_N)$}
Let us enumerate the inequalities that guarantee dominance of the weight in \eqref{eqn:twisted-action-w-kappa}: 

\medskip
\begin{prop}
\label{prop:w-kappa-dom}
The weight $w_\kappa^{-1} \cdot (d_1,\dots,d_n, d_{n+1}, \dots, d_N)$ is dominant if and only if the following conditions are satisfied: 
\medskip
\begin{enumerate}
\item[(0)] If $k_1 - 1 \geq 1$ then 
$$b'_{k_1-1} - b_1 \ \geq \ n+k_1-1.$$ 
If $k_1=1$ then there is no such condition. 
\medskip 
\item[(1)] If $k_2 \geq k_1+2$ then 
  \begin{enumerate}
  \item[(i)] $$ b_1-b'_{k_1} \ \geq \ -n-k_1+2,$$ and 
  \item[(ii)]  $$ b'_{k_2-2} - b_2 \ \geq \ n+ k_2-3.$$ 
  \end{enumerate}
If $k_2 = k_1+1$ then there are no such conditions. 
\medskip
\item[] $$ \vdots$$ 

\medskip
\item[$(l)$] ($1 \leq l \leq n-1$) If $k_{l+1} \geq k_l+2$ then 
  \begin{enumerate}
  \item[(i)] $$ b_l - b'_{k_l+1-l} \ \geq \  -n -k_l +2l,$$ and 
  \item[(ii)]  $$ b'_{k_{l+1}-l-1} - b_{l+1} \ \geq \ n + k_{l+1} -2l -1.$$ 
  \end{enumerate}
If $k_{l+1} = k_l+1$ then there are no such conditions. 

\medskip
\item[] $$ \vdots$$ 

\medskip
\item[($n-1$)] If $k_n \ \geq \ k_{n-1}+2$ then 
  \begin{enumerate}
  \item[(i)] $$ b_{n-1} - b'_{k_{n-1}+2-n} \ \geq \  n- k_{n-1}-2,$$ and 
  \item[(ii)]  $$b'_{k_n-n} - b_n \ \geq \ -n + k_n + 1.$$ 
  \end{enumerate}
If $k_n = k_{n-1}+1$ then there are no such conditions. 

\medskip
\item[$(n)$] If $k_n \leq N-1$ then 
$$
b_n - b'_{k_n+1-n} \ \geq \ n-k_n 
$$
If $k_n = N$ then there is no such condition. 
\end{enumerate}
In the above $n+1$ conditions, some of them might be empty, however not all can be empty.
\end{prop}

\begin{proof}
The tedious argument has the same flavour for each case $(1), (2), \dots (l), \dots (n-1), (n)$; as a representative, let us verify (1). 
 If $k_2 \geq k_1+2$ then looking at the relevant part of $w_\kappa$: 
$$
\begin{pmatrix}
\dots   &  k_1  &  k_1+1  & \dots   &  k_2-1       &  k_2 & \dots  \\
\dots  &  1      &  n+k_1  & \dots   &  n+k_2-2    &  2   & \dots 
\end{pmatrix} 
$$ 
we will have two dominance conditions: comparing entries at steps $k_1$ and $k_1+1$ gives
\begin{equation}
\label{eqn:step-k-1}
d_{w_\kappa(k_1)} + k_1 - w_\kappa(k_1) \ \geq \ d_{w_\kappa(k_1+1)}+ k_1+1- w_\kappa(k_1+ 1)
\end{equation}
and, similarly, comparing entries at steps $k_2-1$ and $k_2$ gives
\begin{equation}
\label{eqn:step-k-2-1}
d_{w_\kappa(k_2-1)} + k_2-1 - w_\kappa(k_2-1) \ \geq \ d_{w_\kappa(k_2)}+ k_2- w_\kappa(k_2).
\end{equation}
Now, \eqref{eqn:step-k-1} unravels to $b_1 + k_1-1 \geq b'_{k_1}+1-n$ which is (1)(i), and similarly, \eqref{eqn:step-k-2-1} unravels to 
$b'_{k_2-2}+1-n \geq b_2+k_2-2$ which is (1)(ii). 
However, if $k_2=k_1+1$, then the corresponding part of the permutation $w_\kappa$ just collapses to 
$$
\begin{pmatrix}
\dots   &  k_1   &  k_2 & \dots  \\
\dots  &  1       &  2   & \dots 
\end{pmatrix} 
$$ 
and dominance is assured since $b_1 \geq b_2$.  
\end{proof}

\medskip
\begin{prop}
\label{prop:w-kappa-v-dom}
The weight $w_{\kappa^\v}^{-1} \cdot (c_1,\dots,c_n, c'_1, \dots, c'_{n'})$ is dominant if and only if the following conditions are satisfied: 
\medskip
\begin{enumerate}
\item[($0^\v$)] If $k^\v_1 - 1 \geq 1$ then 
$$
b_n - b'_{k_n+1-n} \ \geq \ n-k_n + (N + (\w-\w')).
$$
If $k^\v_1=1$ then there is no such condition. 

\medskip 
\item[($1^\v$)]  If $k^\v_2 \geq k^\v_1+2$ then 
  \begin{enumerate}
  \item[$(i)^\v$]   $$b'_{k_n-n} - b_n \ \geq \ -n + k_n +1 - (N + (\w-\w')),$$  and 
  \item[$(ii)^\v$]  $$ b_{n-1} - b'_{k_{n-1}+2-n} \ \geq \  n- k_{n-1}-2 + (N + (\w-\w')).$$
\end{enumerate}
If $k^\v_2 = k^\v_1+1$ then there are no such conditions.

\medskip
\item[] $$ \vdots$$ 

\medskip 
\item[($l^\v$)]  If $k^\v_{l+1} \geq k^\v_l+2$ then 
  \begin{enumerate}
  \item[$(i)^\v$]   $$b'_{k_{n-l+1} -n +l-1} - b_{n-l+1} \ \geq \ k_{n-l+1} -n +(2l-1) - (N + (\w-\w')),$$  and 
  \item[$(ii)^\v$]  $$ b_{n- l} - b'_{k_{n- l}+ 1 -n +l} \ \geq \  - k_{n-l} +n -2l + (N + (\w-\w')).$$
\end{enumerate}
\smallskip
If $k^\v_{l+1} = k^\v_l+1$ then there are no such conditions.

\medskip
\item[] $$ \vdots$$ 

\medskip
\item[($(n-1)^\v$)] If $k^\v_n \ \geq \ k^\v_{n-1}+2$ then 
  \begin{enumerate}
  \item[$(i)^\v$] $$ b'_{k_2-2} - b_2 \ \geq \ n+ k_2-3 - (N + (\w-\w')),$$ and  
  \item[$(ii)^\v$]  $$ b_1-b'_{k_1} \ \geq \ -n-k_1+2 + (N + (\w-\w')).$$ 
  \end{enumerate}
If $k^\v_n = k^\v_{n-1}+1$ then there are no such conditions. 

\medskip
\item[$(n^\v)$] If $k^\v_n \leq N-1$ then 
$$
b'_{k_1-1} - b_1 \ \geq \ n+k_1-1 - (N + (\w-\w')).
$$ 
If $k^\v_n = N$ then there is no such condition. 
\end{enumerate}
\end{prop}

\begin{proof}
Apply Prop.\,\ref{prop:w-kappa-dom} while replacing 
\smallskip
\begin{itemize}
\item $k_j$ by $k^\v_j = N+1-k_{n+1-j},$  
\smallskip
\item $b_j$ by $c_j = \w - b_{n+1-j},$ and 
\smallskip
\item $b'_j$ by $c'_j = \w' - b'_{n'-j+1}$,  
\end{itemize}
As an illustrative example, let us make these replacements in case $(1)(i)$ of Prop.\,\ref{prop:w-kappa-dom}, then we get 
$$
c_1 - c'_{k_1^\v} \geq -n-k_1^\v+2 \quad \iff \quad 
c_1 - c'_{N+1-k_n} \geq -n+k_n+1 -N, 
$$
which may be written as 
$$
(\w-b_n) - (\w-b'_{k_n-n}) \geq -n+k_n+1 -N \quad \iff \quad b'_{k_n-n} - b_n \geq -n+k_n+1 - (N + (\w-\w')),
$$
giving us case $(1)^\v(i)^\v$. Similarly, all the other cases may be verified. 
\end{proof}

\medskip
\begin{rem}
\label{rem:compare-w-k-w-kv}
{\rm 
Let us  note the following `duality' relations between the various cases of Prop.\,\ref{prop:w-kappa-dom} and Prop.\,\ref{prop:w-kappa-v-dom}. 
\begin{itemize}
\item $k^\v_1 = 1 \iff k_n=N$. \\ 
(Compare $(0)^\v$ of Prop.\,\ref{prop:w-kappa-v-dom} with $(n)$ of Prop.\,\ref{prop:w-kappa-dom}.) 
\item $k^\v_n = N \iff k_1=1$.  \\ 
(Compare $(n)^\v$ with $(0)$.) 
\item $k^\v_j \geq k^\v_{j-1} + 2 \iff k_{n+2-j} \geq k_{n+1-j}+2,$ for $2 \leq j \leq n$. \\
 (Compare $(1)^\v(i)^\v$ with $(n-1)(ii)$ and $(1)^\v(ii)^\v$ with $(n-1)(i)$.)
\end{itemize}
In this comparison, an inequality of the form $b_i - b'_j \geq \beta$ in Prop.\,\ref{prop:w-kappa-dom} corresponds to  
$b_i - b'_j \geq \beta + (N + (\w-\w'))$ in Prop.\,\ref{prop:w-kappa-v-dom}. Similarly, an inequality of the form 
$b'_j - b_i \geq \beta$ in Prop.\,\ref{prop:w-kappa-dom} corresponds to  
$b'_j - b_i \geq \beta - (N + (\w-\w'))$  in Prop.\,\ref{prop:w-kappa-v-dom}. 
}\end{rem}

\medskip
\paragraph{\bf The inner structure of the cuspidal width - I} 
For the weight $\mu$, written as above $\mu = ((b_1,\dots,b_n), \, (c_1,\dots, c_n))$, recall its cuspidal parameters from \eqref{eqn:cuspidal-parameters-alpha} and \eqref{eqn:cuspidal-parameters-beta}: 
$$
\alpha_i \ = \ -b_{n-i+1} + \tfrac{(n-2i+1)}{2} \quad 
\beta_i \ = \ -c_i - \tfrac{(n-2i+1)}{2}. 
$$ 
Similarly, for $\mu' = ((b'_1,\dots,b'_{n'}), \, (c'_1,\dots, c'_{n'}))$, we have 
$$
\alpha'_j \ = \ -b'_{n'-j+1} + \tfrac{(n'-2j+1)}{2}, \quad  
 \beta'_j \ = \ -c'_j - \tfrac{(n'-2j+1)}{2}. 
$$ 
For $1 \leq i \leq n$ and $1 \leq j \leq n'$, define $\ell_{i,j} := \alpha_i - \beta_i - \alpha'_j + \beta'_j.$
Applying purity we have: 
\begin{equation}
\label{eqn:lij}
\ell_{i,j} \ = \ 2(b'_{n'-j+1} - b_{n-i+1}) + (N+(\w - \w')) + -2n' + 2(j-i).
\end{equation}
These $nn'$ integers are ordered thus: 
\begin{equation}
\label{eqn:l_ij}
\begin{array}{ccccccc}
\ell_{1,1} & < & \ell_{1,2} & < & \cdots & < & \ell_{1,n'} \\
\vgreat & & \vgreat & & & & \vgreat \\
\ell_{2,1} & < & \ell_{2,2} & < & \cdots & < & \ell_{2,n'} \\ 
\vgreat & & \vgreat & & & & \vgreat \\
& \vdots & & & \vdots & &  \\ 
\vgreat & & \vgreat & & & & \vgreat \\
\ell_{n,1} & < & \ell_{n,2} & < & \cdots & < & \ell_{n,n'} \\
\end{array}
\end{equation}
Recall the cuspidal width is defined as 
$$
\ell(\mu,\mu') \ = \ \min\{|\ell_{i,j}| \ : \ 1 \leq i \leq n, \ 1 \leq j \leq n'\}.
$$
From \eqref{eqn:l_ij}, we see that the location of $0$ relative to these $nn'$ integers is important to determine the cuspidal width.

\medskip
\paragraph{\bf On how $\mu$ and $\mu'$ determine $\kappa$} 
Consider the $j$-th column of \eqref{eqn:l_ij}. Define $\ell_{0,j} = \infty$ (or a large positive integer), and  
$\ell_{n+1,j} = -\infty$ (or a large negative integer). For each $1 \leq j \leq n'$, define $r_j$ with $0 \leq r_j \leq n$ such that 
$$
\begin{array}{c}
\ell_{r_j,j} \\
\vgreater \\ 
0 \\ 
\vgreat \\
\ell_{r_j+1,j}.
\end{array}
$$ 
The integer $r_j$ defines the location of $0$ in the $j$-th column. For example, if all the $\ell_{*,j} \geq 0$ then $r_j = n,$ and similarly, if 
all $\ell_{*,j} < 0$ then $r_j = 0.$ Note that 
$$
0 \leq r_1 \leq r_2 \leq \cdots \leq r_{n'} \leq n.
$$
Next, define a string of integers $s_j$ by:
$s_j = r_j + j-1;$ then 
$$
0 \leq s_1 < s_2 < \cdots < s_{n'} \leq N-1.
$$
Now define $\kappa = k_1 < \cdots < k_n$ by: 
\begin{equation}
\label{eqn:kappa-defn}
\{k_1,\dots, k_n\} \ := \ \{1,2,\dots, N\} \setminus \{N-s_{n'}, N-s_{n'-1}, \dots, N-s_1\}.
\end{equation}

\medskip
\paragraph{\bf The inner structure of the cuspidal width - II}
Suppose there are $p$ strict inequalities in the sequence $r_1 \leq r_2 \leq \cdots \leq r_{n'}$, i.e, we have: 
$$
r_1 = \cdots = r_{t_1}  \ <  \ r_{t_1+1} = \cdots = r_{t_2}  \ < \ \cdots \  = r_{t_p} < r_{t_p+1} = \cdots = r_{n'}. 
$$
Let us denote the common values thus: 
\begin{equation}
\label{eqn:r-(1)}
r^{(1)} \ := \ r_1 = \cdots = r_{t_1}, \quad 
r^{(2)} \ := \ r_{t_1+1} = \cdots = r_{t_2}, \quad \dots, \quad
r^{(p+1)} \ := \ r_{t_p+1} = \cdots = r_{n'}. 
\end{equation}
Note that $1 \leq t_1 < t_2 < \cdots < t_p < n'.$ 
Define the quantity: 
\begin{equation}
\label{eq:def-delta}
\delta \ := \ 2(p+1) - \delta(r_1,0) - \delta(r_{n'},n),
\end{equation}
where in the last two terms, $\delta(i,j) = 1$ if $i=j$ and $\delta(i,j) =0$ if $i \neq j.$ We have the following

\medskip
\begin{lemma}
The cuspidal width $\ell(\mu,\mu')$ is the minimum of the set  
$$
\L \ := \ 
\{\ell_{r^{(1)},\, 1}, \ 
-\ell_{r^{(1)}+1, \, t_1}, \ 
\ell_{r^{(2)}, \, t_1+1}, \ 
-\ell_{r^{(2)}+1, \, t_2}, \ \dots \ 
\ell_{r^{(p+1)}, \, t_p+1}, \ 
-\ell_{r^{(p+1)}+1, \, n'} \}
$$
with the understanding that 
\begin{itemize}
\item if $\delta(r_1,0) = 1$ then $r^{(1)}=0$ and we delete the term $\ell_{r^{(1)},\, 1}$ from $\L$, and similarly, 
\item if $\delta(r_{n'},n) = 1$ then $r^{(p+1)} = n$ and we delete the term $-\ell_{r^{(p+1)}+1, \, n'}$ from $\L.$
\end{itemize}
The cardinality of the set $\L$ is $\delta.$
\end{lemma}

\begin{proof}
This follows from \eqref{eqn:l_ij}; the cardinality of $\L$ follows from \eqref{eq:def-delta}. 
\end{proof}

\medskip
\paragraph{\bf The proof of the combinatorial lemma - I}

The proof of (2) $\iff$ (3) 
of the combinatorial lemma (for the case of an imaginary quadratic extension) follows from the following 

\medskip
\begin{prop}
\label{prop:simple-CL}
The following are equivalent: 
\begin{enumerate}
\item $-N+2-\ell(\mu,\mu') \ \leq \ (\w-\w') \ \leq \ -N-2+\ell(\mu,\mu').$ 
\smallskip
\item The element $w = (w_\kappa, w_{\kappa^\v})$ satisfies $w^{-1}\cdot(\mu+\mu')$ is dominant. 
\end{enumerate}
\end{prop}
Note that the requirement of the Kostant representative to be balanced is automatically taken care of by (2), since by Lem.\,\ref{lem:length-w-kappa},\,(2), 
we have $l(w) = l(w_\kappa) + l(w_{\kappa^\v}) = nn'.$  

\begin{proof}
The information contained in the inequalities 
$$
-N+2-\ell(\mu,\mu') \ \leq \ (\w-\w') \ \leq \ -N-2+\ell(\mu,\mu')
$$ 
is clearly equivalent to the set of $2\delta$ inequalities
\begin{equation}
\label{eqn:each-ell-bounds}
\ell \ \geq \ 2 + (N+(\w - \w')) \quad {\rm and} \quad \ell \geq 2 - (N+(\w - \w')), \quad \forall \ell \in \L.
\end{equation}

Let us begin the analysis of various cases and consider each of the above inequalities: 

\medskip

\begin{itemize}
\item Suppose $r_1 = 0.$  From \eqref{eqn:kappa-defn} it follows that $r_1 = 0 \iff k_n \leq N-1 \iff k_1^\v \geq 2.$ 
The condition $r_1 = \cdots = r_{t_1} = 0$  (which means the first $t_1$ many columns of \eqref{eqn:l_ij} are negative)
implies that $N-t_1+1,\dots, N-1, N$ are deleted in defining $\kappa$ in \eqref{eqn:kappa-defn}, 
hence $k_n = N-t_1.$ 
Now, consider the term 
$\ell = -\ell_{r^{(1)}+1, \, t_1} = -\ell_{1, N-k_n} \in \L.$ From \eqref{eqn:lij} we have 
$$
-\ell_{1,N-k_n} \ = \ 2(b_n - b'_{k_n+1-n}) - (N+(\w - \w')) + 2n' - 2(N-k_n-1).
$$
Applying \eqref{eqn:each-ell-bounds} to $-\ell_{1,N-k_n}$ gives us: 
$$
b_n - b'_{k_n+1-n} \ \geq \ n-k_n + (N+(\w - \w')), \quad {\rm and} \quad b_n - b'_{k_n+1-n} \ \geq \ n-k_n, 
$$
which are the same as the bounds in case-$(0^\v)$ of Prop.\,\ref{prop:w-kappa-v-dom} and case-$(n)$ of Prop.\,\ref{prop:w-kappa-dom}.

\medskip

\item Suppose $r_{n'} = n.$ From \eqref{eqn:kappa-defn} it follows that $r_1 = 0 \iff k_1 \geq 2 \iff k_n^\v \leq N-1.$ 
The condition $r_{t_p+1} = \cdots = r_{n'} = n$ 
(which means that in \eqref{eqn:l_ij} the last $t_p$-columns are all non-negative) 
implies that $1,2,\dots, N-(n+t_p)$ are deleted in getting $\kappa$ in \eqref{eqn:kappa-defn}, hence 
$k_1 = n'-t_p+1.$ Now, consider the term 
$\ell = -\ell_{r^{(p+1)}, \, t_p +1} = -\ell_{n, n'-k_1+2} \in \L.$ From \eqref{eqn:lij} we have 
$$
\ell_{n, n'-k_1+2} \ = \ 2(b'_{k_1-1} - b_1) + (N+(\w - \w')) - 2k_1-2n+4.
$$
Applying \eqref{eqn:each-ell-bounds} to $\ell_{n, n'-k_1+2}$ gives us: 
$$
b'_{k_1-1} - b_1 \ \geq \ n+k_1-1 , \quad {\rm and} \quad b'_{k_1-1} - b_1 \ \geq \ n+k_1 -1 + (N+(\w - \w')), 
$$
which are exactly the bounds described case-$(0)$ of Prop.\,\ref{prop:w-kappa-dom} and case-$(n)^\v$ of Prop.\,\ref{prop:w-kappa-v-dom}. 

\medskip

\item Suppose $r_1 \geq 1.$ Then the shape of $\kappa$ is of the form:
$$
\kappa \ = \ \{ \dots, N-r_1-t_1, \widehat{N-r_1+ - t_1}, \dots, \widehat{N - r_1}, N-r_1+1, \dots, N-1, N \}, 
$$
where the $\widehat{a}$ means that $a$ is deleted from that list. This implies that 
$$
k_n = N, \ k_{n-1} = N-1, \ \dots, k_{n-r_1+1} = N-r_1+1, \ k_{n-r_1} = N-r_1-t_1, \dots
$$
Hence we see that 
$$
\mbox{if $\ l \ := \ n - r_1 \ $ then  $ \ k_l = N-r_1-t_1, \ \ k_{l+1} = N-r_1+1.$}
$$
In particular, $k_{l+1} - k_l = 1+t_1 \geq 2.$ 
Put $l^\v = n-l+1$ then, by definition of $\kappa^\v$ we also have $k^\v_{l^\v}  - k^\v_{l^\v-1} \geq 2.$ 
Note that $l^\v = n-(n-r_1)+1 = r_1+1.$ Hence we have $k^\v_{r_1+1} - k^\v_{r_1} \geq 2.$ 

\medskip

Consider the elements $\ell_{r_1, 1}$ and $-\ell_{r_1+1, t_1}$ in $\L.$ 
Note that 
$$
\ell_{r_1,1} = 2(b'_{n'} - b_{n-r_1+1}) + (N + (\w-\w')) - 2n' + 2(1-r_1).
$$
If we apply \eqref{eqn:each-ell-bounds} to $\ell_{r_1,1}$ we get 
$$
b'_{n'} - b_{n-r_1+1} \geq n'+r_1 \quad {\rm and} \quad 
b'_{n'} - b_{n-r_1+1} \geq n'+r_1 - (N+(\w-\w')). 
$$
We will leave it to the reader check that these are exactly the inequalities we get from 
case-$(l)(ii)$ of Prop.\,\ref{prop:w-kappa-dom} and case-$(r_1)^\v(i)^\v$ of Prop.\,\ref{prop:w-kappa-v-dom}.  

\medskip

Next, note that 
$$
-\ell_{r_1+1, t_1} = -2(b'_{n'-t_1+1} - b_{n-r_1}) - (N + (\w-\w')) + 2n' - 2(t_1-r_1-1).
$$
Apply \eqref{eqn:each-ell-bounds} to $-\ell_{r_1+1, t_1}$ to get 
$$
b_{n-r_1} - b'_{n'-t_1+1} \geq  -n' + t_1 - r_1 
$$ 
and 
$$ 
b_{n-r_1} - b'_{n'-t_1+1}  \geq -n' + t_1 - r_1 + (Nr_1+(\w-\w')). 
$$
We will leave it to the reader check that these are exactly the inequalities we get from 
case-$(l)(i)$ of Prop.\,\ref{prop:w-kappa-dom} and case-$(r_1)^\v(ii)^\v$ of Prop.\,\ref{prop:w-kappa-v-dom}.  

\medskip

Let us summarize the above three cases as: 
\begin{enumerate}
\item If $r_1 = 0$ then  $(n)$ and $(0^\v)$ hold. 
\item If $r_{n'} = n$ then $(0)$ and $(n)^\v$ hold. 
\item If $r_1 \geq 1$ then $(n-r_1)(i),$ $(n-r_1)(ii)$, $(r_1^\v)(i)^\v$ and $(r_1^\v)(ii)^\v$ hold. (Furthermore, cases 
$(1)^\v$ through $(r_1-1)^\v$ are empty and $(n-r_1+1)$ through $(n)$ are empty.) 
\end{enumerate}

\medskip

\item It should be clear now, that for each $q$ with $1 \leq q \leq p$, using $t_q$ or $r^{(q)}$ as the anchor, we get all the cases of Prop.\,\ref{prop:w-kappa-dom} and Prop.\,\ref{prop:w-kappa-v-dom}, 
and hence $w^{-1}\cdot(\mu+\mu')$ is dominant. 
\end{itemize}

The entire argument is reversible, i.e., if the cases of Prop.\,\ref{prop:w-kappa-dom} and Prop.\,\ref{prop:w-kappa-v-dom} hold the inequalities in \eqref{eqn:each-ell-bounds} are satisfied. 
This completes the proof of Prop.\,\ref{prop:simple-CL}. 
\end{proof}

\medskip
\paragraph{\bf The general totally imaginary field}
Now if $F$ is any totally imaginary field, then the proof reduces to working with pairs of complex embeddings $(\eta_v, \bar\eta_v)$ for 
a $v \in \place_\infty$; it is entirely analogous to Sect.\,\ref{sec:general-cm-field}. We will leave the details to the reader. 

\medskip
\subsubsection{\bf The combinatorial lemma at an arithmetic level} 
\label{sec:comb-lem-arith-level}
All the three statements in the combinatorial lemma work at an arithmetic level. Take $\mu \in X^+_{00}(T_n \times E)$ and 
$\mu' \in X^+_{00}(T_{n'} \times E)$, and $\sigma_f \in \Coh_{!!}(G_n, \mu),$  $\sigma'_f \in \Coh_{!!}(G_{n'}, \mu'),$ and for 
$\iota : E \to \C$, consider the statement of the lemma for ${}^\iota\mu, {}^\iota\mu', {}^\iota\sigma$ and ${}^\iota\sigma'$; let us add some comments 
for each of (1), (2) and (3) of the lemma:  
\begin{enumerate}
\item From Sec.\,\ref{sec:crit-set-arith-level} it follows that $-\tfrac{N}{2}$ and $1-\tfrac{N}{2}$ are critical for 
$L(s, {}^\iota\sigma \times {}^\iota\sigma'^\v)$ for any $\iota : E \to \C.$ 

\smallskip
\item Since $\mu$ and $\mu'$ are strongly-pure, it is easy to see that the abelian width $a({}^\iota\mu, {}^\iota\mu')$ and the cuspidal width 
$\ell({}^\iota\mu, {}^\iota\mu')$ are independent of $\iota.$ (See Cor.\,\ref{cor:crit-set-ind-iota}.) 
For the assertion for cuspidal width, the reader may check from definitions that the 
$\ell({}^\iota\mu, {}^\iota\mu')$ is given by taking the minimum of 
$|-2\mu^{\iota\circ\tau}_{n-i+1}+2\mu'^{\iota\circ\tau}_{n'-j+1} + n-n' + 2j-2i + \w -\w'|$ over all $\tau : F \to E,$ and all indices 
$1 \leq i \leq n, \ 1 \leq j \leq n'.$ As $\tau$ varies over $\Hom(F,E)$, $\iota\circ\tau$ varies over $\Hom(F,\C),$ making the above minimum 
independent of $\iota.$

\smallskip
\item Write $w \in W^G$ as $w = (w^\tau)_{\tau: F \to E}.$ We will say $w \in W^P$ is balanced if $l(w^\tau) + l(w^{\overline{\tau}^\iota}) = \dim(U_{P_0})$ for all $\tau \in \Hom(F,E)$ and for all $\iota : E \to \C$; recall that $\iota$ induces a complex conjugation $\tau \mapsto {\overline{\tau}^\iota}$ on $\Hom(F,E)$. (See Rem.\,\ref{rem:w-base-change} below.) 
\end{enumerate}
It should now be clear that $(1) \iff (2) \iff (3)$ of the lemma is independent of $\iota : E \to \C.$

\medskip

\begin{rem}
\label{rem:w-base-change}{\rm 
Strongly-pure weights $\mu \in X^+_{00}(T_n \times E)$ and 
$\mu' \in X^+_{00}(T_{n'} \times E)$ being the base-change from $F_1$ (Prop.\,\ref{prop:strong-pure-weights-base-change}), it follows when the conditions of the combinatorial lemma hold,  
that the Kostant representative $w = (w^\tau)_{\tau : F \to E}$ is also the base-change from $F_1$ in the sense that if $\tau|_{F_1} = \tau'|_{F_1}$ then 
$w^{\tau} = w^{\tau'}.$ 
}\end{rem}

\medskip
\section{\bf Archimedean intertwining operator}
\label{sec:T-st-infinity}

\subsection{The case of $\GL_2$} 
The calculations in this subsection are in principle the same as in Harder~\cite[Sect.\,3.5]{harder-inventiones}, but we need to go through this exercise to reorganise our thoughts, while using inputs from \cite[Chap.\,9]{harder-raghuram-book}, so as to generalize them to $\GL_N$ in 
the next subsection. The main result of this subsection is Prop.\,\ref{prop:basic-Tst-GL2-rational-classes}.

\subsubsection{\bf Explicit cohomology class for $\GL_2$}
\label{sec:explcit-class-GL2-infinity}
Let $\mu = ((b_1,b_2), (c_1,c_2))$ be a pure dominant integral weight for $\GL_2(\C)$ as a real group. Integrality means $b_1,b_2,c_1,c_2 \in \Z$; dominance is $b_1 \geq b_2$ and $c_1 \geq c_2$; purity means 
$b_1+c_2 = b_2+c_1$, which allows us to define $m := b_1-b_2 +1 = c_1-c_2+1$. The cuspidal parameters are
$(\alpha_1,\alpha_2) = (-b_2+ \tfrac12,-b_1-\tfrac12)$ and $(\beta_1,\beta_2) = (-c_1-\tfrac12, -c_2+\tfrac12).$
We have the induced representation
$$
\J_\mu \ = \ \Ind_{B_2(\C)}^{\GL_2(\C)}\left(
z^{-b_2+ \tfrac12} \, \bar{z}^{-c_1-\tfrac12} \ \otimes \ z^{-b_1-\tfrac12} \, \bar{z}^{-c_2+\tfrac12} \right).
$$
Recall, $\GL_2(\C) = B_2(\C) \SU(2)$ with $T_c^{(1)} := B_2(\C) \cap \SU(2) \approx \SU(1) \approx \S^1.$ Let us write $e^{i\theta}$ for an element of $\S^1$ which is the element 
$\left(\begin{smallmatrix} e^{i\theta} & 0 \\ 0 & e^{-i\theta}\end{smallmatrix}\right)$ in $T_c^{(1)}$. 
If $(\tau_k,V_k)$ denotes the irreducible representation of $\SU(2)$ of dimension $k$, and $\chi_{2m}(e^{i \theta}) = e^{i (2m)\theta},$
then
\begin{equation}
\label{sec:J-mu-SU(2)}
\J_\mu \ = \ \Ind_{T_c^{(1)}}^{\SU(2)}
\left(\chi_{2m}\right) \ \approx \ 
\tau_{2m+1} \oplus \tau_{2m+3} \oplus \cdots \oplus \tau_{2m+2k+1} \oplus \cdots, 
\end{equation}
since by Frobenius reciprocity, any irreducible representation of $\SU(2)$ that appears in $\J_\mu$ has to contain the character $\chi_{2m}$ with multiplicity one. Note that $\tau_{2m+1}$ is the minimal $K$-type in 
the induced representation $\J_\mu;$ we denote $\J_\mu(\tau_{2m+1})$ for this minimal $K$-type as it sits inside the ambient $\J_\mu$. 
Let' us next describe $(\rho_\mu,\M_\mu)$ restricted to $\SU(2).$ We have 
$\M_{(b_1,b_2)} = \Sym^{b_1-b_2}(\C^2) \otimes \det^{b_2}$ as a representation of $\GL_2(\C),$ 
where $\C^2$ is the standard representation. Hence,  
$\rho_{(b_1,b_2)}|_{\SU(2)} = \tau_m.$  Similarly, $\rho_{(c_1,c_2)}|_{\SU(2)} = \tau_m.$  If $g \in \SU(2)$ then $\bar{g} = {}^t g^{-1}$, hence 
$\bar{\tau}_m = \tau_m^\vee = \tau_m.$ This implies that $\rho_\mu|_{\SU(2)} = \tau_m \otimes \tau_m.$  Recall Clebsch-Gordon for $\SU(2)$: for $p \geq q \geq 1$ we have
\begin{equation}
\label{sec:clebsch-gordon-SU(2)}
\tau_p \otimes \tau_q \ = \ \tau_{p-q+1} \oplus \tau_{p-q+3} \oplus \cdots \oplus \tau_{p+q-1}.
\end{equation}
Applying this to $p = q = m$ we get:
\begin{equation}
\label{sec:rho-mu-SU(2)}
\rho_\mu|_{\SU(2)} \ = \ 
\tau_1 \oplus \tau_3 \oplus \cdots \oplus \tau_{2m-1}.
\end{equation}
Denote $\M_\mu(\tau_{2m-1})$ for the copy of $\tau_{2m-1}$ as it sits inside $\M_\mu.$ 
Let $\g_2 := \gl_2(\C)$ and $\tilde\k_2 := \R \oplus \u_2(\C)$ be the Lie algebras of the connected real Lie groups $\GL_2(\C)$ and $Z_2(\C)\U(2),$ respectively. 
Then the Adjoint-action of $\SU(2)$ on $\g_2/\tilde\k_2$ is irreducible whose complexification is isomorphic to $\tau_3.$ Furthermore, we have:
$\wedge^0(\g_2/\tilde\k_2) \approx \wedge^3(\g_2/\tilde\k_2)  \approx \tau_1$ and 
$\wedge^1(\g_2/\tilde\k_2) \approx \wedge^2(\g_2/\tilde\k_2)  \approx \tau_3.$
We can now describe the complex $\Hom_{\SU(2)}(\wedge^\bullet(\g_2/\tilde\k_2), \J_\mu \otimes \M_\mu).$ Apply 
\eqref{sec:J-mu-SU(2)} and \eqref{sec:rho-mu-SU(2)} to $\J_\mu \otimes \M_\mu$ and then apply \eqref{sec:clebsch-gordon-SU(2)} to see that 
the smallest $p$ for which $\tau_p$ can occur in $\J_\mu \otimes \M_\mu$ is $p=3$ and this is realized exactly once as:
$$
\tau_3 \ \hookrightarrow \ \tau_{2m+1} \otimes \tau_{2m-1} \ = \ \J_\mu(\tau_{2m+1}) \otimes \M_\mu(\tau_{2m-1}). 
$$
Hence, $\Hom_{\SU(2)}(\wedge^\bullet(\g_2/\tilde\k_2), \J_\mu \otimes \M_\mu) \neq 0 \iff \bullet = 1, 2,$ and is one-dimensional in these degrees. 
Knowing that the differentials for this complex are zero, we deduce that $\J_\mu \otimes \M_\mu$ has nonvanishing $(\g_2, \tilde\k_2)$-cohomology 
only in degrees $1$ and $2$ and the cohomology group is one-dimensional in these degrees. 
Fix a basis $[\J_\mu]$ for: 
\begin{equation}
\label{eqn:[J-mu]}
H^1(\g_2, \tilde{\k}_2; \J_\mu \otimes \M_\mu) \ = \ 
H^1(\gl_2(\C), Z_2(\C)\U(2); \J_\mu \otimes \M_\mu) \ = \ \C [\J_\mu].
\end{equation}
Now, we express $[\J_\mu] \in 
\Hom_{Z_2(\C)\U(2)}(1\!\!1, \wedge^1(\g_2/\tilde\k_2)^* \otimes \J_\mu(\tau_{2m+1}) \otimes \M_\mu(\tau_{2m-1})),$ as 
$$
[\J_\mu] \ = \ 
\sum_{i, \alpha} X_i^* \otimes \phi_{i, \alpha} \otimes m_{\alpha}, 
$$
where $\{X_i^*\}$ is a basis for $(\g_2/\tilde\k_2)^*$, and $\{m_\alpha\}$ is a basis for $\M_\mu.$ (Of course, if $m_\alpha \notin 
\M_\mu(\tau_{2m-1})$ then $\phi_{i, \alpha} = 0.$) We call the finite set $\{\phi_{i, \alpha}\}$ of vectors in $\J_\mu$ as 
{\it cohomological vectors}. Since $H^1$ has dimension one, a scaling of the basis element $[\J_\mu]$ means jointly scaling this 
finite set of cohomological vectors. Furthermore, 
we contend, via an explicit version of Clebsch--Gordon, that one of 
the $\phi_{i, \alpha}$ is a highest weight vector of the lowest $K$-type $\J_\mu(\tau_{2m+1}).$ Call this particular vector 
as the {\it distinguished cohomological vector} for a given choice of $[\J_\mu]$.

\medskip
\subsubsection{\bf The highest weight vector of the lowest $K$-type in $\J_\mu$}
We can explicitly describe such a vector $f_\mu$; first of all, since $f_\mu$ is in the induced representation $\J_\mu$ we have 
\begin{equation}
\label{eqn:f-mu-1}
f_\mu(\begin{pmatrix}z & * \\ & w \end{pmatrix} g)
 \ = \ 
z^{-b_2+ \tfrac12} \, \bar{z}^{-c_1-\tfrac12} \cdot w^{-b_1-\tfrac12} \, \bar{w}^{-c_2+\tfrac12} \cdot \left|\frac{z}{w}\right|_\C^{1/2} f_\mu(g), 
\end{equation}
for all $g \in \GL_2(\C)$ and $z,w\in \C^\times.$ Next, we note:  
\begin{equation}
\label{eqn:f-mu-2}
f_\mu( 
\begin{pmatrix} e^{i\alpha} &  \\ & e^{-i\alpha} \end{pmatrix}
g \begin{pmatrix} e^{i\beta} &  \\ & e^{-i\beta} \end{pmatrix})  = \ 
e^{i(2m)\alpha}
e^{i(2m)\beta} f_\mu(g), 
\end{equation}
for all $g \in \GL_2(\C)$. The left-equivariance under $T_c^{(1)}$ is by \eqref{eqn:f-mu-1}, and the right-equivariance 
under $T_c^{(1)}$ is because of being the highest weight vector in $\tau_{2m+1}$. Finally, $f_\mu$ is completely determined 
by its values on $\SU(2)$, for which, observe that
$\SU(2) = T_c^{(1)} \cdot \SO(2) \cdot T_c^{(1)}.$
For the values of $f_\mu$ on $\SO(2)$, recall that the weight-vectors of $\tau_{2m+1}$ maybe enumerated as $\{f_{-2m}, f_{-2m+2}, \dots, f_{2m-2}, f_{2m}\}$ where 
$T_c^{(1)}$ acts on $f_k$ via the character $e^{i \beta} \mapsto e^{ik\beta}.$ So our $f_\mu$ is $f_{2m}$ up to a scalar multiple. 
Let $\r(\theta) = \left(\begin{smallmatrix} \cos(\theta) & -\sin(\theta) \\ \sin(\theta) & \cos(\theta)  \end{smallmatrix}\right)$; then the weight 
vectors $\{f_{-2m}, f_{-2m+2}, \dots, f_{2m-2}\}$ may be normalized so that: 
\begin{equation}
\label{eqn:SO(2)-on-f-2m}
\r(\theta) \cdot f_{2m} \ = \ 
\cos^{2m}(\theta)f_{2m} \ + \ \cos^{2m-2}(\theta)\sin^2(\theta) f_{2m-2} \ + \ \cdots \ + \ \sin^{2m}(\theta)f_{-2m}. 
\end{equation}
(Think of the model for $\tau_{2m+1}$ consisting of homogeneous polynomials of degree $2m$ in two variables.) 
Using the analogue of \eqref{eqn:f-mu-2} for the other weight vectors we see that $f_k(I) = 0$ if $k \neq 2m$. (Here $I$ is the $2 \times 2$ identity matrix.) 
Evaluating \eqref{eqn:SO(2)-on-f-2m} on $I$ we get: 
\begin{equation}
\label{eqn:f-mu-3}
f_\mu(\r(\theta)) \ = \ \cos^{2m}(\theta), 
\end{equation}
where, we have normalized $f_\mu$ by $f_{\mu}(I) = 1.$ Putting \eqref{eqn:f-mu-1}, \eqref{eqn:f-mu-2} and \eqref{eqn:f-mu-3} together we can write: 
\begin{equation}
\label{eqn:f-mu-4}
f_\mu(\begin{pmatrix}z & * \\ & w \end{pmatrix} \r(\theta) \begin{pmatrix} e^{i\beta} &  \\ & e^{-i\beta} \end{pmatrix}) \ = \ 
z^{-b_2+ 1} \, \bar{z}^{-c_1} \cdot w^{-b_1-1} \bar{w}^{-c_2} \cdot
\cos^{2m}(\theta) \cdot
e^{i(2m)\beta}. 
\end{equation}

\medskip
\subsubsection{\bf The cohomology class $[\J_\mu]_0$}
The compact Lie group $\SO(2)$ is the real points of an algebraic group defined over $\Q$, whose $\Q$-points we denote $\SO(2)(\Q)$; this consists 
of all those $\r(\theta)$ such that $\cos(\theta), \, \sin(\theta) \in \Q$.  
We will scale the cohomology class $[\J_\mu]$, such that the distinguished cohomological vector is rational, i.e., takes rational values on $\SO(2)(\Q);$ we denote 
this class by $[\J_\mu]_0$. Observe that $[\J_\mu]_0$ is well-defined only up to homothety by $\Q^\times.$ By \eqref{eqn:f-mu-4}, we see 
that some $\Q^\times$-multiple of $f_\mu$ is a distinguished cohomological vector for $[\J_\mu]_0.$

\medskip
\subsubsection{\bf The intertwining operator $T_{\rm st}$}

Consider the induced representation 
$$
\J_\mu \ = \ \Ind_{B_2(\C)}^{\GL_2(\C)}\left(
z^{-b_2+ \tfrac12} \, \bar{z}^{-c_1-\tfrac12} \ \otimes \ z^{-b_1-\tfrac12} \, \bar{z}^{-c_2+\tfrac12} \right) \ \ {\rm as} \ \ \Ind_{B_2(\C)}^{\GL_2(\C)}(\chi_1(\mu) \otimes \chi_2(\mu)), 
$$
where $\chi_1(\mu)(z) = z^{-b_2+ \tfrac12} \, \bar{z}^{-c_1-\tfrac12}$ and $\chi_2(\mu)(z) = z^{-b_1-\tfrac12} \, \bar{z}^{-c_2+\tfrac12}.$
The standard intertwining operator from $\J_\mu$ to its `companion' induced representation 
$$
T_{\rm st} : \Ind_{B_2(\C)}^{\GL_2(\C)}(\chi_1(\mu) \otimes \chi_2(\mu)) \ \longrightarrow \ 
\Ind_{B_2(\C)}^{\GL_2(\C)}(\chi_2(\mu) \otimes \chi_1(\mu))
$$
is given by the integral: 
\begin{equation}
\label{eqn:T-st-GL2}
T_{\rm st}(f)(g) \ = \ 
\int_\C f \left(
\begin{pmatrix} 0 & 1 \\ -1 & 0 \end{pmatrix}
\begin{pmatrix} 1 & u \\ 0 & 1 \end{pmatrix} 
g \right) \, du
\end{equation}
where $du$ is the Lebesgue measure on $\C$; if $u = x + \i y$ then $du = dx\,dy.$

\begin{prop}
\label{prop:irreducible-isomorphism-GL2}
Suppose $s = -1$ and $s=0$ are regular points for both $L(s, \chi_1(\mu)\chi_2(\mu)^{-1})$ and $L(1-s, \chi_1(\mu)^{-1}\chi_2(\mu)).$ Then, 
the representation $\Ind_{B_2(\C)}^{\GL_2(\C)}(\chi_1(\mu) \otimes \chi_2(\mu))$ is irreducible, and the standard intertwining operator $T_{\rm st}$ is an isomorphism. 
\end{prop}

\begin{proof}
Irreducibility follows from \cite[Chap.\,2, Thm.\,3]{godement}. The proof of  $T_{\rm st}$ being an isomorphism follows the same argument as in the proof of \cite[Prop.\,7.54]{harder-raghuram-book}. We will elaborate further when we deal with $\GL_N$; see Prop.\,\ref{prop:irreducible-isomorphism-GLN} below. 
\end{proof}

\medskip
\subsubsection{\bf The highest weight vector of the lowest $K$-type on the `other side'}

Since $T_{\rm st}$ is an isomorphism of $\GL_2(\C)$-modules, it maps the minimal $\SU(2)$-type in $\Ind(\chi_1(\mu) \otimes \chi_2(\mu))$ 
isomorphically onto the minimal $\SU(2)$-type in $\Ind(\chi_2(\mu) \otimes \chi_1(\mu))$, and within these $\SU(2)$-types, it maps $f_\mu$, which is the highest weight vector for $T_c^{(1)}$ described above 
to a multiple of the highest weight vector on the other side, which we denote $\tilde{f_\mu}$. We have the analogues of 
\eqref{eqn:f-mu-1} and \eqref{eqn:f-mu-2} for $\tilde{f_\mu}$: 

\begin{equation}
\label{eqn:tilde-f-mu-1}
\tilde f_\mu(\begin{pmatrix}z & * \\ & w \end{pmatrix} g)
 \ = \ 
z^{-b_1-\tfrac12} \, \bar{z}^{-c_2+\tfrac12} \cdot  w^{-b_2+ \tfrac12} \, \bar{w}^{-c_1-\tfrac12} \cdot \left|\frac{z}{w}\right|_\C^{1/2} \tilde f_\mu(g), 
\end{equation}
\begin{equation}
\label{eqn:tilde-f-mu-2}
\tilde f_\mu( 
\begin{pmatrix} e^{i\alpha} &  \\ & e^{-i\alpha} \end{pmatrix}
g \begin{pmatrix} e^{i\beta} &  \\ & e^{-i\beta} \end{pmatrix})  = \ 
e^{i(-2m)\alpha}
e^{i(2m)\beta} \tilde f_\mu(g). 
\end{equation}
for all $g \in \GL_2(\C).$ 
But, \eqref{eqn:tilde-f-mu-2} also says that $\tilde f_\mu(I) = 0$ (since $m \geq 1$). Let 
$w_0 = \left(\begin{smallmatrix} 0 & 1 \\ -1 & 0 \end{smallmatrix}\right) \ = \ \r(-\pi/2).$
Then, using \eqref{eqn:SO(2)-on-f-2m}, and evaluating at $w_0$ we see: $(\r(\theta)\cdot \tilde f_\mu)(w_0) = \cos^{2m}(\theta) \cdot \tilde f_\mu(w_0).$ (The other summands vanish on $w_0$ using the analogue 
of \eqref{eqn:tilde-f-mu-2}.) Hence 
$$
\tilde f_\mu(\r(\theta - \pi/2)) \ = \ \tilde f_\mu (w_0 \r(\theta)) \ = \ \cos^{2m}(\theta) \cdot \tilde f_\mu(w_0).
$$
Change $\theta \mapsto \theta +\pi/2$ and noting $\cos(\theta+\pi/2) = - \sin(\theta)$ we get the analogue of \eqref{eqn:f-mu-3}: 
\begin{equation}
\label{eqn:tilde-f-mu-3}
\tilde f_\mu(\r(\theta)) \ = \ \sin^{2m}(\theta), 
\end{equation}
where, we have normalized $\tilde f_\mu$ by $\tilde f_{\mu}(w_0) = 1.$ From \eqref{eqn:tilde-f-mu-1}, \eqref{eqn:tilde-f-mu-2} and \eqref{eqn:tilde-f-mu-3}, we have:
\begin{equation}
\label{eqn:tilde-f-mu-4}
\tilde f_\mu(\begin{pmatrix}z & * \\ & w \end{pmatrix} \r(\theta) \begin{pmatrix} e^{i\beta} &  \\ & e^{-i\beta} \end{pmatrix}) \ = \ 
z^{-b_1} \, \bar{z}^{-c_2+1} \cdot w^{-b_2} \bar{w}^{-c_1-1} \cdot
\sin^{2m}(\theta) \cdot
e^{i(2m)\beta}. 
\end{equation}

\medskip
\subsubsection{\bf The basic intertwining calculation for $\GL_2$}

\begin{prop}
\label{prop:basic-Tst-GL2}
$$
T_{\rm st}(f_\mu) \ \approx_{\Q^\times} \ 
\frac{L(0, \chi_1 \chi_2^{-1})}{L(1, \chi_1 \chi_2^{-1})} \, 
\tilde f_\mu, 
$$
where, $\approx_{\Q^\times}$ means equality up to a nonzero rational number. 
\end{prop}

\begin{proof}
It is clear that $T_{\rm st}(f_\mu)$ is a scalar multiple of $\tilde f_\mu.$ To compute that scalar we evaluate $T_{\rm st}(f_\mu)$ at $w_0$: 
$$
T_{\rm st}(f_\mu)(w_0) \ = \ 
\int_\C f_\mu \left(
\begin{pmatrix} 0 & 1 \\ -1 & 0 \end{pmatrix}
\begin{pmatrix} 1 & u \\ 0 & 1 \end{pmatrix} 
\begin{pmatrix} 0 & 1 \\ -1 & 0 \end{pmatrix}
 \right) \, du
\ = \ 
\int_\C f_\mu \left(
\begin{pmatrix} 1 & 0 \\ -u & 1 \end{pmatrix}
 \right) \, du. 
$$
Change to polar coordinates: $u = r e^{i\theta}.$ Note that
$$
\begin{pmatrix} 1 & 0 \\ -r e^{i\theta} & 1 \end{pmatrix} \ = \ 
\begin{pmatrix} e^{-i \theta/2} & 0 \\ 0 & e^{i \theta/2} \end{pmatrix}
\begin{pmatrix} 1 & 0 \\ -r  & 1 \end{pmatrix}
\begin{pmatrix} e^{i \theta/2} & 0 \\ 0 & e^{-i \theta/2} \end{pmatrix}.
$$
Hence, applying \eqref{eqn:f-mu-1} and \eqref{eqn:f-mu-2} we get
$$
f_\mu \left(
\begin{pmatrix} 1 & 0 \\ - r e^{i\theta}& 1 \end{pmatrix} \right) \ = \ 
e^{-i(2m) \theta/2} 
f_\mu \left(\begin{pmatrix} 1 & 0 \\ - r & 1 \end{pmatrix}\right)
e^{i(2m) \theta/2} \ = \ 
f_\mu \left(\begin{pmatrix} 1 & 0 \\ - r & 1 \end{pmatrix}\right).
$$
Next, we note 
$$
\begin{pmatrix} 1 & 0 \\ - r & 1 \end{pmatrix} \ = \ 
\begin{pmatrix} \Delta_r^{-1} & -r\Delta_r^{-1} \\ 0 &  \Delta_r \end{pmatrix}
\begin{pmatrix} \Delta_r^{-1} & r\Delta_r^{-1} \\ - r\Delta_r^{-1} & \Delta_r^{-1} \end{pmatrix}, 
$$
where $\Delta_r = \sqrt{1+r^2}.$ Note that $\left(\begin{smallmatrix} \Delta_r^{-1} & r\Delta_r^{-1} \\ - r\Delta_r^{-1} & \Delta_r^{-1} \end{smallmatrix}\right) = \r(\alpha)$ with $\alpha = \tan^{-1}(-r).$
From \eqref{eqn:f-mu-1} and \eqref{eqn:f-mu-3} we get
$$
f_\mu \left(\begin{pmatrix} 1 & 0 \\ - r & 1 \end{pmatrix}\right) \ = \ 
\frac{1}{\Delta_r^{2m+1}}. 
$$
The integral evaluates to: 
\begin{equation}
\label{eqn:T-st-f-mu-explicit-GL2}
T_{\rm st}(f_\mu)(w_0) \ = \ \int_{\theta = 0}^{2\pi} \int_{r=0}^\infty \frac{r \, dr\, d\theta}{(\sqrt{1+r^2})^{2m+1}}  \ = \ \frac{2\pi}{2m-1}.
\end{equation}

Now, $\chi_1\chi_2^{-1}(z) = z^m \bar{z}^{-m}$, and by \eqref{eqn:abelian-local-l-factor} we have
$L(s, \chi_1\chi_2^{-1}) = 2 (2\pi)^{-(s+m)} \Gamma(s+m).$ Hence, 
\begin{equation}
\label{eqn:ratio-GL2-inft-explicit}
\frac{L(0, \chi_1 \chi_2^{-1})}{L(1, \chi_1 \chi_2^{-1})} \ = \ 
\frac{(2\pi)^{- m}}{(2\pi)^{-1-m}} 
\frac{\Gamma(m)}{\Gamma(m+1)} \ = \ \frac{2\pi}{m}. 
\end{equation}
The proof follows from \eqref{eqn:T-st-f-mu-explicit-GL2} and \eqref{eqn:ratio-GL2-inft-explicit}. 
\end{proof}

\medskip
\subsubsection{\bf Arithmetic interpretation of the intertwining calculation}

Denote the induced representation in the range of $T_{\rm st}$ as $\tilde\J_\mu = \Ind(\chi_2(\mu) \otimes \chi_1(\mu))$
Now, fix a cohomology class $[\tilde\J_\mu]_0$: 
\begin{equation}
\label{eqn:[tilde-J-mu]}
H^1(\g_2, \tilde{\k}_2; \tilde\J_\mu \otimes \M_\mu) \ = \ 
H^1(\gl_2(\C), Z_2(\C)\U(2); \tilde\J_\mu \otimes \M_\mu) \ = \ \C [\tilde\J_\mu]_0,
\end{equation}
characterised by the property that its distingusihed cohomological vector is rational; hence, up to $\Q^\times$-multiples, the vector 
$\tilde f_\mu$ is a cohomological vector for $\tilde\J_\mu.$ 
Consider the map induced in cohomology by the operator $T_{\rm st} : \J_\mu \to \tilde\J_\mu$; at the level of generators it will map 
$[\J_\mu] = \sum_{i, \alpha} X_i^* \otimes \phi_{i, \alpha} \otimes m_{\alpha}$ to 
$\sum_{i, \alpha} X_i^* \otimes T_{\rm st}(\phi_{i, \alpha}) \otimes m_{\alpha}.$ Then, in terms of the cohomology classes with rational 
distinguished cohomological vectors, Prop.\,\ref{prop:basic-Tst-GL2} may be stated as:  

\begin{prop}
\label{prop:basic-Tst-GL2-rational-classes}
$$
T_{\rm st}([\J_\mu]_0) \ \approx_{\Q^\times} \ 
\frac{L(0, \chi_1 \chi_2^{-1})}{L(1, \chi_1 \chi_2^{-1})} \, 
[\tilde\J_\mu]_0, 
$$
\end{prop}

\medskip
\subsubsection{\bf Rational classes via Delorme's Lemma}
Recall Delorme's Lemma (see Borel--Wallach \cite[Thm.\,III.3.3]{borel-wallach}) which in the current context can be explicated as: 
\begin{multline}
\label{eqn:delorme-gl2}
H^1(\g_2, \tilde{\k}_2; \J_\mu \otimes \M_\mu) \ \simeq \\
H^0(\g_1, \k_1; z^{-b_2+1} \bar{z}^{-c_1} \otimes \M_{(b_2-1)(c_1)}) \otimes 
H^0(\g_1, \k_1; z^{-b_1-1} \bar{z}^{-c_2} \otimes \M_{(b_1+1)(c_2)}), 
\end{multline}
where $\g_1 = \gl_1(\C)$, $\k_1 = \su(1),$ for $b, c \in \Z$ we abbreviate the character $z \mapsto z^{b} \bar{z}^{c}$ simply as 
$z^{b} \bar{z}^{c}$, and $\M_{(b)(c)}$ is the algebraic representation $z^{b} \bar{z}^{c}$ of the real group $\C^\times.$  Note 
that on the right hand side, in each factor, we are looking at the relative Lie algebra cohomology for $\GL_1(\C)$ of 
$z^{-b} \bar{z}^{-c} \otimes \M_{(b)(c)},$ which is nothing but the trivial character! For brevity, denote
$H^0_{b,c} = H^0(\g_1, \k_1; z^{-b} \bar{z}^{-c} \otimes \M_{(b)(c)}).$
Parse the isomorphism in Delorme's Lemma: the map $f \mapsto f(1_2)$ for $f \in \J_\mu$ induces an isomorphism coming from Frobenius reciprocity: 
\begin{multline*}
H^1(\g_2, \tilde{\k}_2; \aInd(z^{-b_2+1} \bar{z}^{-c_1} \otimes  z^{-b_1-1} \bar{z}^{-c_2}) \otimes \M_\mu) 
\ \simeq \\ 
H^1(\b_2, \k_{B_2}; (z^{-b_2+1} \bar{z}^{-c_1} \otimes  z^{-b_1-1} \bar{z}^{-c_2}) \otimes \M_\mu), 
\end{multline*}
where $\b_2$ (resp., $\k_{B_2}$) 
is the real Lie algebra of $B_2(\C)$ (resp., $\U(2) \cap B_2(\C)$). 
The proof of \cite[Thm.\,III.3.3]{borel-wallach} gives that
\begin{multline*}
H^1(\b_2, \k_B; (z^{-b_2+1} \bar{z}^{-c_1} \otimes  z^{-b_1-1} \bar{z}^{-c_2}) \otimes \M_\mu) \ \simeq \\
H^0(\t_2, \k_{T_2}; (z^{-b_2+1} \bar{z}^{-c_1} \otimes  z^{-b_1-1} \bar{z}^{-c_2}) \otimes H^1(\u_{B_2}, \M_\mu)), 
\end{multline*}
where $\t_2$, $\k_{T_2}$, and $\u_{B_2}$ are the real Lie algebras of the diagonal torus $T_2(\C)$ in $B_2(\C)$, its maximal compact $\U(2) \cap T_2(\C)$, 
and the unipotent radical of $B_2(\C)$, respectively. 
To apply Kostant's theorem \eqref{eqn:kostant} we need the Kostant representatives of length $1$ for the Borel subgroup in the 
real reductive group $\GL_2(\C);$  if $w_0 = \left(\begin{smallmatrix} 0 & 1 \\ -1 & 0 \end{smallmatrix}\right)$ then the required Kostant representatives are 
$w_l = (w_0,1)$ and $w_r = (1,w_0).$ By direct calculation we have: 
$$
w_l \cdot \mu \ = \ 
(w_0, 1) \cdot ((b_1, b_2), (c_1, c_2)) \ = \ 
((b_2-1, b_1+1)(c_1,c_2)). 
$$
Hence, $\M_{w_l \cdot \mu},$ as an algebraic irreducible representation for the diagonal torus in $\GL_2(\C),$ is  
$\M_{(b_2-1)(c_1)} \otimes \M_{(b_1+1)(c_2)},$ giving us \eqref{eqn:delorme-gl2} that we rewrite as: 
\begin{equation}
\label{eqn:delorme-gl2-simple}
\gamma_1 : H^1(\g_2, \tilde{\k}_2; \J_\mu \otimes \M_\mu) \ \stackrel{\approx}{\longrightarrow} \ 
H^0_{(b_2-1, c_1)} \otimes 
H^0_{(b_1+1, c_2)}. 
\end{equation}
Fix a basis $\omega_{(b,c)}$ for $H^0_{(b,c)}$ which is the rational class corresponding to the cohomology of the trivial representation. 
We take for $[\J_\mu]_0$, the basis element $H^1(\g_2, \tilde{\k}_2; \J_\mu \otimes \M_\mu),$ such that 
$\gamma_1([\J_\mu]_0) = \omega_{(b_2-1, c_1)} \otimes \omega_{(b_1+1, c_2)}.$ 

Now, we work with the cohomology class for the induced module $\tilde{\J}_\mu$ in the codomain of $T_{\rm st}.$ Here, the integral in 
\eqref{eqn:T-st-GL2} tells us to consider Frobenius reciprocity via the map $\tilde{f} \mapsto \tilde{f}(w_0),$ which  
induces an isomorphism: 
\begin{multline*}
H^1(\g_2, \tilde{\k}_2; \aInd(z^{-b_1} \bar{z}^{-c_2+1} \otimes  z^{-b_2} \bar{z}^{-c_1-1}) \otimes \M_\mu) 
\ \simeq \\ 
H^1(\bar{\b}_2, \k_{\bar{B}}; (z^{-b_1} \bar{z}^{-c_2+1} \otimes  z^{-b_2} \bar{z}^{-c_1-1}) \otimes \M_\mu),
\end{multline*}
where $\bar{B}$ is the Borel subgroup of $\GL_2(\C)$ of lower triangular matrices that is opposite to $B_2(\C)$. In this situation we 
use the Kostant representative $w_r = (1,w_0)$ to give ourselves the isomosphism: 
\begin{equation}
\label{eqn:delorme-gl2-simple-2}
\gamma_{w_0} : H^1(\g_2, \tilde{\k}_2; \tilde{\J}_\mu \otimes \M_\mu) \ \stackrel{\approx}{\longrightarrow} \ 
H^0_{(b_1, c_2-1)} \otimes 
H^0_{(b_2, c_1+1)}. 
\end{equation}
We take for $[\tilde\J_\mu]_0$, the basis element $H^1(\g_2, \tilde{\k}_2; \tilde\J_\mu \otimes \M_\mu),$ such that 
$\gamma_{w_0}([\J_\mu]_0) = \omega_{(b_1, c_2-1)} \otimes \omega_{(b_2, c_1+1)}.$  It helps to keep the following diagram in mind: 
$$
\xymatrix{
H^1(\g_2, \tilde{\k}_2; \J_\mu \otimes \M_\mu) 
\ar[d]_{T_{\rm st}^\bullet} \ar[r]^{\gamma_1} & 
H^0_{(b_2-1, c_1)} \otimes H^0_{(b_1+1, c_2)} \ar@{=}[d] \\ 
H^1(\g_2, \tilde{\k}_2; \tilde{\J}_\mu \otimes \M_\mu) \ar[r]_{\gamma_{w_0}}& 
H^0_{(b_1, c_2-1)} \otimes H^0_{(b_2, c_1+1)}
}$$
The diagram is not commutative! Prop.\,\ref{prop:basic-Tst-GL2-rational-classes} says that it is commutative up to nonzero rational numbers and a particular 
ratio of archimedean $L$-values. The reader is referred to \cite[Sect.\,9.6]{harder-raghuram-book}.

\medskip
\subsection{The case of $\GL_N$}

Now, we generalize Prop.\,\ref{prop:basic-Tst-GL2-rational-classes} to the case of $\GL_N,$ giving us the main result of this subsection in 
Prop.\,\ref{prop:basic-Tst-GLN-in-cohomology}.

\medskip
\subsubsection{\bf The induced representations and the standard intertwining operator}
Take strongly-pure weights $\mu$ and $\mu'$ as in Sect.\,\ref{sec:crit-comb-lemma}. 
Fix an archimedean place $v$ (which we often drop simply to avoid tedious notation). Consider the induced representation
$$
\sigma_v \ = \ \J_{\mu} \ = \ 
\Ind_{B_n(\C)}^{\GL_n(\C)}
\left(z^{\alpha_1} \bar{z}^{\beta_1} \otimes \cdots \otimes z^{\alpha_n} \bar{z}^{\beta_n} \right),  \quad \alpha_i, \beta_i \in \tfrac{(n-1)}{2} + \Z; 
$$
see \eqref{eqn:cuspidal-parameters-alpha} and \eqref{eqn:cuspidal-parameters-beta} for the cuspidal parameters $\alpha_i$ and $\beta_i.$  Abbreviate this as: 
$$
\sigma_v \ = \ \J_{\mu} \ = \ \psi_1 \times \cdots \times \psi_n; \quad \psi_i(z) =  z^{\alpha_i} \bar{z}^{\beta_i}.
$$
Similarly, we have
$$
\sigma'_v \ = \ \J_{\mu'} \ = \ 
\Ind_{B_{n'}(\C)}^{\GL_{n'}(\C)}
\left(z^{\alpha'_1} \bar{z}^{\beta'_1} \otimes \cdots \otimes z^{\alpha'_{n'}} \bar{z}^{\beta'_{n'}} \right),  \quad \alpha'_j, \beta'_j \in \tfrac{(n'-1)}{2} + \Z, 
$$
which we abbreviate as: 
$$
\sigma'_v \ = \ \J_{\mu'} \ = \ \psi'_1 \times \cdots \times \psi'_{n'}; \quad \psi'_j(z) =  z^{\alpha'_j} \bar{z}^{\beta'_j}.
$$
We are interested in the standard intertwining operator
$$
T_{\rm st} \ : \ 
\aInd_{P_{(n,n')}(\C)}^{\GL_N(\C)}(\J_\mu \times \J_{\mu'})  \ \longrightarrow \ \aInd_{P_{(n',n)}(\C)} ^{\GL_N(\C)} (\J_{\mu'}(-n) \times \J_\mu(n')).
$$
which, in terms of normalized induced representations looks like:
\begin{equation}
\label{eqn:T-st-GLN}
T_{\rm st} \ : \ 
\Ind_{P_{(n,n')}(\C)}^{\GL_N(\C)}(\J_\mu(-n'/2) \times \J_{\mu'}(n/2))  \ \longrightarrow \ 
\Ind_{P_{(n',n)}(\C)} ^{\GL_N(\C)} (\J_{\mu'}(n/2) \times \J_\mu(-n'/2)).
\end{equation}
Write 
\begin{eqnarray*}
\J_\mu(-n'/2) \ & = & \ \Ind_{B_n(\C)}^{\GL_n(\C)}(\chi_1 \otimes \cdots \otimes \chi_n),  \quad \chi_i = \psi_i(-n'/2), \quad {\rm and} \\
\J_{\mu'}(n/2) \ & = & \ \Ind_{B_{n'}(\C)}^{\GL_{n'}(\C)}(\chi'_1 \otimes \cdots \otimes \chi'_{n'}),  \quad \chi'_j = \psi'_j(n/2). 
\end{eqnarray*}
Apply transitivity of normalized induction to the representation in the domain of \eqref{eqn:T-st-GLN} to get: 
\begin{multline*}
\Ind_{P_{(n,n')}(\C)}^{\GL_N(\C)}\left(
\Ind_{B_n(\C)}^{\GL_n(\C)}(\chi_1 \otimes \cdots \otimes \chi_n) \times 
\Ind_{B_{n'}(\C)}^{\GL_{n'}(\C)}(\chi'_1 \otimes \cdots \otimes \chi'_{n'}) \right) \ = \\ 
\Ind_{B_N(\C)}^{\GL_N(\C)}(\chi_1 \otimes \cdots \otimes \chi_n \otimes \chi'_1 \otimes \cdots \otimes \chi'_{n'}) 
\ =: \ 
\chi_1 \times \cdots \times \chi_n \times \chi'_1 \times \cdots \times \chi'_{n'}, 
\end{multline*}
and, similarly, the induced representation in the target is: 
$\chi'_1 \times \cdots \times \chi'_{n'} \times \chi_1 \times \cdots \times \chi_n.$
Hence, \eqref{eqn:T-st-GLN} takes the shape: 
\begin{equation}
\label{eqn:T-st-GLN-from-BN}
T_{\rm st} \ : \ 
\chi_1 \times \cdots \times \chi_n \times \chi'_1 \times \cdots \times \chi'_{n'}
 \ \longrightarrow \ 
\chi'_1 \times \cdots \times \chi'_{n'} \times \chi_1 \times \cdots \times \chi_n. 
\end{equation}
For a function $f \in \chi_1 \times \cdots \times \chi_n \times \chi'_1 \times \cdots \times \chi'_{n'}$ we have the intertwining integral
\begin{equation}
\label{eqn:T-st-GLN-integral}
T_{\rm st}(f)(g) \ = \ \int_{M_{n \times n'}(\C)} f\left(w_0 \begin{pmatrix}1_n & u \\ & 1_{n'}\end{pmatrix} g \right) \, du, 
\end{equation}
where $w_0$ is the element of the Weyl group of $\GL_N$ given by the following permutation: 
$$
w_0 \ = \ 
\begin{pmatrix}
1 &  2 & \dots       &  n-1  & n &  n+1 & n+2 & \dots & N-1 & N \\
n'+1 & n'+2 & \dots  & N-1 & N & 1      & 2     & \dots & n'-1 & n'
\end{pmatrix}, 
$$
and the measure $du$ on $M_{n \times n'}(\C)$ in in the integral is taken as the product of the Lebesgue measures on each coordinate of $u$.

\medskip
\begin{prop}
\label{prop:irreducible-isomorphism-GLN}
Assume that the archimedean local factors $L(s, \J_\mu \times \J_{\mu'}^{\sf v})$ and $L(1-s, \J_\mu^{\sf v}  \times \J_{\mu'})$ are finite at 
$s = -N/2$ and $s=1-N/2.$ Then 
\begin{enumerate}
\item the representations $\chi_1 \times \cdots \times \chi_n \times \chi'_1 \times \cdots \times \chi'_{n'}$ and 
$\chi'_1 \times \cdots \times \chi'_{n'} \times \chi_1 \times \cdots \times \chi_n$ are irreducible; and furtheremore,  
\item the standard intertwining integral $T_{\rm st}$ in \eqref{eqn:T-st-GLN-integral} converges and gives an isomorphism between these two irreducible representations. 
\end{enumerate}
\end{prop}

\begin{proof}
The proof follows from the Langlands--Shahidi machinery. For brevity, only for this proof, let $\sigma  = \J_{\mu}$ and $\sigma'  = \J_{\mu'}.$ Let 
$$
I_P^G(s, \sigma \otimes  \sigma') = 
{\rm Ind}_{P_{(n,n')}(\C)}^{\GL_N(\C)}((\sigma  \otimes |\ |^{\frac{n'}{N}s}) \otimes ( \sigma' \otimes |\ |^{\frac{-n}{N}s})). 
$$
The $s$-variable is introduced using the fundamental weight corresponding to the simple root that is deleted for the maximal standard parabolic subgroup $P_{(n,n')}$ whose Levi quotient is the 
block diagonal subgroup $\GL_n \times \GL_{n'}.$ Similarly, we let
$$
I_Q^G(-s, \sigma \otimes  \sigma') = 
{\rm Ind}_{P_{(n',n)}(\C)}^{\GL_N(\C)}( ( \sigma' \otimes |\ |^{\frac{-n}{N}s})   \otimes (\sigma  \otimes |\ |^{\frac{n'}{N}s})). 
$$
The standard intertwining operator $T_{\rm st}(s, w_0) : I_P^G(s, \sigma \otimes  \sigma') \to I_Q^G(-s, \sigma \otimes  \sigma')$ is given by the integral \eqref{eqn:T-st-GLN-integral}. 
Under the hypothesis of the proposition, it follows from Casselman--Shahidi \cite[Prop.\,5.3]{casselman-shahidi} that the induced representations 
$I_P^G(-N/2, \sigma \otimes  \sigma') = \chi_1 \times \cdots \times \chi_n \times \chi'_1 \times \cdots \times \chi'_{n'}$ and 
$I_Q^G(N/2, \sigma \otimes  \sigma') = \chi'_1 \times \cdots \times \chi'_{n'} \times \chi_1 \times \cdots \times \chi_n$
are irreducible. The operator $T_{\rm st} = T_{\rm st}(-N/2, w_0)$ being an isomorphism follows exactly as in the proof of 
\cite[Prop.\,7.54]{harder-raghuram-book}; this part of the proof uses Shahidi's 
results on local constants \cite{shahidi-duke80}. 
\end{proof}

\medskip
\subsubsection{\bf Factorizing the intertwining operator}

For $1 \leq i \leq N-1$, let ${\sf s}_i = (i, i+1)$ be the $i$-th simple reflection corresponding to the $i$-th simple root $\alpha_i = e_i-e_{i+1}.$ Its easy to see that
a positive root $e_i - e_j$ (positivity is $i < j$) is mapped to a negative root by $w_0$ if and only if $1 \leq i \leq n$ and $n+1 \leq j \leq N,$ hence
$l(w_0) = nn'.$
Furthermore, its easy to see that: 
$$
w_0 \ = \ 
({\sf s}_{n'} \dots {\sf s}_2 {\sf s}_1) \cdots ({\sf s}_{N-2}\dots {\sf s}_n {\sf s}_{n-1}) ({\sf s}_{N-1}\dots {\sf s}_{n+1} {\sf s}_n), 
$$
where the right hand side is grouped into $n$ parenthetical expressions each of which is a product of $n'$ simple reflections; hence 
giving a minimal expression of $w_0$ in terms of $l(w_0)$ many simple reflections. This gives a factorization: 
\begin{equation}
\label{eqn:T-factorised}
T_{\rm st} = T_{\rm st}(w_0) \ = \\ 
\left(T({\sf s}_{n'}) \circ \cdots \circ T({\sf s}_2)\circ T({\sf s}_1)\right) \circ \cdots  \circ \left(T({\sf s}_{N-1}) \circ \cdots \circ T({\sf s}_{n+1})\circ T({\sf s}_n) \right), 
\end{equation}
which is well-known in the Langlands--Shahidi method; see, for example, \cite[Thm.\,4.2.2]{shahidi-book} as applied to our situation.

\medskip
\begin{exam}
{\rm To visualise such a factorisation, consider the simple but nontrivial example: take $n = 3$ and $n'=2$, then
the right hand side of \eqref{eqn:T-factorised} is the sequence of operators: 

\begin{enumerate}
\item $T({\sf s}_3) \ : \ \chi_1 \times \chi_2 \times \chi_3 \times \chi'_1 \times \chi'_2 
\ \longrightarrow \ \chi_1 \times \chi_2 \times  \chi'_1 \times \chi_3 \times \chi'_2$ 

\item $T({\sf s}_4) \ : \ \chi_1 \times \chi_2 \times  \chi'_1 \times \chi_3 \times \chi'_2  
\ \longrightarrow \ \chi_1 \times \chi_2 \times  \chi'_1 \times \chi'_2 \times \chi_3$

\item $T({\sf s}_2) \ : \ \chi_1 \times \chi_2 \times  \chi'_1 \times \chi'_2 \times \chi_3 
 \ \longrightarrow \ \chi_1 \times \chi'_1  \times  \chi_2 \times \chi'_2 \times \chi_3$
 
\item $T({\sf s}_3) \ : \ \chi_1 \times \chi'_1  \times  \chi_2 \times \chi'_2 \times \chi_3
  \ \longrightarrow \  \chi_1 \times \chi'_1  \times  \chi'_2 \times \chi_2 \times \chi_3$
  
\item $T({\sf s}_1) \ : \ \chi_1 \times \chi'_1  \times  \chi'_2 \times \chi_2 \times \chi_3 
  \ \longrightarrow \  \chi'_1 \times \chi_1  \times  \chi'_2 \times \chi_2 \times \chi_3$
  
\item $T({\sf s}_2) \ : \ \chi'_1 \times \chi_1  \times  \chi'_2 \times \chi_2 \times \chi_3
 \ \longrightarrow \ \chi'_1 \times \chi'_2  \times   \chi_1\times \chi_2 \times \chi_3.$
\end{enumerate}
}\end{exam}
\medskip

The point is that at every intermediate stage, there are only two characters $\chi_i$ and $\chi'_j$ that are getting switched. The corresponding integral is happening over 
the coordinate $u_{ij}$ in the variable $u \in M_{n \times n'}(\C)$ that appears in \eqref{eqn:T-st-GLN-integral}. The measure $du$, as mentioned above, is the product of the Lebesgue measures 
$du_{ij}.$ Such an intermediate integral is the induction to $\GL_N$ of a $\GL_2$-intertwining integral, and we have seen that it corresponds to scaling by a factor 
$L_\infty(0, \chi_i \chi_j'^{-1})/L_\infty(1, \chi_i \chi_j'^{-1})$ (up to a nonzero rational). This implies that $T_{\rm st}$ will have a scaling factor of the product of all 
intermediate scaling factors, towards which note the easy lemma:

\begin{lemma}
\label{lem:l-factor-infinity-correct-ratio}
$$
\prod_{i=1}^n \prod_{j=1}^{n'} 
\frac{L_\infty(0, \chi_i \chi_j'^{-1})}{L_\infty(1, \chi_i \chi_j'^{-1})} \ = \ 
\prod_{i=1}^n \prod_{j=1}^{n'} 
\frac{L_\infty(-\tfrac{N}{2}, \psi_i \psi_j'^{-1})}{L_\infty(1-\tfrac{N}{2}, \psi_i \psi_j'^{-1})} \ = \ 
\frac{L_\infty(-\tfrac{N}{2}, \sigma_\infty \times \sigma_\infty'^{\sf v})}{L_\infty(1-\tfrac{N}{2}, \sigma_\infty \times \sigma_\infty'^{\sf v})}.
$$
\end{lemma}

\medskip
\subsubsection{\bf The intertwining operator in cohomology}
Let $\sJ = \sJ^0$ stand for the underlying $(\g_N,\k_N)$-module of 
$\Ind_{B_N(\C)}^{\GL_N(\C)}( \chi_1 \otimes \cdots \otimes \chi_n \otimes \chi'_1 \otimes \cdots \otimes \chi'_{n'})$
and 
similarly, $\tilde\sJ = \sJ^{nn'}$ that of  
$\Ind_{B_N(\C)}^{\GL_N(\C)}( \chi'_1 \otimes \cdots \otimes \chi'_{n'} \otimes \chi_1 \otimes \cdots \otimes \chi_n).$ 
Rewrite the factorization in \eqref{eqn:T-factorised} as: 
$$
T_{\rm st} = T^{nn'} \circ \cdots \circ T^2 \circ T^1, \quad T^k : \sJ^{k-1} \to \sJ^k \ \mbox{for $1 \leq k \leq nn'.$}
$$
with each $\sJ^k$ being an irreducible principal series representation, and each $T^k$ is the induction of a $\GL_2$-intertwining operator as explained. 
Note that  
$$
\Ind_{B_N(\C)}^{\GL_N(\C)}(\chi_1 \otimes \cdots \otimes \chi_n \otimes \chi'_1 \otimes \cdots \otimes \chi'_{n'}) = 
\aInd_{B_N(\C)}^{\GL_N(\C)}(\xi_1 \otimes \cdots \otimes \xi_n \otimes \xi'_1 \otimes \cdots \otimes \xi'_{n'}),
$$
where 
$\xi_i = \chi_i\left(\tfrac{N-2i+1}{2}\right) = \psi_i\left(\tfrac{n-2i+1}{2}\right)$ and 
$\xi'_j = \chi'_j\left(\tfrac{N-2j-2n+1}{2}\right) = \psi'_j\left(\tfrac{n'-2j+1}{2}\right),$ 
are all {\it algebraic} characters of $\C^\times$. Similarly, each $\sJ^k$ is the algebraic parabolic induction of an algebraic character of the diagonal torus. 
Delorme's lemma identifies the one-dimensional cohomology group 
$H^{b_N^\C}(\g_N, \k_N; \sJ^k \otimes \M_\lambda)$ as a tensor product of the $\GL_1$ cohomology groups for the $\xi_i$'s and $\xi_j'$'s; 
as in \eqref{eqn:delorme-gl2-simple}, but simplifying notations we have: 
$$
\gamma_k : H^{b_N^\C}(\g_N, \k_N; \sJ^k \otimes \M_\lambda)  \ \stackrel{\approx}{\longrightarrow} \ 
\mbox{(product of $\GL_1$ cohomology groups)}.
$$
This product of $\GL_1$-cohomology groups may be identified with each other for $1 \leq k \leq nn'.$ 
Fixing a rational basis $\omega_{(b,c)}$ for 
each of the $\GL_1$-classes and so for their tensor product, we define a basis element $[\sJ^k]_0$ for 
$H^{b_N^\C}(\g_N, \k_N; \sJ^k \otimes \M_\lambda)$ via $\gamma_k^{-1}.$  

\medskip

We start with $T^1 : \sJ^0 \to \sJ^1$ and note that this is the induction from $(n-1, 2, n'-1)$-parabolic subgroup of $\GL_N$ of the $\GL_2$-intertwining operator that switches $\chi_n$ and $\chi'_1$. Prop.\,\ref{prop:basic-Tst-GL2-rational-classes} applied to $T^1$ gives: 
$$
(T^1)^\bullet ([\sJ^0]_0) \ \approx_{\Q^\times} \ 
\frac{L(0, \chi_n \chi'_1{}^{-1})}{L(1, \chi_n \chi'_1{}^{-1})} [\sJ^1]_0. 
$$
At the next step, from the factorisation in \eqref{eqn:T-factorised}, we will get: 
$$
(T^2)^\bullet ([\sJ^1]_0) \ \approx_{\Q^\times} \ 
\frac{L(0, \chi_n \chi'_2{}^{-1})}{L(1, \chi_n \chi'_2{}^{-1})} [\sJ^2]_0, 
$$
and so on. Using Lem.\,\ref{lem:l-factor-infinity-correct-ratio}, Prop.\,\ref{prop:basic-Tst-GL2-rational-classes} generalizes to the following 

\begin{prop}
\label{prop:basic-Tst-GLN-in-cohomology}
$$
T_{\rm st}^\bullet([\sJ]_0) \ \approx_{\Q^\times} \ 
\frac{L_\infty(-\tfrac{N}{2}, \sigma_\infty \times \sigma_\infty'^{\sf v})}{L_\infty(1-\tfrac{N}{2}, \sigma_\infty \times \sigma_\infty'^{\sf v})} \, 
[\tilde \sJ]_0. 
$$
\end{prop}
The reader is referred to Harder \cite{harder-tifr} where a hope is expressed in general, and verified in the context therein, that the rational number implicit in $\approx_{\Q^\times}$ has a simple
shape. See \eqref{eqn:T-st-f-mu-explicit-GL2} and \eqref{eqn:ratio-GL2-inft-explicit} above in the simplest possible case of $n = n' = 1.$

\medskip

\section{\bf The main theorem on special values of $L$-functions for $\GL_n \times \GL_{n'}$}
\label{sec:main-theorem}

Before the main theorem on $L$-values (Thm.\,\ref{thm:main}) can be stated and proved, 
two technical results on the boundary cohomology are necessary; the 
first is what is known as a `Manin--Drinfeld' principle and the second is on rank-one Eisenstein cohomology.  

\medskip
\subsection{A Manin--Drinfeld Principle}
\label{sec:manin-drinfeld-principle} 
The main purpose of this subsection is to state and prove Thm.\,\ref{thm:manin-drinfeld}. 

\subsubsection{\bf Kostant representatives}

To begin, two important lemmas about Kostant representatives from \cite[Sect.\,5.3.2]{harder-raghuram-book} are recorded below. 
Recall that $P = \Res_{F/\Q}(P_0)$ and $P_0 = P_{(n,n')}$ is the maximal parabolic subgroup of type $(n,n')$ of $G_0 = \GL_N/F.$ Let $Q_0 = P_{(n',n)}$ be the associate parabolic, and $Q = \Res_{F/\Q}(Q_0)$.
Let $\bfpi_{M_{P_0}} = \bfpi_{G_0} - \{\alpha_{P_0}\}$. 
 Let $w_{P_0}$ be the unique element of $W_0 = W_{G_0}$ such that $w_{P_0}(\bfpi_{M_{P_0}}) \subset \bfpi_{G_0}$ and 
$w_{P_0}(\alpha_{P_0}) < 0,$ it is the longest Kostant representative for $W^{P_0}.$ 
%Let $w_P \in W_G$ be the element such that $w_P^\tau = w_{P_0}$ for all $\tau : F \to E.$
\medskip

\begin{lemma}
\label{lem:kostant-P-Q}
With notations as above, one has: 
\begin{enumerate}
\item The map $w \mapsto w^\prime:= w_{P}\,w$ gives a bijection $W^P \to W^Q$. If $w= (w^\tau)_{\tau : F \to E}$, then by definition, 
$w_{P}w = (w_{P_0}w^\tau)_{\tau:F \to E}.$ 
\item This bijection has the property that $l(w^\tau) + l(w'^\tau) = \dim{(U_{P_0^\tau})}$. 
\item $w$ is balanced if and only if $w'$ is balanced.
\end{enumerate}
\end{lemma}

\medskip 

Similarly, there is the following self-bijection of $W^P$: 

\medskip
\begin{lemma}
\label{lem:kostant-P-P}
Let $w_G$ be the element of longest length in the Weyl group $W_G$ of $G$, and similarly, 
let $w_{M_P}$ be the element of longest length in the Weyl group $W_{M_P}$. 
Then: 
\begin{enumerate}
\item The map $w \mapsto w^{\mathsf v}:= w_{M_P}\cdot w \cdot w_G$ gives a bijection $W^P \to W^P$. 
\item This bijection has the property that $l(w^\tau) + l(w^{\v\tau}) = \dim{(U_{P_0^\tau})}.$ 
\item $w$ is balanced if and only if $w^\v$ is balanced.
\end{enumerate}
\end{lemma}

\medskip
\subsubsection{\bf Induced representations in boundary cohomology}
\label{sec:ind-rep-bdry-coh}

The conditions imposed by the combinatorial lemma has a consequence on the occurrences of induced representations as Hecke summands in the boundary cohomology. 
Recall that $E$ is a large enough Galois extension of $\Q$ that takes a copy of $F$. Consider strongly-pure weights 
$\mu \in X^+_{00}(T_n \times E)$ and 
$\mu' \in X^+_{00}(T_{n'} \times E).$ Let $\sigma_f \in \Coh_{!!}(G_n, \mu)$ and $\sigma'_f \in \Coh_{!!}(G_{n'}, \mu')$ be strongly-inner Hecke-summands. The effect of the related
balanced representatives: $w', w^\v$ and $w^\v{}'$ on certain weights are recorded in the following

\begin{prop}
\label{prop:w-wv-mu-muv}
Assume that the weights $\mu$ and $\mu'$ satisfy the conditions of the combinatorial lemma. Hence there is a balanced element $w \in W^P$ such that $\lambda := w^{-1} \cdot (\mu + \mu')$ is dominant. Then (after recalling the notations in Sect.\,\ref{sec:tate-twists}): 

\begin{enumerate}

\item $w'\cdot \lambda = (\mu' -n \bfgreek{delta}_{n'}) + (\mu + n' \bfgreek{delta}_{n}),$

\medskip

\item $w^\v \cdot \lambda^\v = (\mu^\v - n'\bfgreek{delta}_{n}) + (\mu'^\v + n \bfgreek{delta}_{n'}),$ and  

\medskip

\item $w^\v{}' \cdot \lambda^\v =  \mu'^\v + \mu^\v.$ 
\end{enumerate}
\end{prop}

For a proof of the above proposition, the reader is referred to \cite[Sect.\,5.3.4]{harder-raghuram-book}. The appearance of various induced modules in boundary cohomology in bottom and top degrees are recorded in the following

\begin{prop}
\label{prop:ind-modules-top-bottom}
Let the notations be as above. 

\medskip
\begin{enumerate}

\item The module $\aInd_{P(\A_f)}^{G(\A_f)}(\sigma_f \times \sigma'_f)$ appears in $H^{\q_b}(\partial_P\SG, \tM_{\lambda, E}),$ \\ where 
$\q_b = b_N^F = b_n^F + b_{n'}^F + \frac12 \dim(U_P).$

\medskip

\item The module $\aInd_{Q(\A_f)}^{G(\A_f)}(\sigma'_f(n) \times \sigma_f(-n'))$ appears in $H^{\q_b}(\partial_Q\SG, \tM_{\lambda,E}).$ 
\end{enumerate}
\medskip
The contragredient of the algebraically-induced modules is 
$$
 \aInd_{P(\A_f)}^{G(\A_f)}(\sigma_f \times \sigma'_f)^\v \ = \ \aInd_{P(\A_f)}^{G(\A_f)}(\sigma_f^\v(n') \times \sigma'^\v_f(-n)).
$$ 
Furthermore, for the contragredients and cohomology in top-degree we have: 

\begin{enumerate}
\medskip

\item[(3)] $\aInd_{P(\A_f)}^{G(\A_f)}(\sigma_f^\v(n') \times \sigma'^\v_f(-n))$  appears in $H^{\q_t}(\partial_P\SG, \tM_{\lambda^\v, E}),$ \\ 
where $\q_t = t_N^F - 1 = t_n^F + t_{n'}^F + \frac12 \dim(U_P).$

\medskip

\item[(4)] $\aInd_{Q(\A_f)}^{G(\A_f)}(\sigma'^\v_f \times \sigma_f^\v )$  appears in $H^{\q_t}(\partial_Q\SG, \tM_{\lambda^\v, E}).$ 
\end{enumerate}
\end{prop}

\begin{proof}
For (1), use the summand in Prop.\,\ref{prop:bdry-coh-2} indexed by the balanced Kostant representative $w \in W^P$ provided by the combinatorial lemma. 
For (2), use $w' \in W^Q$ from Lem.\,\ref{lem:kostant-P-Q}, and then use (1) of Prop.\,\ref{prop:w-wv-mu-muv}. For (3), use $w^\v \in W^P$ from Lem.\,\ref{lem:kostant-P-P}, and then use (2) of Prop.\,\ref{prop:w-wv-mu-muv}. For (4) use $w^\v{}' \in W^Q$ from Lem.\,\ref{lem:kostant-P-Q} and \ref{lem:kostant-P-P}, and (3) of Prop.\,\ref{prop:w-wv-mu-muv}. The assertions of the cohomology degrees 
is clear from Prop.\,\ref{prop:j-lambda-range-degrees} and \ref{prop-numerology}. 
\end{proof}

\medskip
\subsubsection{\bf The Manin--Drinfeld principle}
\label{sec:manin-drinfeld-notations}

Continue with the notations $\mu \in X^+_{00}(T_n \times E)$, 
$\mu' \in X^+_{00}(T_{n'} \times E),$ $\sigma_f \in \Coh_{!!}(G_n, \mu)$, and $\sigma'_f \in \Coh_{!!}(G_{n'}, \mu').$ 
Assume that $\mu$ and $\mu'$ satisfy the conditions of the combinatorial lemma (Lem.\,\ref{lem:comb-lemma}), and let 
$\lambda = w^{-1}\cdot(\mu+\mu').$ 
Let $K_f$ be an open-compact subgroup of $G(\A_f)$ such that 
$\aInd_{P(\A_f)}^{G(\A_f)}(\sigma_f \times \sigma'_f)$ has nonzero $K_f$-fixed vectors; suppose $\sf{k}$ is the dimension of these $K_f$-fixed vectors. Let 
$$
I_b^\place(\sigma_f, \sigma'_f)_{P, w} \ := \ 
\aInd_{P(\A_f)}^{G(\A_f)} \left(H^{b_n^F+b_{n'}^F}(\SMP, \tM_{w \cdot \lambda})(\sigma_f \times \sigma'_f) \right)^{K_f}, 
$$
and, similarly, define
$$
I_b^\place(\sigma'_f(n) , \sigma_f(-n'))_{Q, w'} \ := \ 
\aInd_{Q(\A_f)}^{G(\A_f)} \left(H^{b_n^F+b_{n'}^F}(\SMQ, \tM_{w' \cdot \lambda})(\sigma'_f(n) \times \sigma_f(-n')) \right)^{K_f}.
$$
Now, go to `top-degree' for the contragredient modules and define:    
$$
I_t^\place(\sigma_f^\v(n') , \sigma'^\v_f(-n))_{P, w^\v} \ := \ 
\aInd_{P(\A_f)}^{G(\A_f)} \left(H^{t_n^F+t_{n'}^F}(\SMP, \tM_{w^\v \cdot \lambda})(\sigma_f^\v(n') \times \sigma'^\v_f(-n)) \right)^{K_f}, 
$$
and, similarly, define  
$$
I_t^\place(\sigma'^\v_f , \sigma_f^\v)_{Q, w^\v{}'} \ := \ 
\aInd_{Q(\A_f)}^{G(\A_f)} \left(H^{t_n^F+t_{n'}^F}(\SMQ, \tM_{w^\v{}' \cdot \lambda})(\sigma'^\v_f \times \sigma_f^\v ) \right)^{K_f}.
$$

\bigskip
\begin{thm}
\label{thm:manin-drinfeld}
Let the notations be as above. 
\begin{enumerate}
\item The sum 
$$
I_b^\place(\sigma_f, \sigma'_f)_{P, w} \ \oplus \ I_b^\place(\sigma'_f(n) , \sigma_f(-n'))_{Q, w'} 
$$
is a $2\sf{k}$-dimensional $E$-vector space that is isotypic in $H^{q_b}(\PBSC, \tM_{\lambda,E}).$ Note that if $Q = P$ then $w' \neq w.$ Furthermore, 
there is a $\HGS$-equivariant projection:
$$
\fR^b_{\sigma_f, \sigma'_f} \ : \ 
H^{q_b}(\PBSC, \tM_{\lambda,E}) \ \longrightarrow \ 
I_b^\place(\sigma_f, \sigma'_f)_{P, w} \oplus I_b^\place(\sigma'_f(n) , \sigma_f(-n'))_{Q, w'}.
$$

\medskip

\item 
Similarly, in `top-degree', the sum 
$$
I_t^\place(\sigma_f^\v(n') , \sigma'^\v_f(-n))_{P, w^\v} \ \oplus \ I_t^\place(\sigma'^\v_f , \sigma_f^\v)_{Q, w^\v{}'} 
$$
is a $2\sf{k}$-dimensional $E$-vector space that is isotypic in $H^{q_t}(\PBSC, \tM_{\lambda^\v,E}).$ Note that if $Q = P$ then $w^\v{}' \neq w^\v.$ Furthermore, 
there is a $\HGS$-equivariant projection:
$$
\fR^t_{\sigma_f, \sigma'_f} \ : \ H^{q_t}(\PBSC, \tM_{\lambda^\v, E}) \ \longrightarrow \ 
I_t^\place(\sigma_f^\v(n') , \sigma'^\v_f(-n))_{P, w^\v} \oplus I_t^\place(\sigma'^\v_f , \sigma_f^\v)_{Q, w^\v{}'}. 
$$
\end{enumerate}
\end{thm}

The above theorem is the exact analogue of \cite[Thm.\,5.12]{harder-raghuram-book}, and the proof is identical. To help the reader, the two key-ideas are adumbrated as follows: 

\medskip
\begin{itemize}
\item 
There is a spectral sequence--built from the cohomology of various boundary strata $\partial_R\SG$, as $R$ runs over $G(\Q)$-conjugacy classes of parabolic subgroups of $G$--that converges to the boundary cohomology $H^\bullet(\partial\SG, -)$. This spectral sequence was alluded to in Sect.\,\ref{sec:coh-of-bdry} and is discussed in greater detail in \cite[Sect.\,4.1]{harder-raghuram-book}. The basic idea here is that up to semi-simplification the cohomology of the boundary is built from parabolically induced representations. 

\medskip
\item Recall the strong multiplicity one theorem of Jacquet and Shalika for isobaric automorphic representations 
\cite[Thm.\,4.4]{jacquet-shalika-II}. The two induced modules in $I_b^\place(\sigma_f, \sigma'_f)_{P, w}$ and $I_b^\place(\sigma'_f(n) , \sigma_f(-n'))_{Q, w'}$ are themselves, of course, $\HGS$-equivalent, and more importantly, after applying Jacquet--Shalika, they are not almost-everywhere equivalent to any other induced module anywhere else in boundary cohomology; see \cite[Sect.\,5.3.3]{harder-raghuram-book} for more details. 
\end{itemize}

\medskip
\subsection{Eisenstein cohomology}
\label{sec:eis-coh-main-thm}

All the statements in \cite[Chap.\,6]{harder-raghuram-book} go through {\it mutatis mutandis} in the current situation. Therefore, the discussion below is very brief and just enough details are provided for this article to be reasonably self-contained, and to be able to state the main theorem on rank-one 
Eisenstein cohomology in Thm.\,\ref{thm:rank-one-eis} below.

\medskip
\subsubsection{\bf Poincar\'e duality and consequences} 
The Poincar\'e duality pairings on $\SGK$ and $\PBSC$ are compatible with the maps in the long exact sequence in Sect.\,\ref{sec:long-e-seq}: 
$$
\xymatrix{
H^\bullet(\SGK, \tM_{\lambda, E}) \ar[d]^{\r^*} & \times & H^{\d-\bullet}_c(\SGK, \tM_{\lambda^\v, E}) & 
\longrightarrow & E \\
H^\bullet(\PBSC, \tM_{\lambda, E}) & \times &  H^{\d-1-\bullet}(\PBSC, \tM_{\lambda^\v, E}) \ar[u]^{\fd^*} & \longrightarrow & E
}
$$
Here ${\rm dim}(\SGK)  = b_N^F + t_N^F =: \d$; and so ${\rm dim}(\PBSC) = \d -1 = q_b + q_t.$
A consequence of the above diagram is that Eisenstein cohomology, defined as
$$
H^q_{\rm Eis}(\PBSC, \tM_{\lambda, E}) \ := \ {\rm Image}
\left(H^q(\SGK, \tM_{\lambda, E}) \stackrel{\r^\bullet}{\longrightarrow} H^q(\PBSC, \tM_{\lambda, E}) \right), 
$$
is a maximal isotropic subspace of boundary cohomology, i.e., 
$$
H^q_{\Eis}(\PBSC, \tM_{\lambda, E}) \ = \ 
H^{\d-1-q}_{\Eis}(\PBSC, \tM_{\lambda^\v, E})^\perp. 
$$

\medskip
\subsubsection{\bf Main result on rank-one Eisenstein cohomology} 
\label{sec:rank-one-eis-coh}
With notations as in Sect.\,\ref{sec:manin-drinfeld-notations}, 
consider the following maps  
starting from global cohomology $H^{q_b}( \SGK,\tM_{\lambda,E})$ 
and ending with an isotypic component in boundary cohomology: 
\begin{equation*}
\xymatrix{
H^{q_b}( \SGK,\tM_{\lambda,E})  \ar[d]^{\r^*}  \\ 
H^{q_b}(\PBSC,\tM_{\lambda,E})  \ar[d]^{\fR_{\sigma_f, \sigma'_f}^b}  \\ 
I_b^\place(\sigma_f, \sigma'_f)_{P, w} \ \oplus \ I_b^\place(\sigma'_f(n) , \sigma_f(-n'))_{Q, w'}
}\end{equation*}
Recall, from Thm.\,\ref{thm:manin-drinfeld},  that  
$I^\place_b(\ts,\pts)_{P, w} \oplus  I^\place_b(\sigma_f'(n), \sigma_f(-n'))_{Q, w^\prime}$   
is a $E$-vector space of dimension $2{\mathsf  k}$.  In the self-associate case replace $Q$ by $P$.  
The proof of the main result on Eisenstein cohomology stated below also needs the analogue of the above maps for cohomology in degree $q_t$ for 
the coefficient system $\tM_{\lambda^\v,E}.$

\medskip
\begin{thm}
\label{thm:rank-one-eis}
For brevity, let 
$$
\fI^b(\sigma_f, \sigma_f') \ :=  \ 
\fR_{\sigma_f, \sigma'_f}^b(H^{q_b}_{\Eis}(\PBSC, \tM_{\lambda,E})), \quad 
\fI^t(\sigma_f, \sigma_f')^\v \  :=  \ 
\fR_{\sigma_f, \sigma'_f}^t(H^{q_t}_{\Eis}(\PBSC,\tM_{\lambda^\v,E})).
$$

\begin{enumerate}
\item In the non-self-associate cases ($n \neq n'$) we have: 

\smallskip

\begin{enumerate}
\item $\fI^b(\sigma_f, \sigma_f')$ is an $E$-subspace of dimension ${\mathsf  k}$. 

\smallskip

\item  $\fI^t(\sigma_f, \sigma_f')^\v$  is an $E$-subspace of dimension ${\mathsf  k}$. 

\end{enumerate}

\smallskip

\item In the self-associate case ($n=n'$) the same assertions hold by putting $Q = P.$ 
\end{enumerate}
\end{thm}

It helps to have a mental picture of when ${\mathsf  k} = 1$, i.e., then 
$\fI^b(\sigma_f, \sigma_f')$ is a line in the ambient two-dimensional space
$I^\place_b(\ts,\pts)_{P, w} \oplus  I^\place_b(\sigma_f'(n), \sigma_f(-n'))_{Q, w^\prime}$; 
as will be seen later, the ``slope'' of this line contains arithmetic information about $L$-values.

\begin{proof}[A very brief sketch of proof]
The proof works exactly as explained in \cite[Sect.\,6.2.2]{harder-raghuram-book}, and involves two basic steps: 
\begin{enumerate}
\item[(i)] The first step is to show that both $\fI^b(\sigma_f, \sigma_f')$ and 
$\fI^t(\sigma_f, \sigma_f')^\v$ have dimension at least ${\mathsf  k}$; this is achieved by going to a transcendental level and appealing to Langlands's constant term theorem and producing enough cohomology classes in the image. The essential features are reviewed in Sect.\,\ref{sec:l-val-eis-coh} below, and for more details the reader is referred to \cite[Sect.\,6.3.7]{harder-raghuram-book}.  

\medskip
\item[(ii)] The second step, after invoking properties of the Poincar\'e duality pairing reviewed above, is to show that both 
$\fI^b(\sigma_f, \sigma_f')$ and $\fI^t(\sigma_f, \sigma_f')^\v$ have dimension exactly 
${\mathsf  k}.$ This step works exactly as in \cite[Sect.\,6.2.2.2]{harder-raghuram-book}. 
\end{enumerate}
\end{proof}

\medskip
\subsubsection{\bf $L$-values and rank-one Eisenstein cohomology} 
\label{sec:l-val-eis-coh}
The key ingredient in the main theorem on the rationality of these $L$-values 
is the role played by the $L$-values in the above result on rank-one Eisenstein cohomology. 

\medskip
\paragraph{\bf Induced representations}
Let $\sigma$ (resp., $\sigma'$) be a cuspidal automorphic representation of $G_n(\A)$ (resp., $G_{n'}(\A)$). The relation with the previous `arithmetic' notation is that given $\sigma_f \in \Coh_{!!}(G_n,\mu)$ and given $\iota : E \to \C$, think of 
${}^\iota\sigma_f$ to be the finite part of a cuspidal automorphic representation ${}^\iota\sigma$, etc. The $\iota$ is fixed and it is suppressed until otherwise mentioned. 
Consider the induced representation $I_P^G(s, \sigma \otimes \sigma')$ consisting of all smooth functions 
$f : G(\A) \to V_{ \sigma} \otimes V_{ \sigma'}$ such that 
\begin{equation}
\label{eqn:f-in-induced-repn}
f(mug) \ = \ 
\vert\delta_P\vert(m)^{\frac{1}{2}} \vert\delta_P\vert(m)^{\frac{s}{N}} \, ( \sigma \otimes \sigma')(m) \, f(g), 
\end{equation}
for all $m \in M_P(\A)$, $u \in U_P(\A)$, and $g \in G(\A)$; where $V_\sigma$ (resp., $V_{\sigma'}$) is the subspace inside the space of cusp forms on $G_n(\A)$ (resp., $G_{n'}(\A)$) realizing the representation $ \sigma$ (resp., 
$ \sigma'$). In other words,
$I_P^G(s, \sigma \otimes  \sigma') = 
{\mathrm{Ind}}_{P(\A)}^{G(\A)}((\sigma  \otimes |\ |^{\frac{n'}{N}s}) \otimes ( \sigma' \otimes |\ |^{\frac{-n}{N}s})),$
where ${\mathrm{Ind}}_P^G$ denotes the normalized parabolic induction. In terms of 
algebraic or un-normalized induction, we have
\begin{equation}
\label{eqn:a-ind-s=-N/2}
I_P^G(s, \sigma \otimes  \sigma') = {}^{\mathrm a}{\mathrm{Ind}}_{P(\A)}^{G(\A)}
(( \sigma \otimes |\ |^{\frac{n'}{N}s + \frac{n'}{2}}) \otimes 
( \sigma'  \otimes |\ |^{\frac{-n}{N}s - \frac{n}{2}})). 
\end{equation}
Specifically, note the {\it point of evaluation} $s_0 = -N/2$: 
$$
{}^{\mathrm a}{\mathrm{Ind}}_{P(\A)}^{G(\A)}
\left( \sigma \otimes \sigma' \right) \ = \ 
I_P^G(s, \sigma \otimes \sigma')|_{s = -N/2}, \ \ 
{}^{\mathrm a}{\mathrm{Ind}}_{Q(\A)}^{G(\A)}
\left( \sigma'(n) \otimes \sigma(-n') \right) \ = \ 
I_Q^G(s, \sigma' \otimes \sigma)|_{s = N/2}. 
$$
The finite parts of the induced representations appear in boundary cohomology. 

\medskip
\paragraph{\bf Standard intertwining operators}
There is an element $w_{P} \in W_G$, the Weyl group of $G$, which is uniquely determined by the property
$w_{P}(\bfpi_G - \{\bfgreek{alpha}_P\}) \subset \bfpi_G$ and $w_{P}(\bfgreek{alpha}_P) < 0.$ If we write 
$w_{P} = (w_{P_0}^\tau)_{\tau:F \to E}$, then for each $\tau$, as a permutation matrix in $\GL_N$, we have 
$w_{P_0}^\tau = \left[\begin{smallmatrix}& 1_n \\ 1_{n'} & \end{smallmatrix}\right].$
The parabolic subgroup $Q$, which is associate to $P$, corresponds to 
$w_{P}(\bfpi_G - \{\bfgreek{alpha}_P\}).$ Since 
$w_{P_0}^\tau\!{}^{-1} \, \mathrm{diag}(h,h') \, w_{P_0}^\tau = \mathrm{diag}(h',h)$ for all $ \mathrm{diag}(h,h') \in M_{P_0^\tau}$, we get  
$w_{P}(\sigma \otimes \sigma') =  \sigma' \otimes \sigma$ as a representation of $M_{Q}(\A)$. 
The global  standard intertwining operator:
$$
T_{\mathrm{st}}^{PQ}(s, \sigma \otimes \sigma') :  I_P^G(s, \sigma \otimes \sigma') \ \longrightarrow \  
I_Q^G(-s, \sigma' \otimes \sigma)
$$
is given by the integral
\begin{equation}
\label{eqn:global-in-op}
(T_{\mathrm{st}}^{PQ}(s, \sigma \otimes \sigma')f)(\ul g) = \int_{U_{Q}(\A)}f(w_{P_0}^{-1} \, \ul u \, \ul g)\, d\ul u. 
\end{equation}
See \ref{sec:choice-of-mesaure} for the choice of measure $d \ul u.$  
Abbreviate $T_{\mathrm{st}}^{PQ}(s, \sigma \otimes \sigma')$ as $T_{\mathrm{st}}(s, \sigma \otimes \sigma').$
The global standard intertwining operator factorizes as a product of local standard intertwining operators:
$T_{\mathrm{st}}(s, \sigma \otimes \sigma') = 
\otimes_v T_{\mathrm{st}}(s, \sigma_v \otimes \sigma'_v)$
where the local operator 
is given by a similar local integral. (At an archimedean place, the effect of this operator in relative Lie algebra cohomology has already been described in Sect.\,\ref{sec:T-st-infinity}.)

\medskip
\paragraph{\bf Eisenstein series}
\label{sec:eisenstein-series}

Let $f \in I_P^G(s, \sigma \times \sigma^\prime)$, for $\ul g \in G(\A)$ the value 
$f(\ul g)$ is a cusp form on $M_P(\A).$ By the defining equivariance property of $f$, the complex number 
$f(\ul g)(\ul m)$ determines and is determined by $f(\ul m \ul g)(\ul 1)$ for any $\ul m \in M_P(\A).$ 
Henceforth,  
$f \in I_P^G(s, \sigma \times \sigma')$ will be identified with the complex valued function $\ul g \mapsto f(\ul g)(\ul 1),$ i.e., one has embedded: 
$$
I_P^G(s, \sigma \times \sigma^\prime) \ \hookrightarrow \ 
\cC^\infty\left(U_P(\A)M_P(\Q)\backslash G(\A), \omega_\infty^{-1}\right) 
\ \subset \ \cC^\infty\left(P(\Q)\backslash G(\A), \omega_\infty^{-1}\right), 
$$
where $\omega^{-1}_\infty$ is a simplified notation for the central character of $\sigma \otimes \sigma'$ restricted to 
$S(\R)^\circ.$ If $\sigma_f \in \Coh(G_n, \mu)$, $\sigma_f' \in  \Coh(G_{n',} \mu')$ and $\iota : E \to \C$ then 
$\omega_\infty$ is the product of the central characters $\omega_{\M_{{}^\iota\! \mu}} \omega_{\M_{{}^\iota\! \mu'}}$ restricted to $S(\R)^\circ.$ 
Given $f \in I_P^G(s, \sigma \times \sigma')$, thought of as a function on $P(\Q)\backslash G(\A)$, define the corresponding Eisenstein series 
$\Eis_P(s, f) \in \cC^\infty\left(G(\Q)\backslash G(\A), \omega_\infty^{-1}\right)$ by the usual averaging over $P(\Q)\backslash G(\Q)$: 
\begin{equation}
\label{eqn:Eis-P}
\Eis_P(s, f)(\ul g) \ := \ \sum_{\gamma \in P(\Q)\backslash G(\Q)} f(\gamma \ul g), 
\end{equation}
which is convergent if $\Re(s) \gg 0$ and has meromorphic continuation to the entire complex plane. 
This provides an intertwining operator
$$
\Eis_P(s, \sigma\times \sigma^\prime) \ : \ I_P^G(s, \sigma \times \sigma^\prime)  \ \longrightarrow \ 
\cC^\infty\left(G(\Q)\backslash G(\A), \omega_\infty^{-1}\right); 
$$
denote $\Eis_P(s, \sigma\times \sigma^\prime)(f)$ simply as $\Eis_P(s, f).$ 
To construct a map in cohomology, one needs to evaluate at $s=-N/2,$ begging the question whether the
Eisenstein series is holomorphic at $s=-N/2.$ For this, it is well-known that one has to show that the constant term of the Eisenstein series is holomorphic at $s=-N/2.$

\medskip
\paragraph{\bf The constant term map}
\label{sec:constant-term-map} 
 
For the parabolic subgroup $Q$, the constant term map,  denoted
$\cF^Q : \cC^\infty(G(\Q)\backslash G(\A), \omega_\infty^{-1}) 
\to \cC^\infty(M_Q(\Q)U_Q(\A)\backslash G(\A), \omega_\infty^{-1}),$ is given by: 
\begin{equation}
\label{eqn:constant-term-map}
\cF^Q(\phi)(\ul g) \ = \ 
\int\limits_{U_Q(\Q)\backslash U_Q(\A)} \phi(\ul u \, \ul g)\, d\ul u.
\end{equation}

\medskip
\paragraph{\bf The choice of the global measure $d\ul u$}
\label{sec:choice-of-mesaure} 

In the integrals defining the intertwining operator \eqref{eqn:global-in-op} and the constant term map \eqref{eqn:constant-term-map} the 
choice of measure $d\ul u$ on $U_Q(\A)$ needs to be fixed, where $U_Q =  \Res_{F/\Q}(U_{Q_0})$ is the unipotent radical of the maximal parabolic subgroup $Q$; 
recall that $Q_0$ is the standard maximal parabolic subgroup of $\GL(N)$ corresponding to $N = n'+n$. 
To begin, take the global measure ${}^L\!d\ul u$ on $U_Q(\A) = U_{Q_0}(\A_F)$ 
as a product over the coordinates of $U_{Q_0}$ of the additive measure $d\ul x$ on $\A_F$, which in turn is a product $\prod_v dx_v$ of local additive measures $dx_v$ on $F_v^+$; for a finite place $v$ normalise $dx_v$ by ${\rm vol}(\cO_v) = 1$, where $\cO_v$ is the ring of integers of $F_v$, 
and for an archimedean $v$ take $dx_v$ as the Lebesgue measure on $\C.$ The notation ${}^L\!d\ul u$ is to suggest that this measure is well-suited 
for the purposes of the analytic theory of $L$-functions. 
For the constant term map \eqref{eqn:constant-term-map} to correspond to the restriction map in cohomology, the global measure should be normalised by asking 
${\rm vol}(U_{Q_0}(F)\backslash U_{Q_0}(\A_F)) = 1;$ see Borel \cite[Sect.\,6]{borel-intro}. Hence, consider the global measure on 
$U_Q(\A) = U_{Q_0}(\A_F)$:
$$
d\,\ul u \ := \ \frac{1}{{\rm vol}_{{}^L\!d\ul u}(U_{Q_0}(F)\backslash U_{Q_0}(\A_F))} \, {}^L\!d\ul u. 
$$
Of course, 
${\rm vol}_{{}^L\!d\ul u}(U_{Q_0}(F)\backslash U_{Q_0}(\A_F)) = 
{\rm vol}_{d\ul x}(F\backslash \A_F)^{\dim(U_{Q_0})} = {\rm vol}_{d\ul x}(F\backslash \A_F)^{nn'}.$
Recall some classical algebraic number theory  (see Tate \cite[Sect.\,4.1]{tate-thesis}) for the volume of $F\backslash \A_F.$ 
For the measure $d\ul x$ on $\A_F$, the volume of a fundamental domain for the action of $F$ on $\A_F$ is $|\delta_{F/\Q}|^{1/2}.$ 
If the set $\Hom(F,\C)$ of complex embeddings of $F$ is enumerated as, say, $\{\sigma_1,\dots, \sigma_{d_F} \}$, and 
suppose $\{\omega_1,\dots,\omega_{d_F}\}$ is a $\Q$-basis of $F$, then the absolute discriminant of $F$ is defined as 
$\delta_{F/\Q} =  \det[\sigma_i(\omega_j)]^2.$ The square root of the absolute value of the discriminant, $|\delta_{F/\Q}|^{1/2},$ 
as an element of $\R^\times/\Q^\times,$ is independent of the enumeration and the choice of basis. For the main theorem on $L$-values, the choice of measure can be changed by a nonzero rational number, which will still give the same rationality results and the reciprocity laws as in Thm.\,\ref{thm:main}. Define the global measure on
$U_Q(\A) = U_{Q_0}(\A_F)$ as:
\begin{equation}
\label{eqn:normalised-measure}
d\,\ul u \ := \ 
|\delta_{F/\Q}|^{- \tfrac{n n'}{2}} \, {}^L\!d\ul u. 
\end{equation}

\medskip
\paragraph{\bf Theorem of Langlands on the constant term of an Eisenstein series}
\label{sec:constant-term-eisenstein-series}

\medskip
\begin{thm}[Langlands]
\label{thm:langlands}
Let $f  \in I_P^G(s, \sigma \times \sigma').$ 
\begin{enumerate}
\item In the non-self-associate cases ($n \neq n'$), one has: \\
$\cF^Q \circ \Eis_P(s, f) \ = \ T_{\mathrm{st}}(s, \sigma \times \sigma')(f).$ 
    
\smallskip
\item In the self-associate cases ($n = n'$ and $P = Q$), one has: \\
$\cF^P \circ \Eis_P(s, f) \ = \ f + T_{\mathrm{st}}(s, \sigma \times \sigma')(f).$ 
\end{enumerate}
\end{thm}

\medskip
Suppose $f = \otimes_v f_v$ is a pure tensor in $I_P^G(s, \sigma \otimes \sigma')$, and 
for $v \notin \place$ suppose $f_v = f_v^0$ is the normalized spherical vector (normalized to take the value $1$ on the identity), 
and similarly, $\tilde f_v^0$ is such a vector in the $v$-th component of $I_Q^G(-s, \sigma' \otimes \sigma)$, then 
from \cite[Thm.\,6.3.1]{shahidi-book} we have the fundamental analytic identity
\begin{equation}
\label{eqn:basic-analytic-identity}
\cF^Q (\Eis_P(s, f)) \ = \ 
|\delta_{F/\Q}|^{- \tfrac{n n'}{2}} \,
\frac{L^\place(s, \sigma \times \sigma'^{\mathsf v})}{L^\place(s+1, \sigma \times \sigma'^{\mathsf v})} \, 
\otimes_{v \notin \place} \tilde f_v^0 \otimes_{v \in \place} 
T_{\mathrm{st}}(s, \sigma_v \otimes \sigma'_v)f_v.
\end{equation}
The proof of the main theorem on the arithmetic of special values of $L(s, \sigma \times \sigma'^{\mathsf v})$ comes from seeing the contribution of this identity in cohomology.

\medskip
\paragraph{\bf Holomorphy of the Eisenstein series at the point of evaluation}
\label{sec:holomorphy-eisenstein-series-at-point-N/2}

Given weights $\mu \in X^+_{00}(T_n \times E)$ and 
$\mu' \in X^+_{00}(T_{n'} \times E)$ and strongly-inner Hecke-summands 
$\sigma_f \in \Coh_{!!}(G_n, \mu)$ and $\sigma'_f \in \Coh_{!!}(G_{n'}, \mu'),$ recall then that ${}^\iota\sigma$ and ${}^\iota\sigma'$ are cuspidal automorphic representations of 
$G_n(\A) = \GL_n(\A_F)$ and $G_{n'}(\A) = \GL_{n'}(\A_F)$ for any $\iota : E \to \C$.  
The pair $(\mu, \mu')$ of weights 
is  said to be \textit{on the right of the unitary axis} if the abelian width is bounded above by the point of evaluation: 
$$
a(\mu,\mu') \leq -N/2,
$$ 
To explain the terminology, it is clear from the definition of the cuspidal parameters \eqref{eqn:cuspidal-parameters-alpha}, \eqref{eqn:cuspidal-parameters-beta}, 
and the archimedean representation \eqref{eqn:j-lambda-v}, 
that ${}^\iota\sigma = {}^\iota\sigma_u \otimes |\!| \ |\!|^{-\w/2}$ for a unitary cuspidal representation ${}^\iota\sigma_u;$ similarly, 
${}^\iota\sigma' = {}^\iota\sigma'_u \otimes |\!| \ |\!|^{-\w'/2}.$  Hence, we have 
$$
L(s, {}^\iota\sigma \times {}^\iota\sigma'^\v) \ = \ L(s - a(\mu, \mu'), {}^\iota\sigma_u \times {}^\iota\sigma_u'^\v). 
$$
Now, suppose $(\mu, \mu')$ is on the right of the unitary axis, then for the $L$-value in the denominator of the right hand side of \eqref{eqn:basic-analytic-identity} at the point of evaluation 
$s = - \tfrac{N}{2}$ we have: 
$$
L^\place(-  \tfrac{N}{2}+1, {}^\iota\sigma \times {}^\iota\sigma'^\v) \ = \ L^\place(1 -  \tfrac{N}{2} - a(\mu, \mu'), {}^\iota\sigma_u \times {}^\iota\sigma_u'^\v) \ \neq \ 0,
$$
since $1 -  \tfrac{N}{2} - a(\mu, \mu') \geq 1,$ by a nonvanishing result for Rankin--Selberg $L$-functions recalled in 
\cite[Thm.\,7.1, (4)]{harder-raghuram-book}.
Of course, the nonvanishing of this $L$-value, the holomorphy of the Eisenstein series $\Eis_P(s,f)$ or of its constant term are all intimately linked. We have the following well-known result (\cite{moeglin-waldspurger} and \cite[Chap.\,IV \S\,5]{harish-chandra}).

\medskip 
\begin{thm}
\label{thm:at-least-one-holomorphic}
Suppose we are given weights $\mu \in X^+_{00}(T_n \times E)$ and 
$\mu' \in X^+_{00}(T_{n'} \times E),$ 
strongly-inner Hecke-summands $\sigma_f \in \Coh_{!!}(G_n, \mu)$ and $\sigma'_f \in \Coh_{!!}(G_{n'}, \mu'),$ and 
an $\iota : E \to \C.$ Assume that $(\mu, \mu')$ is on the right 
of the unitary axis. Then $\Eis_P(s,f)$ is holomorphic at $s = -N/2$, unless we are in the exceptional case
   $n=n^\prime $ and ${}^\iota\sigma^\prime = {}^\iota\sigma\otimes | \ |^{-n-1}.$
 \end{thm}

\medskip

The poles on the right of the unitary axis are simple and contribute to the residual spectrum
  \cite{harish-chandra} and then the assertion follows from \cite{ moeglin-waldspurger}. The exceptional case exactly corresponds to when the numerator of 
  of the right hand side of \eqref{eqn:basic-analytic-identity} at the point of evaluation is a pole at $1$, i.e., when 
  $L(-  \tfrac{N}{2} - a(\mu, \mu'), {}^\iota\sigma_u \times {}^\iota\sigma_u'^\v)$ is a pole at $1$; which is possible only when $n=n^\prime $ and ${}^\iota\sigma^\prime = {}^\iota\sigma\otimes |\!| \ |\!|^{-n-1}.$ 
To parse this further:  If ${}^\iota\sigma^\prime = {}^\iota\sigma \otimes |\!| \ |\!|^r$ for any integer $r$, then $\mu' = \mu-r$. Then the cuspidal parameters for $\mu'$ and $\mu$ are related thus: 
$\alpha'^v_i = \alpha^v_i+r, \ \beta'^v_i = \beta^v_i+r;$ hence the cuspidal width $\ell(\mu, \mu') = 0.$ For the main theorem on $L$-values, we will assume the conditions imposed by the combinatorial lemma, and in particular we will have $\ell(\mu,\mu') \geq 2$ to guarantee at least two critical values. Hence, the exceptional case will not be relevant to us.

\medskip
\subsection{The main theorem on $L$-values}
\label{sec:main-thm-L-values}

\subsubsection{\bf Statement of the main theorem}

\begin{thm}
\label{thm:main}
Let $n$ and $n'$ be two positive integers. 
Let $F$ be a totally imaginary field, and $E$ a finite Galois extension of $\Q$ that contains a copy of $F$. 
Consider strongly-pure weights $\mu \in X^+_{00}(T_n \times E)$ and 
$\mu' \in X^+_{00}(T_{n'} \times E).$ Let $\sigma_f \in \Coh_{!!}(G_n, \mu)$ and $\sigma'_f \in \Coh_{!!}(G_{n'}, \mu')$ be strongly-inner Hecke-summands, 
and assume that $E$ is large enough to contain all the Hecke-eigenvalues for $\sigma_f $ and $\sigma'_f.$
Let $\iota : E \to \C$ be an embedding. Recall then that ${}^\iota\sigma$ and ${}^\iota\sigma'$ are cuspidal automorphic representations of 
$G_n(\A) = \GL_n(\A_F)$ and $G_{n'}(\A) = \GL_{n'}(\A_F)$, respectively. Put $N = n+n'.$ 

\smallskip

Suppose that $m \in \tfrac{N}{2} + \Z$ is such that both $m$ and $1+m$ are critical for the Rankin--Selberg $L$-function $L(s, {}^\iota\sigma \times {}^\iota\sigma'^\v).$

\begin{enumerate}
 \medskip
\item If for some $\iota$, $L(1 + m, {}^\iota\sigma \times {}^\iota\sigma'^\v) = 0$, then $1+m -a(\mu,\mu') = \tfrac12$ and 
   $$
   L(1 + m, {}^\iota\sigma \times {}^\iota\sigma'^\v) \ = \ L(\tfrac12, {}^\iota\sigma_u \times {}^\iota\sigma'^\v_u) = 0
   $$ 
   is the central critical value. Furthermore,  $L(1+m, {}^\iota\sigma \times {}^\iota\sigma'^\v) = 0$ for every $\iota.$

   \medskip

   \item Assume $F$ is in the {\bf CM}-case. 
   Suppose that $L(1 +m, {}^\iota\sigma \times {}^\iota\sigma'^\v) \neq 0$, then we have 
   $$
  |\delta_{F/\Q}|^{- \tfrac{n n'}{2}} \,
   \frac{L(m, {}^\iota\sigma \times {}^\iota\sigma'^\v)}{L(1+m, {}^\iota\sigma \times {}^\iota\sigma'^\v)} \ \in \ \iota(E).
   $$
  Since  $L(m, {}^\iota\sigma \times {}^\iota\sigma'^\v) = L(-N/2, {}^\iota\sigma(N/2+m) \times {}^\iota\sigma'^\v)$, the pair 
  $(\mu-N/2-m, \mu')$ of weights satisfies Lem.\,\ref{lem:comb-lemma} giving a balanced Kostant representative $w \in W^P$. Let 
  $w' \in W^Q$ be determined by Lem.\,\ref{lem:kostant-P-Q}. 
   Furthermore, for any $\gamma \in \Gal(\bar\Q/\Q)$ we have 
   $$
   \gamma\left(
   |\delta_{F/\Q}|^{- \tfrac{n n'}{2}} \,
   \frac{L(m, {}^\iota\sigma \times {}^\iota\sigma'^\v)}{L(1+m, {}^\iota\sigma \times {}^\iota\sigma'^\v)}
   \right) \ = \ 
   \varepsilon_{\iota, w}(\gamma) \cdot \varepsilon_{\iota, w'}(\gamma) \cdot
   |\delta_{F/\Q}|^{- \tfrac{n n'}{2}} \,
   \frac{L(m, {}^{\gamma\circ\iota}\sigma \times {}^{\gamma \circ\iota} \sigma'^\v)}{L(1+m, {}^{\gamma \circ \iota}\sigma \times {}^{\gamma \circ\iota}\sigma'^\v)}. 
   $$

  \medskip
   \item Assume $F$ is in the {\bf TR}-case. Then $nn'$ is even. 
   Suppose that $L(1 +m, {}^\iota\sigma \times {}^\iota\sigma'^\v) \neq 0$, then we have 
   $$
   \frac{L(m, {}^\iota\sigma \times {}^\iota\sigma'^\v)}{L(1+m, {}^\iota\sigma \times {}^\iota\sigma'^\v)} \ \in \ \iota(E).
   $$
   Furthermore, for any $\gamma \in \Gal(\bar\Q/\Q)$ we have 
   $$
   \gamma\left(
    \frac{L(m, {}^\iota\sigma \times {}^\iota\sigma'^\v)}{L(1+m, {}^\iota\sigma \times {}^\iota\sigma'^\v)}
   \right) \ = \ 
     \frac{L(m, {}^{\gamma\circ\iota}\sigma \times {}^{\gamma \circ\iota} \sigma'^\v)}{L(1+m, {}^{\gamma \circ \iota}\sigma \times {}^{\gamma \circ\iota}\sigma'^\v)}. 
   $$
\end{enumerate}
\end{thm}

\medskip
As mentioned in the introduction, statement (1), is originally due to M\oe glin \cite[Sect.\,5]{moeglin}; our proof below is independent of \cite{moeglin}. More generally, 
the assertion on the vanishing of the central critical value is a special case of 
Deligne's conjecture \cite[Conj.\,2.7, (ii)]{deligne} based on a suggestion of Benedict Gross, that the order of vanishing at 
a critical point is independent of the embedding $\iota$.

\medskip
\subsubsection{\bf Proof of the main theorem}

\medskip
\paragraph{\bf Finiteness of the relevant $L$-values} 
The paragraph after Thm.\,\ref{thm:at-least-one-holomorphic} says that if $L(s, {}^\iota\sigma\otimes \times {}^\iota\sigma'^\v)$
has a pole at $m \in \tfrac{N}{2} + \Z$ then the cuspidal width $\ell(\mu,\mu') = 0$; a situation which is ruled out by the hypothesis
that requires two critical points. Hence all the $L$-values under consideration are finite; a fact that will be used without further comment.

\medskip
\paragraph{\bf It suffices to prove Thm.\,\ref{thm:main} for the point of evaluation $m = -N/2$} 
\label{para:comb-lemma-reduction}
If the theorem is known for the critical points $s = -\tfrac{N}{2}$ and $1-\tfrac{N}{2}$ and for all possible 
$\mu$, $\mu'$, $\sigma_f$, $\sigma'_f$, then one can deduce the theorem 
for any pair of successive critical points $m, \, m+1$ for a given $\sigma_f$ and $\sigma'_f$. 
This follows from using Tate-twists (Sect.\,\ref{sec:tate-twists}) and the combinatorial lemma. 
Take any integer $r$ and replace $\mu$ by $\mu-r \bfgreek{delta}_n$ and $\sigma_f$ by $\sigma_f(r).$ Note that 
${}^\iota(\sigma_f(r)) = {}^\iota\sigma_f \otimes |\!|\ |\!|^r.$  
The condition that 
$-\tfrac{N}{2}$ and $1-\tfrac{N}{2}$ are critical for 
$L(s, {}^\iota\sigma\otimes |\!|\ |\!|^r \times {}^\iota\sigma'^\v) = L(s +r, {}^\iota\sigma\otimes \times {}^\iota\sigma'^\v)$, after the combinatorial lemma, bounds the abelian width $a(\mu-r\bfgreek{delta}_n, \mu')$ by the cuspidal width $\ell(\mu-r\bfgreek{delta}_n, \mu')$ 
as in (2) of Lem.\,\ref{lem:comb-lemma}. 
Now, the crucial point is that, whereas for the abelian width one has $a(\mu-r\bfgreek{delta}_n, \mu') = a(\mu, \mu') - r$, but for the cuspidal width one has independence of $r$ in as much as 
$\ell(\mu-r\bfgreek{delta}_n, \mu') = \ell(\mu, \mu')$. 
This bounds the possible twisting integers $r$ above and below as: 
$$
- \frac{N}{2} + 1 - \frac{\ell(\mu,\mu')}{2} - a(\mu, \mu')
 \ \leq \ -r  \ \leq \  
- \frac{N}{2} -1 + \frac{\ell(\mu,\mu')}{2} - a(\mu, \mu').
$$
As $r$ ranges over this set, 
using the critical set described in Prop.\,\ref{prop-crit-mu-mu'}, one sees that 
$$
  \frac{L(- \tfrac{N}{2}, {}^\iota\sigma(r) \times {}^\iota\sigma'^\v)}{L(1- \tfrac{N}{2}, {}^\iota\sigma(r) \times {}^\iota\sigma'^\v)} \ = \ 
   \frac{L(r - \tfrac{N}{2}, {}^\iota\sigma \times {}^\iota\sigma'^\v)}{L(r+1- \tfrac{N}{2}, {}^\iota\sigma \times {}^\iota\sigma'^\v)} 
$$
runs over the set of all successive pairs of critical points; {\it no more and no less!} The number of possible $r$ is 
$\ell(\mu,\mu')-1$ which is exactly the number of pairs of successive critical points.

\medskip
\paragraph{\bf Being on the right versus on the left of the unitary axis} 
Suppose that $(\mu,\mu')$ is on the right of the unitary axis: $a(\mu,\mu') \leq -\tfrac{N}{2}.$ Then (1) is vacuously true since 
$L(1 - \tfrac{N}{2}, {}^\iota\sigma \times {}^\iota\sigma'^\v) = L(1 - \tfrac{N}{2} - a(\mu,\mu'), {}^\iota\sigma_u \times {}^\iota\sigma'^\v_u) \neq 0$ by a well-known nonvanishing result for Rankin--Selberg $L$-functions as already mentioned before. 
Next, recall that the Eisenstein series $\Eis_P(s,f)$ is holomorphic at $s = -\tfrac{N}{2}$, and at this point of evaluation (2) and (3) will be proved below. 
Granting this, suppose, on the other hand, that $(\mu, \mu')$ is on the left of the unitary axis, i.e., $a(\mu,\mu') > -\tfrac{N}{2}.$ Then, 
$a(\mu',\mu) < \tfrac{N}{2},$ i.e., $(\mu',\mu)$ is on the right of the unitary axis for the point of evaluation is $\tfrac{N}{2}$, so we get 
the holomorphy of $\Eis_Q(s,f)$ at $s = \tfrac{N}{2}$ and whence statement (2) for 
$L(s, {}^\iota\sigma' \times {}^\iota\sigma^\v)$ and for $s = \tfrac{N}{2}.$ And as above, (1) is in fact vacuously true, because 
$L(1+\tfrac{N}{2}, {}^\iota\sigma' \times {}^\iota\sigma^\v) = L(1+\tfrac{N}{2} - a(\mu',\mu), {}^\iota\sigma'_u \times {}^\iota\sigma^\v_u) \neq 0.$  Statement (2) for this situation says: 
$$
 |\delta_{F/\Q}|^{- \tfrac{n n'}{2}} \,  \frac{L(\tfrac{N}{2}, {}^\iota\sigma' \times {}^\iota\sigma^\v)}{L(1+\tfrac{N}{2}, {}^\iota\sigma' \times {}^\iota\sigma^\v)}
 \ \in \ \iota(E), 
$$
where the $L$-value in the denominator is not zero. Suppose the $L$-value in the numerator is $0$ (which can happen in the special case 
$\tfrac{N}{2} - a(\mu',\mu) = \tfrac12$) then the Galois equivariance in (2) implies that $L(\tfrac{N}{2}, {}^\iota\sigma' \times {}^\iota\sigma^\v) = 0$ for 
{\it every} $\iota.$
Applying the functional equation (\cite[(3), Thm.\,7.1]{harder-raghuram-book}) to the above ratio of $L$-values we get:
$$
|\delta_{F/\Q}|^{- \tfrac{n n'}{2}} \cdot
\frac{\varepsilon(\tfrac{N}{2}, {}^\iota\sigma' \times {}^\iota\sigma^\v)}
{\varepsilon(1+\tfrac{N}{2}, {}^\iota\sigma' \times {}^\iota\sigma^\v)} \cdot 
\frac{L(1- \tfrac{N}{2}, {}^\iota\sigma \times {}^\iota\sigma'^\v)}
{L(- \tfrac{N}{2}, {}^\iota\sigma \times {}^\iota\sigma'^\v)} 
\ \in \ \iota(E).
$$
For brevity, let $\tau := {}^\iota\sigma' \times {}^\iota\sigma^\v$. The global epsilon-factor is a product of local factors as  
$\varepsilon(s, \tau) = \prod_v \varepsilon(s, \tau_v, \psi_v)$, where we have fixed an additive character $\psi : F \backslash \A_F \to \C^\times.$ (See, for example, \cite[Sect.\,10.1]{shahidi-book}.) 
At a non-archimedean place $v$, the local factor has the form $\varepsilon(s, \tau_v, \psi_v) = 
W(\tau_v) q_v^{(1/2 - s)(c(\tau_v)+c(\psi_v))}$, where $W(\tau_v)$ is the local root number, $q_v$ the cardinality of the residue field of $F_v$, 
and $c(\tau_v)$ and $c(\psi_v)$ are integers defined to be the conductoral exponents of the respective data; it follows that  
$\varepsilon(N/2, \tau_v, \psi_v)/\varepsilon(1+N/2, \tau_v, \psi_v)$ is an integral power of $q_v$. 
At an archimedean place, it follows from 
\cite[(4.7)]{knapp} that the local factor is a constant, and hence the ratio is $1$. Whence,  
$\varepsilon(\tfrac{N}{2}, {}^\iota\sigma' \times {}^\iota\sigma^\v)/\varepsilon(1+\tfrac{N}{2}, {}^\iota\sigma' \times {}^\iota\sigma^\v)$
is a nonzero integer, 
from which it follows that
\begin{equation}
\label{eqn:reciprocal}
|\delta_{F/\Q}|^{- \tfrac{n n'}{2}} \,
\frac{L(1- \tfrac{N}{2}, {}^\iota\sigma \times {}^\iota\sigma'^\v)}
{L(- \tfrac{N}{2}, {}^\iota\sigma \times {}^\iota\sigma'^\v)} 
\ \in \ \iota(E).
\end{equation}
From the functional equation it is clear that
$
L(1- \tfrac{N}{2}, {}^\iota\sigma \times {}^\iota\sigma'^\v) = 0 \iff 
L(\tfrac{N}{2}, {}^\iota\sigma' \times {}^\iota\sigma^\v) = 0, 
$
proving (1) when $(\mu, \mu')$ is on the left of the unitary axis. 
If $L(1- \tfrac{N}{2}, {}^\iota\sigma \times {}^\iota\sigma'^\v) \neq 0$, then taking reciprocal of the ratio on the left hand side of 
\eqref{eqn:reciprocal}, (2) will follows when $(\mu, \mu')$ is on the left of the unitary axis. (See also \cite[Sect.\,7.3.2.5]{harder-raghuram-book} for
a slightly different way to argue this point if $(\mu, \mu')$  is on the left of the unitary axis.) The discussion is the same for (3).

\medskip
\paragraph{\bf Proof of the rationality result}
\label{para:conclusion-proof-rationality}
After the above reductions, it suffices to prove (2) and (3) of Thm.\,\ref{thm:main} when $(\mu,\mu')$ is on the right of the unitary axis and for $m = -N/2.$ 
This involves Eisenstein cohomology. Assume henceforth that $(\mu,\mu')$ is on the right of the unitary axis. For now $F$ is any totally imaginary field 
(either {\bf CM} or {\bf TR}). 
By Thm.\,\ref{thm:rank-one-eis}, 
the subspace  $\fI^b(\sigma_f, \sigma_f')$, which is the image of global cohomology in the $2{\mathsf k}$-dimensional $E$-vector space 
$I^\place_b(\sigma_f,\sigma_f')_{P,w} \oplus  I^\place_b(\sigma'_f(n),\sigma_f(-n'))_{Q,w'}$, 
is a ${\mathsf k}$-dimensional $E$-subspace; and furthermore, from the proof of that theorem, 
we get an intertwining operator $T_{\mathrm{Eis}}(\sigma, \sigma')$ defined over $E$  
such that in the non-self-associate case ($n\neq n'$) we have: 
$$
\fI^b(\sigma_f, \sigma_f') \ = \ 
\left\{ \left(\xi, \, T_{\mathrm{Eis}}(\sigma, \sigma')(\xi)\right) \ | \ 
\xi \in I^\place_b(\sigma_f,\sigma_f')_{P,w} \right\}, 
$$
and in the self-associate case we will have: 
$$
\fI^b(\sigma_f, \sigma_f') \ = \ 
\left\{ \left(\xi, \, \xi+ T_{\mathrm{Eis}}(\sigma, \sigma')(\xi)\right) \ | \ 
\xi \in I^\place_b(\sigma_f,\sigma_f')_{P,w} \right\}.
$$
The idea of the proof is to take $T_{\mathrm{Eis}}(\sigma, \sigma')$ to a transcendental level, use the constant term theorem which gives $L$-values, and then descend back to an arithmetic level, giving a rationality result for the said $L$-values. 
Take an $\iota: E \to \C$, and consider $T_{\mathrm{Eis}}(\sigma, \sigma') \otimes_{E, \iota} \C$. We have:
$$
\fI^b({}^\iota\!\sigma_f, {}^\iota\!\sigma_f') \ = \ 
\left\{ \left(\xi \, , \, T_{\mathrm{st}}(-\tfrac{N}{2}, {}^\iota\!\sigma \otimes {}^\iota\!\sigma')^\bullet\xi \right) \ : \ 
\xi \in I^\place_b({}^\iota\!\sigma_f, {}^\iota\!\sigma_f')_{P, {}^\iota\!w}
 \right\}
$$
in the non-self-associate case, and with analogous modification for the self-associate case. Here, 
$T_{\mathrm{st}}(-\tfrac{N}{2}, {}^\iota\!\sigma \otimes {}^\iota\!\sigma')^\bullet$ is the map induced by the standard intertwining operator in relative Lie algebra cohomology. For brevity, let $\ul{\sigma} = \sigma \times \sigma'$. 
The global operator $T_{\mathrm{st}}(-\tfrac{N}{2}, {}^\iota\!\ul{\sigma})^\bullet$ factors into local standard intertwining operators. 
The discussion in \cite[Sect.\,7.3.2.1]{harder-raghuram-book} involving rationality properties of local standard intertwining operators at 
finite places goes through verbatim in our situation, and we have for 
$T_{\mathrm{Eis}}(\sigma, \sigma') \otimes_{E, \iota} \C$ the following expression (which is the exact analogue of 
\cite[(7.38)]{harder-raghuram-book}): 
\begin{multline}
  \label{eqn:Tst-prod-1}
 |\delta_{F/\Q}|^{- \tfrac{n n'}{2}} \,
\frac{L (-\tfrac{N}{2}, {}^\iota\sigma_f \times {}^\iota\sigma_f'^{\mathsf v})}
{L (1-\tfrac{N}{2}, {}^\iota\sigma_f \times {}^\iota\sigma_f'^{\mathsf v})} \, 
\cdot\bigotimes_{v\in {\place}_\infty} 
T_{\mathrm{st}}(-\tfrac{N}{2}, {}^\iota\ul{\sigma}_v)^\bullet \otimes \\ 
\otimes \Bigl(\bigl(\bigotimes_{v \in \place_f}  
T^{\mathrm{ar}}_{\mathrm{norm}}( \ul{\sigma}_v)(1)
\otimes \bigotimes_{v \not\in \place}T^{\mathrm{ar}}_{\mathrm{loc}}( \ul{\sigma}_v)(1) 
\Bigr)\otimes_{E,\iota}\C. 
\end{multline}
The local operators $T^{\mathrm{ar}}_{\mathrm{norm}}( \ul{\sigma}_v)(1)$ and $T^{\mathrm{ar}}_{\mathrm{loc}}( \ul{\sigma}_v)(1)$ are exactly as in {\it loc.\,cit.}; the point of immediate interest for us being that they are defined over $E$. 
For the archimedean component $T_{\mathrm{st}}(-\tfrac{N}{2}, {}^\iota\ul{\sigma}_v)^\bullet$, 
we use Prop.\,\ref{prop:basic-Tst-GLN-in-cohomology}, to get for 
$T_{\mathrm{Eis}}(\sigma, \sigma') \otimes_{E, \iota} \C$ the following expression involving values of the completed $L$-function: 
\begin{equation}
  \label{eqn:Tst-prod-complete}
|\delta_{F/\Q}|^{- \tfrac{n n'}{2}} \,
\frac{L (-\tfrac{N}{2}, {}^\iota\sigma \times {}^\iota\sigma'^{\mathsf v})}
{L (1-\tfrac{N}{2}, {}^\iota\sigma \times {}^\iota\sigma'^{\mathsf v})} \, \cdot \, 
\Bigl(\bigl(\bigotimes_{v \in \place_f}  
T^{\mathrm{ar}}_{\mathrm{norm}}( \ul{\sigma}_v)(1)
\otimes \bigotimes_{v \not\in \place}T^{\mathrm{ar}}_{\mathrm{loc}}( \ul{\sigma}_v)(1) 
\Bigr)\otimes_{E,\iota}\C. 
\end{equation}
We conclude that the complex number 
$|\delta_{F/\Q}|^{- \tfrac{n n'}{2}} \, 
L (-\tfrac{N}{2}, {}^\iota\sigma \times {}^\iota\sigma'^{\mathsf v})/L (1-\tfrac{N}{2}, {}^\iota\sigma \times {}^\iota\sigma'^{\mathsf v})$ is in 
$\iota(E).$ When $F$ is in case-{\bf TR}, existence of a critical point implies $nn'$ is even (Cor.\,\ref{cor:F-tot-real-nn'-even}), which forces
$|\delta_{F/\Q}|^{- \tfrac{n n'}{2}} \in \Q^\times.$ 

\medskip

\medskip
\paragraph{\bf Proof of reciprocity}
\label{para:conclusion-proof-reciprocity}

For Galois equivariance, apply $\gamma \in \Gal(\bar{\Q}/\Q)$, 
to the objects and maps in the first paragraph of 
Sect.\,\ref{sec:rank-one-eis-coh}, while keeping in mind the behaviour of cohomology groups as Hecke modules 
under changing the base field $E$. Assume now that $F$ is a totally imaginary field in the {\bf CM}-case. The claim is 
that Galois-action and Eisenstein operator $T_{\mathrm{Eis}}(\sigma, \sigma')$ intertwine as: 
\begin{equation}
\label{eqn:gal-equiv}
(1 \otimes \gamma) \circ \left(T_{\mathrm{Eis}}(\sigma, \sigma') \otimes_{E,\iota} \bar{\Q} \right) \ = \ 
\varepsilon_{\iota, w}(\gamma) \cdot \varepsilon_{\iota, w'}(\gamma) \cdot  
T_{\mathrm{Eis}}(\sigma, \sigma') \otimes_{E, \gamma \circ \iota} \bar{\Q}. 
\end{equation}
From this claim and \eqref{eqn:Tst-prod-complete} the reciprocity law will follow. To prove this claim, take $n \neq n'$ (the reader can easily modify the argument for the self-associate case) and 
consider the following diagram 
$$
\xymatrix{
I^\place_b(\sigma_f,\sigma_f')_{P,w} \otimes_{E,\iota} \bar{\Q} 
\ar[rrrr]^{T_{\mathrm{Eis}}(\sigma, \sigma') \otimes_{E,\iota} 1_{\bar{\Q}}}
 \ar[d]^{1 \otimes \gamma} 
& & & & 
I^\place_b(\sigma'_f(n),\sigma_f(-n'))_{Q,w'} \ar[d]^{1 \otimes \gamma} \otimes_{E,\iota} \bar{\Q} \\
I^\place_b(\sigma_f,\sigma_f')_{P,w} \otimes_{E, \gamma \circ \iota} \bar{\Q} 
\ar[rrrr]^{T_{\mathrm{Eis}}(\sigma, \sigma') \otimes_{E, \gamma \circ \iota} 1_{\bar{\Q}}}
 & & & &
 I^\place_b(\sigma'_f(n),\sigma_f(-n'))_{Q,w'} \otimes_{E, \gamma \circ \iota} \bar{\Q}
}
$$
The left (resp., right) vertical arrow introduces the signature $\varepsilon_{\iota, w}(\gamma)$ (resp., $\varepsilon_{\iota, w'}(\gamma)$), and the diagram commutes up to the product 
of these two signatures. For the left vertical arrow, recall from \ref{sec:manin-drinfeld-notations} that 
the induced module $I^\place_b(\sigma_f,\sigma_f')_{P,w}$ appears in boundary cohomology as 
$\aInd_{P(\A_f)}^{G(\A_f)} \left(H^{b_n^F+b_{n'}^F}(\SMP, \tM_{w \cdot \lambda, E})(\sigma_f \times \sigma'_f) \right)^{K_f}.$ Hence, 
$I^\place_b(\sigma_f,\sigma_f')_{P,w} \otimes_{E,\iota} \bar{\Q}$ appears as a Hecke-summand in 
$\aInd_{P(\A_f)}^{G(\A_f)} \left(H^{b_n^F+b_{n'}^F}(\SMP, \tM_{{}^\iota\!w \cdot {}^\iota\! \lambda, \bar\Q})\right)$. 
Recall from \eqref{eqn:galois-signature} that $\gamma$ maps the the highest weight vector of 
the coefficient system $\M_{{}^\iota\!w \cdot {}^\iota\! \lambda, \bar\Q}$ to $\varepsilon_{\iota, w}(\gamma)$ times 
the highest weight vector of the coefficient system $\M_{{}^{\gamma\circ\iota}\!w \cdot {}^{\gamma \circ \iota}\! \lambda, \bar\Q}$ explaining the homothety by 
$\varepsilon_{\iota, w}(\gamma)$ in the left vertical arrow. Similarly, the induced module
$I^\place_b(\sigma'_f(n),\sigma_f(-n'))_{Q,w'} \otimes_{E,\iota} \bar{\Q} $ appears in  
$\aInd_{P(\A_f)}^{G(\A_f)} \left(H^{b_n^F+b_{n'}^F}(\SMP, \tM_{{}^\iota\!w' \cdot {}^\iota\! \lambda, \bar\Q})\right)$ for $w'$ related to $w$ by Lem.\,\ref{lem:kostant-P-Q}.  
Hence, the right vertical arrow, one gets a homothety by $\varepsilon_{\iota, w'}(\gamma)$. Hence the claim, and whence 
the reciprocity law. 

\medskip
If the totally imaginary field $F$ is in case {\bf CM} then the proof of (2) is complete. In case of {\bf TR}, the rationality and the 
Galois-equivariance, as will be now shown, will simplify to give the corresponding statements in (3). Assume now that $F$ is {\bf TR}, and that $F_1$ is the maximal totally real subfield of $F$. 
Existence of a critical point implies $nn'$ is even by Cor.\,\ref{cor:F-tot-real-nn'-even}, hence $|\delta_{F/\Q}|^{- \tfrac{n n'}{2}} \in \Q^\times$ and so we may absorb it into $\iota(E)$ and ignore it 
from the Galois-equivariance. The simplified Galois equivariance in (3) follows from the following 

\medskip
\begin{lemma}
\label{lem:trivial-signature-TR}
Suppose that $F$ is {\bf TR}, then 
$\varepsilon_{\iota, w}(\gamma)  = 1 = \varepsilon_{\iota, w'}(\gamma).$
\end{lemma}

\medskip
\begin{proof}
Recall the notations from \ref{sec:notation-TR}: 
$\Sigma_{F_1} = \{\nu_1,\dots,\nu_{d_1}\}$, $\varrho : \Sigma_F \to \Sigma_{F_1}$, $k = 2k_1,$ and 
$\varrho^{-1}(\nu_j) = \{\eta_{j1}, \bar\eta_{j1},\dots, \eta_{jk_1} \bar\eta_{jk_1}\}.$ Since $\mu$ and $\mu'$ are strongly-pure weights that are the base-change from weight over 
$F_1$, the Kostant representative $w$ (and then so also $w'$) has the property that all the constituents $w^{\eta},$ as $\eta$ varies over $\varrho^{-1}(\nu_j),$ 
are copies of the same element of $\perm_N$--the Weyl group of $\GL_N$; in particular, since $w$ is balanced, $l(w^\eta) + l(w^{\bar\eta}) = nn'$ and $l(w^\eta) = l(w^{\bar\eta})$ since 
$\eta$ and $\bar\eta$ have the same restriction to $F_1$; hence $l(w^\eta) = nn'/2$. Recalling the notations from Sec.\,\ref{sec:galois-infty-bdry}, 
consider the wedge product
$$
e_{\Phi_{^\iota \! w, [\nu_j]}}^*
 \ := \ e_{\Phi_{^\iota \! w^{\eta_{j1}}}}^* \wedge e_{\Phi_{^\iota \! w^{\bar\eta_{j1}}}}^* \wedge  \cdots \wedge e_{\Phi_{^\iota \! w^{\eta_{jk_1}}}}^* \wedge e_{\Phi_{^\iota \! w^{\bar\eta_{jk_1}}}}^*. 
$$
All the individual factors such as $e_{\Phi_{^\iota \! w^{\eta_{ji}}}}^*$ or $e_{\Phi_{^\iota \! w^{\bar\eta_{ji}}}}^*$ are identical and have degree $nn'/2$. Hence the total degree of 
$e_{\Phi_{^\iota \! w, [\nu_j]}}^*$ is $nn'/2 \cdot k = nn'k_1.$ From \eqref{eqn:e-Phi} one gets: 
$$
e_{\Phi_{^\iota \! w}}^* \ = \ 
e_{\Phi_{^\iota \! w, [\nu_1]}}^* \wedge \cdots \wedge e_{\Phi_{^\iota \! w, [\nu_{d_1}]}}^*. 
$$
Denote the action of $\gamma$ on $\Sigma_{F_1},$ for the ordering fixed above, as $\pi_{F_1}(\gamma),$ and let $\varepsilon_{F_1}(\gamma)$ denote its signature. 
Then one has
$$
(1 \otimes \gamma) e_{\Phi_{^\iota \! w}}^* \ = \ 
\varepsilon_{F_1}(\gamma)^{(nn'k_1)^2} 
e_{\Phi_{^{\gamma \circ \iota} \! w}}^*; 
$$
from Def.\,\ref{def:signature}, one has $\varepsilon_{\iota, w}(\gamma) = \varepsilon_{F_1}(\gamma)^{(nn'k_1)^2} = 1$ since $nn'k_1 \equiv 0 \pmod{2}.$ 
Similarly, $\varepsilon_{\iota, w'}(\gamma) = 1.$
\end{proof}

\medskip
This concludes the proof of Thm.\,\ref{thm:main}. 
%\end{proof}

\medskip
\subsection{Compatibility with Deligne's Conjecture}
\label{sec:deligne}

\subsubsection{\bf Statement of Deligne's Conjecture}
In this subsection Deligne's celebrated conjecture on the special values of motivic $L$-functions is formulated for the ratios of successive 
successive critical $L$-values for Rankin--Selberg $L$-functions. 
The notations of \cite{deligne} will be freely used; a motive $M$ over $\Q$ with coefficients in a field $E$ will be thought in terms of its Betti, de Rham, and $\ell$-adic realizations. Attached to a critical motive $M$ are its periods $c^\pm(M) \in (E \otimes \C)^\times$ as in {\it loc.\,cit.,} that are well-defined in $(E \otimes \C)^\times/E^\times.$ We begin with a relation between the two periods over a totally imaginary base field $F$. Recall from the introduction that if $F$ is in the {\bf CM}-case, then $F_1$ is its maximal CM subfield which is totally imaginary quadratic 
over the totally real subfield $F_1$; suppose $F_1 = F_0(\sqrt{D})$ for a totally negative $D \in F_0$, then define 
$$
\Delta_{F_1} \ := \ \sqrt{N_{F_0/\Q}(D)}, \quad \Delta_F \ := \ \Delta_{F_1}^{[F:F_1]}.
$$ 
If $F$ is in the {\bf TR}-case, then $F_1 = F_0$ is the maximal totally real subfield, then define
$$
\Delta_{F_1} \ := \ 1, \quad \Delta_F \ := \ \Delta_{F_1}^{[F:F_1]} = 1. 
$$ 
A proof of the following proposition due to Deligne is given in \cite{deligne-raghuram}: 

\medskip
\begin{prop}
\label{prop:deligne}
Let $M_0$ be a pure motive of rank $n$ over a totally imaginary number field $F$ with coefficients in a number field $E$. Put $M = \Res_{F/\Q}(M_0),$ and suppose that $M$ has no middle Hodge type. Let $c^\pm(M)$ be the periods defined in \cite{deligne}. 
Then 
$$
\frac{c^+(M)}{c^-(M)} \ = \  (1 \otimes \Delta_F)^n, \quad \mbox{in $(E \otimes \C)^\times/ E^\times$.}
$$
\end{prop}
Under the identification $E \otimes \C = \prod_{\iota : E \to \C} \C$, the element $1 \otimes \Delta_F$ is $\pm 1$ in each component of $(E \otimes \C)^\times/ E^\times$, since its square is trivial.  
Based on Prop.\,\ref{prop:deligne}, Deligne's conjecture \cite{deligne} for the ratios of successive critical values of the completed $L$-function of $M$ may be stated as:  

\medskip
\begin{con}[Deligne]
\label{con:deligne}
Let $M_0$ be a pure motive of rank $n$ over a totally imaginary $F$ with coefficients in $E$. Put $M = \Res_{F/\Q}(M_0),$ and suppose that 
$M$ has no middle Hodge type. For $\iota : E \to \C$ let $L(s, \iota, M)$ denote the completed $L$-function attached to $(M, \iota).$ Put 
$L(s,M) =  \{L(s, \iota, M)\}_{\iota : E \to \C}$ for the array of $L$-functions taking values in $E \otimes \C.$ Suppose $m$ and $m+1$ are critical 
integers for $L(s,M)$, and assuming that $L(m+1,M) \neq 0$, we have
$$
\frac{L(m,M)}{L(m+1,M)} \ = \  (1\otimes \i^{d_F/2} \Delta_F)^n, \quad \mbox{in $(E \otimes \C)/ E^\times$.}
$$
\end{con}

A word of explanation is in order, since, in \cite{deligne}, Deligne formulated his conjecture for critical values of $L_f(s,M)$--the finite-part of the $L$-function attached to $M.$ 
From Conj.\,2.8 and (5.1.8) of \cite{deligne} for $M$ as above, one can deduce : 
$$
\frac{L_f(m, M)}{L_f(m+1, M)} \ = \ (1 \otimes (2 \bpi \i)^{-n \cdot d_F/2}) \, \frac{c^\pm(M)}{c^\mp(M)}, \quad \mbox{in $E \otimes \C$}.
$$
Knowing the $L$-factor at infinity one has: $L_\infty(m, M)/L_\infty(m+1, M) = 1 \otimes (2\bpi)^{d_F/2};$ hence for the completed $L$-function, one can deduce:
\begin{equation}
\label{eqn:deligne-con-ratios}
\frac{L(m, M)}{L(m+1, M)} \ = \ (1 \otimes \i^{n\cdot d_F/2}) \, \frac{c^\pm(M)}{c^\mp(M)}. 
\end{equation}
It is clear now that \eqref{eqn:deligne-con-ratios} and Prop.\,\ref{prop:deligne} give Conj.\,\ref{con:deligne}. 

\medskip
There is conjectural correspondence between $\sigma_f \in \Coh_{!!}(\Res_{F/\Q}(\GL_n/F), \mu/E)$ and a pure regular motive $M(\sigma_f)$ of rank $n$ over $F$ with coefficients
in $E$ (see \cite{clozel} or \cite[Chap.\,7]{harder-raghuram-book}). Given such a $\sigma_f$ and also $\sigma'_f \in \Coh_{!!}(\Res_{F/\Q}(\GL_{n'}/F), \mu'/E)$,  
Conj.\,\ref{con:deligne} applied to  $M = M(\sigma_f) \otimes M(\sigma_f'^\v)$ gives the following conjecture 
or the Rankin--Selberg $L$-functions $L(s, {}^\iota\sigma \times {}^\iota\sigma'^\v)$:

\begin{con}[Deligne's conjecture for Rankin--Selberg $L$-functions]
\label{con:deligne-rankin-selberg}
Let the notations and hypotheses be as in Thm.\,\ref{thm:main}. 
Then: 
$$
(\i^{d_F/2} \Delta_F)^{nn'} \, \frac{L(m, {}^\iota\sigma \times {}^\iota\sigma'^\v)}{L(m+1, {}^\iota\sigma \times {}^\iota\sigma'^\v)} \ \in \iota(E), 
$$
and, furthermore, for every $\gamma \in \Gal(\bar\Q/\Q)$ we have the reciprocity law: 
$$
\gamma \left(
(\i^{d_F/2} \Delta_F)^{nn'} \, \frac{L(m, {}^\iota\sigma \times {}^\iota\sigma'^\v)}{L(m+1, {}^\iota\sigma \times {}^\iota\sigma'^\v)} \right) \ = \ 
(\i^{d_F/2} \Delta_F)^{nn'} \, 
\frac{L(m, {}^{\gamma\circ\iota}\sigma \times {}^{\gamma\circ\iota}\sigma'^\v)}{L(m+1, {}^{\gamma\circ\iota}\sigma \times {}^{\gamma\circ\iota}\sigma'^\v)}.
$$
\end{con}

\medskip
\subsubsection{\bf Thm.\,\ref{thm:main} implies Conj.\,\ref{con:deligne-rankin-selberg}}
\label{sec:thm-implies-conj}
If $F$ is in the {\bf TR}-case, then existence of a critical point implies $nn'$ is even (Cor.\,\ref{cor:F-tot-real-nn'-even}), hence 
$(\i^{d_F/2} \Delta_F)^{nn'} = \i^{d_Fnn'/2} = \pm 1$ that may be absorbed into $\iota(E)$ and ignored from the Galois-equivarance, which is exactly 
the content of (3) of Thm.\,\ref{thm:main}. Assume henceforth that $F$ is in the {\bf CM}-case. The required 
compatibility follows from the equality of signatures in the following  

\begin{prop}
\label{prop:signatures-equal}
$$
\frac{ \gamma(|\delta_{F/\Q}|^{n n'/2})} {|\delta_{F/\Q}|^{n n'/2}} \cdot 
\varepsilon_{\iota, w}(\gamma) \cdot \varepsilon_{\iota, w'}(\gamma) \ = \ 
\frac{ \gamma((\i^{d_F/2} \Delta_F)^{nn'})}{(\i^{d_F/2} \Delta_F)^{nn'}} .
$$
\end{prop}

The proof uses the following lemma. 

\medskip
\begin{lemma}
\label{lem:compatibility_tot_imag}
Let $F$ be a totally imaginary field in the {\bf CM}-case and suppose $F_1$ the maximal CM subfield of $F$. 
Then, as elements of $\C^\times/\Q^\times$, we have: 
$$
|\delta_{F/\Q}|^{1/2} \ = \ 
\i^{d_F/2} \cdot \Delta_F \cdot \left(N_{F_1/\Q}(\delta_{F/F_1})\right)^{1/2}.
$$
\end{lemma}

\begin{proof}[Proof of Lem.\,\ref{lem:compatibility_tot_imag}]
Transitivity of discriminant for the tower of fields $F/F_0/\Q$ 
gives: 
$\delta_{F/\Q} = \delta_{F_0/\Q}^{[F:F_0]} \cdot N_{F_0/\Q}(\delta_{F/F_0}).$ Since the degree $[F:F_0] = 2[F:F_1]$ is even, one has  
$$
|\delta_{F/\Q}|^{1/2} \ = \ |N_{F_0/\Q}(\delta_{F/F_0})|^{1/2} \pmod{\Q^\times}.
$$
Next, one has 
$\delta_{F/F_0} = \delta_{F_1/F_0}^{[F:F_1]} \cdot N_{F_1/F_0}(\delta_{F/F_1}),$ by 
using transitivity of discriminant for the tower of fields $F/F_1/F_0$; using the $F_0$-basis $\{1, \sqrt{D}\}$ for 
$F_1$, one has $\delta_{F_1/F_0} = 4D;$ therefore, 
\begin{multline*}
N_{F_0/\Q}(\delta_{F/F_0}) \ = \ N_{F_0/\Q}(4D)^{[F:F_1]} \cdot N_{F_0/\Q}(N_{F_1/F_0}(\delta_{F/F_1})) \\ 
= \ N_{F_0/\Q}(D)^{[F:F_1]} \cdot N_{F_1/\Q}(\delta_{F/F_1}) \pmod{\Q^\times{}^2}.
\end{multline*}
Since $F_1/\Q$ is a CM-extension, $N_{F_1/\Q}(\delta_{F/F_1}) > 0$; hence 
$$
|\delta_{F/\Q}|^{1/2} \ = \  |N_{F_0/\Q}(\delta_{F/F_0})|^{1/2} \ = \ 
|N_{F_0/\Q}(D)|^{[F:F_1]/2} \cdot \left(N_{F_1/\Q}(\delta_{F/F_1})\right)^{1/2} \pmod{\Q^\times}. 
$$
Since $D \ll 0$ in $F_0$, we see that $(-1)^{[F_0:\Q]}N_{F_0/\Q}(D) > 0$. Hence, 
\begin{multline*}
|N_{F_0/\Q}(D)|^{[F:F_1]/2} \ = \ ((-1)^{[F_0:\Q]}N_{F_0/\Q}(D))^{[F:F_1]/2} \ = \ 
(\i^{[F_0:\Q]} \Delta_{F_1})^{[F:F_1]} \\ 
= \ \i^{[F_0:\Q][F:F_1]} \Delta_F \ = \ \i^{d_F/2} \Delta_F.
\end{multline*}
\end{proof}

After the above lemma, the proof of Prop.\,\ref{prop:signatures-equal} follows from the following 

\medskip
\begin{lemma}
\label{lem:equality-signatures}
$$
\varepsilon_{\iota, w}(\gamma) \cdot \varepsilon_{\iota, w'}(\gamma) \ = \  
\frac{\gamma \left(N_{F_1/\Q}(\delta_{F/F_1})^{nn'/2}\right)}{N_{F_1/\Q}(\delta_{F/F_1})^{nn'/2}}.
$$
\end{lemma}

\medskip
\begin{proof}
Suppose $F$ is a CM field, then $F = F_1$, hence the right hand side is $1$. We contend that in this case the left hand side is also $1$, or that 
$\varepsilon_{\iota, w}(\gamma) = \varepsilon_{\iota, w'}(\gamma).$ One may suppose that the ordering on 
$\Hom(F,E)$ is fixed such that conjugate embeddings are paired: 
$\Hom(F,E) \ = \ \{\tau_1, \bar\tau_1, \tau_2, \bar\tau_2, \dots, \tau_r, \bar\tau_r\}.$
After composing with $\iota : E \to \C$, one gets an enumeration: 
$\Hom(F,\C) \ = \ \{\eta_1, \bar\eta_1, \eta_2, \bar\eta_2, \dots, \eta_r, \bar\eta_r\}.$ For brevity, and hopefully, additional clarity, denote
$e_{\Phi_{w^{\eta, \bar\eta}}}^* := e_{\Phi_{w^{\eta}}}^* \wedge \ e_{\Phi_{w^{\bar\eta}}}^*$, and rewrite \eqref{eqn:e-Phi} as 
\begin{equation}
\label{eqn:e-Phi-rewrite}
e_{\Phi_{^\iota \! w}}^*
 \ := \ e_{\Phi_{w^{\eta_1, \bar\eta_1}}}^* \wedge \cdots \wedge e_{\Phi_{w^{\eta_r, \bar\eta_r}}}^* . 
\end{equation}
The action of $\gamma$ gives: 
\begin{equation}
\label{eqn:gamma-e-Phi-rewrite}
(1\otimes \gamma)(e_{\Phi_{^\iota \! w}}^*)
 \ := \ e_{\Phi_{w^{\gamma\circ\eta_1, \gamma\circ\bar\eta_1}}}^* \wedge \cdots \wedge e_{\Phi_{w^{\gamma\circ\eta_1, \gamma\circ\bar\eta_1}}}^* . 
\end{equation}
The right hand sides of \eqref{eqn:e-Phi-rewrite} and \eqref{eqn:gamma-e-Phi-rewrite} differ by the signature $\varepsilon_{\iota, w}(\gamma)$ that one seeks to identify. 
Each pair of conjugates embeddings $\{\eta_j, \bar\eta_j\}$ corresponds to a place $v_j$ of $F$; as before, $\eta_j$ is called the distinguished embedding--a base point in that pair 
of conjugate embeddings. The ordering on $\Hom(F,\C)$ fixes an ordering $\{v_1,\dots,v_r\}$ on $\place_\infty(F)$. 
Let $\pi_F(\gamma)$ denote the permutation of $\gamma$ on $\place_\infty(F)$, and $\varepsilon_F(\gamma)$ its signature. For each $1 \leq j \leq r$, let 
$l_j =  l(w^{\eta_j})$ and $l_j^* =  l(w^{\bar\eta_j})$; then $l_j + l_j^* = nn'$ since $w$ is balanced. The total degree of $e_{\Phi_{w^{\eta, \bar\eta}}}^*$ is $nn'$; 
interchanging two successive factors of \eqref{eqn:e-Phi-rewrite} introduces the signature $(-1)^{(nn')^2} = (-1)^{nn'}.$ 
Finally, let 
$J_\gamma := \{ j \, | \, \mbox{$\gamma \circ \eta_j$ is not distinguished.}\}.$ Then one has 
\begin{equation}
\label{eqn:varepsilon-w-gamma}
\varepsilon_{\iota, w}(\gamma) \ = \ \varepsilon_F(\gamma)^{nn'} \prod_{j \in J_\gamma} (-1)^{l_jl_j^*}; 
\end{equation}
since the term $\varepsilon_F(\gamma)^{nn'}$ arises by the permutation of the factors of \eqref{eqn:e-Phi-rewrite} to get the factors of \eqref{eqn:gamma-e-Phi-rewrite}; and then within 
each such factor 
indexed by $j \in J_\gamma$ the constituent factors in $e_{\Phi_{w^{\eta_j}}}^* \wedge \ e_{\Phi_{w^{\bar\eta_j}}}^*$ get interchanged. Similarly, 
\begin{equation}
\label{eqn:varepsilon-w'-gamma}
\varepsilon_{\iota, w'}(\gamma) \ = \ \varepsilon_F(\gamma)^{nn'} \prod_{j \in J_\gamma} (-1)^{l(w'^{\eta_j})l(w'^{\bar\eta_j})}.
\end{equation}
From Lem.\,\ref{lem:kostant-P-Q} it follows that $l(w'^{\eta_j}) = nn' - l(w^{\eta_j}) = l_j^*$ and $l(w'^{\bar\eta_j}) = nn' - l(w^{\bar\eta_j}) = l_j$; hence 
$(-1)^{l(w'^{\eta_j})l(w'^{\bar\eta_j})} = (-1)^{l_jl_j^*}$; whence, $\varepsilon_{\iota, w}(\gamma) = \varepsilon_{\iota, w'}(\gamma).$

\medskip
Now suppose $F$ is a totally imaginary field in the {\bf CM}-case and $F_1$ its maximal CM subfield. In preparation, 
fix orderings on $\Sigma_F$, $\Sigma_{F_1}$ and $\place_\infty(F_1)$ in a compatible way as follows: 
\begin{enumerate}
\item fix an ordering $\{w_1,\dots, w_{r_1}\}$ on $\place_\infty(F_1)$; 
\item then fix the ordering $\{\nu_1, \nu_2, \dots, \nu_{r_1}, \bar\nu_1,  \bar\nu_2, \dots,  \bar\nu_{r_1}\}$; where the pair of conjugate embeddings $\{\nu_j, \bar\nu_j\}$ map to $w_j$, and recall that we call $\nu_j$ as the distinguished embedding; 
\item finally, to fix an ordering on $\Sigma_F$, let $\Sigma_F(\nu)$ denote the fiber over $\nu \in \Sigma_{F_1}$ under the canonical restriction map $\Sigma_F \to \Sigma_{F_1}$; 
if $\nu < \nu'$ in $\Sigma_{F_1}$ then each element in $\Sigma_F(\nu)$ is less than every element of $\Sigma_F(\nu')$; and within each fiber $\Sigma_F(\nu)$ fix any ordering. 
\end{enumerate}
The Galois element $\gamma$ induces permutations on $\Sigma_F$, $\Sigma_{F_1}$ and $\place_\infty(F_1)$ giving the commutative diagram: 
$$
\xymatrix{
\Sigma_F \ar[rr]^{\pi_F(\gamma)} \ar[d] & & \Sigma_F \ar[d] \\ 
\Sigma_{F_1} \ar[rr]^{\pi_{F_1}(\gamma)} \ar[d] & & \Sigma_{F_1} \ar[d] \\
\place_\infty(F_1) \ar[rr]^{\pi_{1\infty}(\gamma)} & & \place_\infty(F_1) 
}
$$
Define $\hat{\pi}_{F}(\gamma)$ to be the permutation of $\Sigma_{F}$ that induces $\pi_{F_1}(\gamma)$ on $\Sigma_{F_1}$, and 
$$
\mbox{if \ $\pi_{F_1}(\gamma)(\eta) = \eta'$ \ then \ $\hat{\pi}_{F}(\gamma)$ is an order preserving bijection $\Sigma_F(\nu) \to \Sigma_F(\nu')$.}
$$
Define the permutation $\pi_F'(\gamma)$ of $\Sigma_{F}$ by 
\begin{equation}
\label{eqn:pi-F-gamma-prime}
 \pi_F(\gamma) \ = \ 
\pi_F'(\gamma) \circ  \hat{\pi}_{F}(\gamma). 
\end{equation}
Observe that $\pi_F'(\gamma)$ induces the identity permutation on $\Sigma_{F_1}$, and denote $\pi_{\Sigma_F(\nu)}'(\gamma)$
for the permutation that $\pi_F'(\gamma)$ induces on the fiber $\Sigma_F(\nu)$ above $\nu.$ Let $\varepsilon(*)$ denote the signature of a permutation $*$.
The proof of Lem.\,\ref{lem:equality-signatures} follows from the following two sub-lemmas.

\begin{sublemma}
\label{sublem-1}
$$ 
\varepsilon_{\iota, w}(\gamma) \cdot \varepsilon_{\iota, w'}(\gamma) \ = \ 
\varepsilon(\pi_F'(\gamma))^{nn'}.
$$
\end{sublemma}

\begin{sublemma}
\label{sublem-2}
$$ 
\frac{\gamma \left(N_{F_1/\Q}(\delta_{F/F_1})^{1/2}\right)}{N_{F_1/\Q}(\delta_{F/F_1})^{1/2}}
 \ = \ 
\varepsilon(\pi_F'(\gamma))
$$
\end{sublemma}

\begin{proof}[Proof of Sublemma \ref{sublem-1}]
Define $J_\gamma = \{j \ : \ \mbox{$\pi_{F}(\gamma)(\eta_j)$ is not distinguished}\}$. Keeping in mind that strongly-pure weights such as $\mu$ and $\mu'$ are the base-change 
of (strongly-)pure weights from $F_1$, we deduce that the constituents $({}^\iota\!w^\eta)_{\eta : F \to \C}$ 
of the Kostant representative ${}^\iota\!w$ are such that if $\eta|_{F_1} = \eta'|_{F_1},$ then ${}^\iota\!w^\eta$ and ${}^\iota\!w^{\eta'}$ are the same element in $\perm_N$--the Weyl group of 
$\GL_N.$ For $1 \leq j \leq r_1$ denote 
$l_j = l({}^\iota\!w^{\eta_{ji}})$ and $l_j^* = l({}^\iota\!w^{\bar\eta_{ji}}).$
One has $l_j + l_j^* = nn'$ since $w$ is balanced. Also, denote $l_\nu = l({}^\iota\!w^{\eta})$ for any $\eta \in \Sigma_F(\nu)$. We claim that 
\begin{equation}
\label{eqn:e-w-claim}
\varepsilon_{\iota, w}(\gamma) \ = \ 
\varepsilon(\pi_{1\infty}(\gamma))^{(nn'k)^2} \cdot
\prod_{j \in J_\gamma} (-1)^{l_jl_j^*k^2} \cdot 
\prod_{\nu \in \Sigma_{F_1}} \varepsilon(\pi'_{\Sigma_F(\nu)}(\gamma))^{l_\nu^2}. 
\end{equation}

Recall that the signature $\varepsilon_{\iota, w}(\gamma)$ is determined by the action of $\gamma$ on the wedge-product in \eqref{eqn:e-Phi}: 
$e_{\Phi_{^\iota \! w}}^* = e_{\Phi_{w^{\eta_1}}}^* \wedge \cdots \wedge e_{\Phi_{w^{\eta_{d_F}}}}^*.$ The proof of \eqref{eqn:e-w-claim} boils down to 
becoming aware how the factors in this wedge-product are permuted, and what signature is introduced in un-permuting them. The following scheme 
depicts from bottom to top, the places of $F_1$, embeddings of $F_1$, embeddings of $F$, and the lengths of the Kostant representatives they parametrize:  
$$
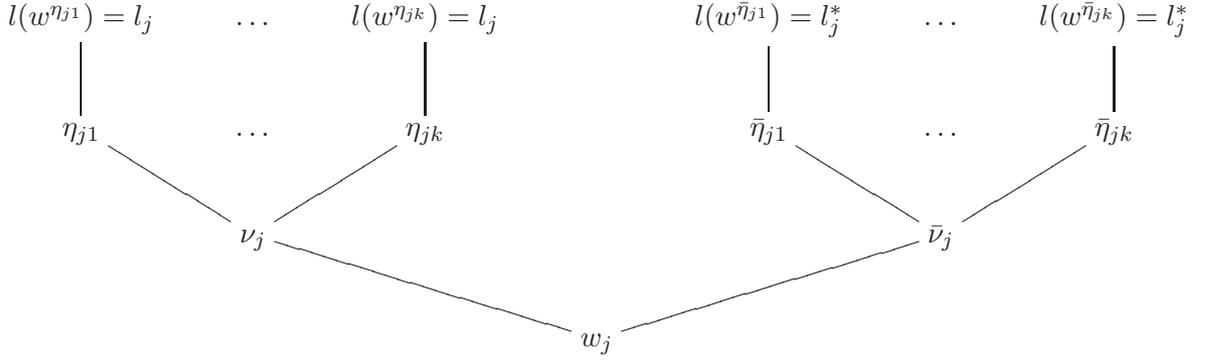
\xymatrix{
l(w^{\eta_{j1}}) = l_j \ar@{-}[d] & \dots & l(w^{\eta_{jk}}) = l_j \ar@{-}[d]& & l(w^{\bar\eta_{j1}}) = l_j^* \ar@{-}[d]& \dots & l(w^{\bar\eta_{jk}}) = l_j^*\ar@{-}[d] \\
\eta_{j1}\ar@{-}[dr]  & \dots & \eta_{jk} \ar@{-}[dl]& & \bar\eta_{j1} \ar@{-}[dr]& \dots & \bar\eta_{jk} \ar@{-}[dl]\\
& \nu_j \ar@{-}[drr] & & & & \bar\nu_j \ar@{-}[dll] & \\
&&& w_j &&& 
}
$$
Group together the wedge-factors as follows: 
$$
e_{\Phi_{^\iota \! w}}^* \ = \ 
e_{\Phi_{[w_1]}}^* \wedge \cdots \wedge e_{\Phi_{[w_{r_1}]}}^*, 
$$
where, for each $1 \leq j \leq r_1$, 
$$
e_{\Phi_{[w_j]}}^* \ = \ e_{\Phi_{[\nu_j]}}^* \wedge e_{\Phi_{[\bar\nu_j]}}^*
$$
and for each $\nu_j \in \Sigma_{F_1}$, 
$$
e_{\Phi_{[\nu_j]}}^* = e_{\Phi_{w^{\eta_{j1}}}}^* \wedge \dots \wedge e_{\Phi_{w^{\eta_{jk}}}}^*, \quad 
e_{\Phi_{[\bar\nu_j]}}^* = e_{\Phi_{w^{\bar\eta_{j1}}}}^* \wedge \dots \wedge e_{\Phi_{w^{\bar\eta_{jk}}}}^*. 
$$
Recall that $e_{\Phi_{w^{\eta_{ji}}}}^*$ has degree $l_j$ and $e_{\Phi_{w^{\bar\eta_{ji}}}}^*$ has degree $l_j^*.$ Hence, 
$e_{\Phi_{[\nu_j]}}^*$ has degree $kl_j$, and $e_{\Phi_{[\bar\nu_j]}}^*$ has degree $kl_j^*.$ Therefore, $e_{\Phi_{[w_j]}}^*$ has 
degree $kl_j + kl_j^* = knn'$. Now, the permutation $\pi_{1\infty}(\gamma)$ on $S_\infty(F_1) = \{w_1,\dots,w_{r_1}\}$ can be 
undone by the signature $\varepsilon(\pi_{1\infty}(\gamma))^{(nn'k)^2}.$ Next, only for those $j \in J_\gamma$, 
the two factors in $e_{\Phi_{[\nu_j]}}^* \wedge e_{\Phi_{[\bar\nu_j]}}^*$ get interchanged, giving the signature $(-1)^{l_jl_j^*k^2}$. 
Finally, adjusting for the action of $\gamma$ on $\Sigma_{F_1}$, i.e., now working with $\pi'_F(\gamma)$, which only permutes internally 
within each fiber $\Sigma_F(\nu)$ over $\nu \in \Sigma_{F_1}$, one sees the signature $\varepsilon(\pi'_{\Sigma_F(\nu)}(\gamma))^{l_\nu^2}$ 
for each such $\nu.$ This proves the claim \eqref{eqn:e-w-claim}.

\medskip
For any integer $a$, since $a^2 \equiv a \pmod{2}$, \eqref{eqn:e-w-claim} simplifies to 
\begin{equation}
\label{eqn:e-w}
\varepsilon_{\iota, w}(\gamma) \ = \ 
\varepsilon(\pi_{1\infty}(\gamma))^{nn'k} \cdot
\prod_{j \in J_\gamma} (-1)^{l_jl_j^*k} \cdot 
\prod_{\nu \in \Sigma_{F_1}} \varepsilon(\pi'_{\Sigma_F(\nu)}(\gamma))^{l_\nu}.
\end{equation}
Similarly, using the relation of $w'$ with $w$, one has: 
\begin{equation}
\label{eqn:e-w'}
\varepsilon_{\iota, w'}(\gamma) \ = \ 
\varepsilon(\pi_{1\infty}(\gamma))^{nn'k} \cdot
\prod_{j \in J_\gamma} (-1)^{l_j^*l_jk^2} \cdot 
\prod_{\nu \in \Sigma_{F_1}} \varepsilon(\pi'_{\Sigma_F(\nu)}(\gamma))^{l_\nu^*}.
\end{equation}
Multiply \eqref{eqn:e-w} and \eqref{eqn:e-w'} to get: 
$$
\varepsilon_{\iota, w}(\gamma)  \cdot 
\varepsilon_{\iota, w'}(\gamma)
 \ = \ \prod_{\nu \in \Sigma_{F_1}} \varepsilon(\pi'_{\Sigma_F(\nu)}(\gamma))^{l_\nu + l_\nu^*} 
\ = \ \left(\prod_{\nu \in \Sigma_{F_1}} \varepsilon(\pi'_{\Sigma_F(\nu)}(\gamma))\right)^{nn'}  
\ = \ \varepsilon(\pi_{F}'(\gamma))^{nn'}.
$$
\end{proof}

\begin{proof}[Proof of Sublemma \ref{sublem-2}]
For $x \in F_1^\times$ one has $N_{F_1/\Q}(x) = \prod_{\nu \in \Sigma_{F_1}}\nu(x) > 0.$ Let $\{\rho_1,\dots,\rho_k\}$ denote the set of 
all embeddings of $F$ into $\bar{F_1}$ over $F_1$, for some algebraic closure $\bar{F_1}$ of $F_1$; 
let $\{\omega_1,\dots,\omega_k\}$ is an $F_1$-basis for $F$; then $\delta_{F/F_1} = \det[\rho_i(\omega_j)]^2.$ Hence, 
$$
N_{F_1/\Q}(\delta_{F/F_1}) \ = \ \prod_{\nu \in \Sigma_{F_1}} \nu(\det[\rho_i(\omega_j)]^2) \ = \ 
\prod_{\nu \in \Sigma_{F_1}} \det[\rho_i^\nu(\omega_j)]^2,
$$
where $\{\rho_1^\nu, \dots, \rho_k^\nu\}$ is the set of all embeddings of $F$ into $\C$ that restrict to $\nu : F_1 \to \C.$ (We may take $\rho_i^\nu$ to be $\tilde\nu \circ \rho_i$ for any 
extension $\tilde\nu : \bar{F_1} \to \C$ of $\nu.$ Whence,  
\begin{equation}
\label{eqn:norm-matrix}
N_{F_1/\Q}(\delta_{F/F_1})^{1/2} \ = \ 
\pm \det
\left[\begin{array}{cccc}
[\rho_i^{\nu_1}(\omega_j)] & & & \\
 & [\rho_i^{\nu_2}(\omega_j)] & & \\
 & & \ddots & \\
 & & & [\rho_i^{\nu_{d_1}}(\omega_j)]
\end{array}\right],
\end{equation}
where the appropriate sign $\pm$ is chosen to make the right hand side positive. Each block $[\rho_i^{\nu}(\omega_j)]$ is a $k \times k$-block. 
Apply $\gamma$ to \eqref{eqn:norm-matrix} and the change in the sign of the determinant on the right is 
the requisite sign $\gamma \left(N_{F_1/\Q}(\delta_{F/F_1})^{1/2}\right)/N_{F_1/\Q}(\delta_{F/F_1})^{1/2}.$ The blocks are permuted according to $\pi_{F_1}(\gamma)$ which 
does not change the sign. Hence, the signature is accounted for by assuming that the blocks remain where they are and looking at how each block's rows are permuted internally; 
in other words, keeping \eqref{eqn:pi-F-gamma-prime} in mind, the requisite signature is 
$$
\prod_{\nu \in \Sigma_{F_1}} \varepsilon(\pi'_{\Sigma_F(\nu)}(\gamma)) \ = \ 
\varepsilon(\pi_{F}'(\gamma)).
$$
\end{proof}

This concludes the proof of Lem.\,\ref{lem:equality-signatures}. 
\end{proof}

This concludes the proof of Prop.\,\ref{prop:signatures-equal}, proving compatbility of our main theorem with Deligne's conjecture.

\medskip
\subsection{An example}
\label{sec:example}

If we take $n = n' = 1$ then the main result and techniques are all due to Harder \cite{harder-inventiones}. However, the signature 
$\varepsilon_{\iota, w}(\gamma) \cdot \varepsilon_{\iota, w'}(\gamma),$ that can be nontrivial in general, is missing in \cite{harder-inventiones}. Furthermore, the subtle distinction between when $F$
is in the {\bf CM}-case or in the {\bf TR}-case is not seen \cite{harder-inventiones} and it becomes apparent only in the larger context of this article.  
This case $n = n' = 1$ is also extensively 
discussed in \cite{raghuram-hecke} wherein examples are constructed to show the non-triviality of these signatures. As an alternative, it is worth the effort to illustrate the content of the main theorem in the 
simplest non-trivial example: when $n = n' = 1$ and $F$ is an imaginary quadratic field, not so much by appealing to Harder \cite{harder-inventiones}, or this article, 
but rather via recourse to modular forms of CM type. Here, $\sigma$ and 
$\sigma'$ are both algebraic Hecke characters and main theorem concerns the ratios of successive critical values of the $L$-function attached to 
the algebraic Hecke character: $\chi = \sigma \sigma'^{-1}.$ After relabelling, take $\sigma = \chi$ an algebraic Hecke character, 
and for $\sigma'$ take the trivial character. This $\GL(1)$-example is instructive, and was helpful to the author to see some finer details. 

\medskip

For an imaginary quadratic field $F,$ let $\Hom(F,\C) = \{\eta, \bar\eta\}$; the choice of $\eta$ is not canonical; it induces an isomorphism 
$\eta : F_\infty \simeq \C.$ Let $\chi : F^\times \backslash \A_F^\times \to \C^\times$ be an algebraic Hecke character; this means that $\chi$ is a continuous homomorphism whose infinite component $\chi_\infty : F_\infty^\times  \to \C^\times$ is of the form $\chi_\infty(z) = z^p \bar{z}^q,$ for integers $p$ and $q$. Then 
$\chi \in \Coh(\GL_1/F, \mu)$ with $\mu = (\mu^\eta, \mu^{\bar\eta})$ and $\mu^\eta = -p$ and $\mu^{\bar\eta} = -q.$ The weight 
$\mu$ is strongly-pure with purity weight $\w = -p-q.$ One also has  
$$
\chi_\infty(z) = \left(\frac{z}{\bar{z}}\right)^{\ell/2} (z\bar{z})^{-\w/2}, \quad \ell = p-q \in \Z.
$$ 
As recalled in \eqref{eqn:abelian-local-l-factor}, the $\Gamma$-factors at infinity (up to nonzero constants and exponentials) on either side of the functional equation are: 
$$
L_\infty(s, \chi) \sim \Gamma(s - \tfrac{\w}{2} + \tfrac{|\ell|}{2}), \quad L_\infty(1-s, \chi^{-1}) \sim \Gamma(1-s + \tfrac{\w}{2} + \tfrac{|\ell|}{2}).
$$ 
Assume, without any loss of generality (if necessary, replacing $\chi$ by $\chi^{-1}$), that $\ell \geq 0,$ i.e., $p \geq q.$ Then 
$L_\infty(s, \chi) \sim \Gamma(s+p)$ and $L_\infty(1-s, \chi^{-1}) \sim \Gamma(1-s-q).$ 
The critical set for $L(s, \chi)$ is the set of $\ell$ consecutive integers: 
$\{1-p, \, 2-p, \dots, \, -q\}.$
The critical set is nonempty if $\ell \geq 1$, and we have $\ell$ many critical points and $\ell-1$ pairs of successive critical points. The cuspidal width $\ell(\mu, 0)$ between $\mu$ and the weight $\mu' = 0$ is 
$\ell(\mu, 0) = \ell.$ If we were to apply the main theorem to the pair $\chi$ and the trivial Hecke character (which is cohomological with respect to $\mu'=0$) then the combinatorial lemma imposes  the condition $\ell \geq 2,$ and Thm.\,\ref{thm:main} gives a rationality result for the ratios $L(m,\chi)/L(m+1,\chi)$ of all successive critical values. This theorem can also be seen independently by appealing to the rationality results of Shimura for $L$-functions of modular forms. 

\medskip

Take $\pi = \pi(\chi) = \AI_F^\Q(\chi)$ to be the automorphic induction of $\chi$ from $F$ to $\Q$. Then $\pi$ is a cuspidal automorphic representation of 
$\GL_2(\A_\Q).$ The representation $\pi_\infty$ at the infinite place is, by definition, $\AI_\C^\R(\chi_\infty),$ which in turn is defined by asking for its Langlands parameter to be the induced representation 
$\Ind_{W_\C}^{W_\R}(\chi_\infty) = \Ind_{\C^\times}^{W_\R}(z \mapsto \left(\frac{z}{\bar{z}}\right)^{\ell/2}) \otimes |\ |_\R^{-\w/2}.$ This is exactly the representation that has cohomology with respect to the irreducible representation of $\GL(2)$ with highest weight $\lambda = (p,q).$ By the standard dictionary between modular forms and automorphic representations (see, for example, 
Gelbart \cite{gelbart}) there is a primitive modular form $f_\chi$ of weight $k = p-q+1$ such that $\pi(\chi) = \pi(f_\chi) \otimes |\ |^{-\w/2}.$ One of the properties of this dictionary gives us the following equality of $L$-functions: 
$$
L(s, f_\chi) \ = \ L(s - \tfrac{(k-1)}{2}, \pi(f_\chi)) \ = \ L(s - \tfrac{(k-1)}{2} + \tfrac{\w}{2}, \pi(\chi)) \ = \ L(s-p, \chi).
$$
The critical set for $L(s, f_\chi)$ is the string of integers $\{1,2,\dots,k-1\}.$ A word about the normalizations of these $L$-functions: first of all $L(s, f_\chi)$ is 
the Hecke $L$-function of the modular form $f_\chi$ which has a functional equation with respect to $s \leftrightarrow k-s$. For a cuspidal automorphic representation $\pi$, as applied to $\pi(\chi)$ or to $\pi(f_\chi)$, the functional equation is with respect to $s \leftrightarrow 1-s$. The $L$-function $L(s,\chi)$ 
also has a functional equation with respect to $s \leftrightarrow 1-s$. Furthermore, for any Dirichlet character $\omega$, by which we mean a character 
$\omega : \Q^\times \backslash \A_\Q^\times \to \C^\times$ of finite-order, there is the equality: 
$$
L(s, f_\chi, \omega) \ = \ L(s-p, \chi \otimes \omega^F),
$$
where $\omega^F := \omega \circ N_{F/\Q}$ is the base-change of $\omega$ from $\Q$ to $F$. In particular, if $\omega = \omega_{F/\Q}$ the 
quadratic Dirichlet character of $\Q$ attached to $F$ by class field theory then: 
$$
L(s, f_\chi, \omega_{F/\Q}) \ = \ L(s, f_\chi),
$$
since the base-change of $\omega_{F/\Q}$ back to $F$ is the trivial character. This is also seen at the level of representations since 
$\pi(\chi) \simeq \pi(\chi) \otimes \omega_{F/\Q}.$

\medskip

From Shimura \cite{shimura-mathann} applied to $f_\chi$, there exists two periods $u^\pm(f_\chi) \in \C^\times$, such that for any critical integer 
$r \in \{1,\dots,k-1\}$, and any primitive Dirichlet character $\psi$, one has
$$
L_f(r, f_\chi, \psi) \ \approx \ (2\bfgreek{pi} \i)^r u^\pm(f_\chi) \g(\psi), 
$$
where 
$\g(\psi)$ is the Gau\ss~sum of $\psi$, and the choice of periods is dictated by the parities of $r$ and $\psi$ via: $\psi(-1) = \pm (-1)^r$; and 
$\approx$ is a simplified notation to mean that the ratio of the left hand side divided by everything on the right hand side is algebraic, and is $\Gal(\bar{\Q}/\Q)$-equivariant: 
$$
\gamma\left(\frac{L_f(r, f_\chi, \psi)}{(2\bfgreek{pi} \i)^r u^\pm(f_\chi) \g(\psi)} \right) \ = \ 
\frac{L_f(r, {}^\gamma f_\chi, {}^\gamma\psi)}{(2\bfgreek{pi} \i)^r u^\pm({}^\gamma f_\chi) \g({}^\gamma\psi)}, \quad 
\forall \gamma \in \Gal(\bar{\Q}/\Q). 
$$
The finite part of the $L$-function $L_f(r, f_\chi, \psi)$ 
is completed using the archimedean $\Gamma$-factor
$L_\infty(s, f_\chi, \psi) = 2(2\bfgreek{pi})^{-s}\Gamma(s)$. In terms of the completed $L$-function the above relation takes the form:
$$
L(r, f_\chi, \psi) \ \approx \ \i^r u^\pm(f_\chi) \g(\psi). 
$$
Take $r=1$ and use the above relation once for $\psi$ the trivial character and then for $\psi = \omega_{F/\Q}$ to deduce: 
$$
u^+(f_\chi) \ \approx \ u^-(f_\chi) \, \g(\omega_{F/\Q}).
$$
Next, apply Shimura's result to $L(s,f_\chi)$ for $s = r$ and $s = r+1$, where $r \in \{1,\dots,k-2\}$ (possible when $k \geq 3$, i.e., $\ell \geq 2$), and 
divide one by the other to deduce: 
$$
\frac{L(r,f_\chi)}{L(r+1,f_\chi)} \ \approx \ \i \, \g(\omega_{F/\Q}), 
$$
while using $\i^2 \in \Q^\times$ and $\g(\omega_{F/\Q})^2 \in \Q^\times.$ Since $L(s, f_\chi) = L(s-p, \chi),$ and putting $r-p = m,$ one gets for the 
ratio of two successive critical values of the completed $L$-function of $\chi$ the rationality result: 
$\i \, \g(\omega_{F/\Q} L(m, \chi)/L(m+1,\chi) \approx ) \in \bar\Q$ and furthermore, 
$$
\gamma\left( \i \, \g(\omega_{F/\Q}) \frac{L(m, \chi)}{L(m+1, \chi)}\right) \ = \ 
\i \, \g(\omega_{F/\Q}) \frac{L(m, {}^\gamma\chi)}{L(m+1, {}^\gamma\chi)}, 
\quad \forall \gamma \in \Gal(\bar{\Q}/\Q). 
$$
One has used that ${}^\gamma f_\chi = f_{{}^\gamma\chi}$ which follows from the definition of $f_\chi$ (see \cite[Sect.\,5]{shimura-1976}). 
To see that the above result is indeed an instance of Thm.\,\ref{thm:main}, one needs the basic fact about quadratic Gauss sums: 
$\i \, \g(\omega_{F/\Q}) = |\delta_{F/\Q}|^{1/2} \pmod{\Q^\times}$. It is shown in \cite{raghuram-imrn} that this example generalizes from $\GL(1)$ 
over an imaginary quadratic extension to $\GL(n)$ over a CM field.

\bigskip

\bigskip

 \end{document}